\documentclass[12pt, reqno]{article}
\usepackage[latin1]{inputenc}
\usepackage{amsxtra}
\usepackage{amsmath}
\usepackage{amssymb}
\usepackage{amsfonts}
\usepackage[all]{xy}
\usepackage{mathrsfs}
\usepackage{amsthm}

\newtheorem{thm}{Theorem}[section]
\newtheorem{lemma}[thm]{Lemma}
\newtheorem{prop}[thm]{Proposition}
\newtheorem{cor}[thm]{Corollary}

\newtheorem{df}[thm]{Definition}

\newtheoremstyle{example}{\topsep}{\topsep}%
     {}
     {}
     {\bfseries}
     {.}
     {2pt}
     {\thmname{#1}\thmnumber{ #2}\thmnote{ #3}}

   \theoremstyle{example}
\newtheorem{rem}[thm]{Remark}
\newtheorem{ex}[thm]{Example}

\numberwithin{equation}{section}

\def\a{{\alpha}}

\def\eps{{\varepsilon}}
\def\l{{\lambda}}

\def\NN{\mathbb{N}}
\def\CC{\mathbb{C}}
\def\FF{\mathbb{F}}
\def\GG{\mathbb{G}}
\def\PP{\mathbb{P}}
\def\RR{\mathbb{R}}
\def\ZZ{\mathbb{Z}}
\def\QQ{\mathbb{Q}}

\def\aen{\mathfrak{a}}

\def\Aen{\mathfrak{A}}

\def\Ac{\mathcal{A}}
\def\Bc{\mathcal{B}}
\def\Cc{\mathcal{C}}

\def\Dc{\mathcal{D}}
\def\Ec{\mathcal{E}}
\def\Fc{\mathcal{F}}
\def\Gc{\mathcal{G}}
\def\Lc{\mathcal{L}}
\def\Mc{\mathcal{M}}
\def\Nc{\mathcal{N}}
\def\Hc{\mathcal{H}}
\def\Oc{\mathcal{O}}
\def\Pc{\mathcal{P}}
\def\Qc{\mathcal{Q}}

\def\Sc{\mathcal{S}}
\def\Tc{\mathcal{T}}

\def\on{\operatorname}

\def\et{{\rm{et}}}

\def\Hom{{\operatorname{Hom}}}
\def\un{{\operatorname{un}}}

\def\Weil{{\rm{Weil}}}

\def\tb{\mathbf{t}}

\def\homo{\operatorname{\it \mathscr{H}\kern-.25em om}}
\def\ext{\operatorname{\it \mathscr{E}\kern-.25em xt}}
\def\edo{\operatorname{\it \mathscr{E}\kern-.25em nd}}
\def\der{\operatorname{\it \mathscr{D}\kern-.25em er}}

\def\*{\circledast}
\def\Coh{{\mathcal{C}oh}}
\def\qcom{{\on{qcom}}}
\def\LS{{\mathcal{LS}}}
\def\coh{{\operatorname{coh}\nolimits}}
\def\Bun{{\mathcal{B}un}}
\def\Bunn{{\operatorname{Bun}\nolimits}}
\def\bun{{\operatorname{bun}\nolimits}}
\def\Vect{{\operatorname{Vect}\nolimits}}
\def\Tor{\mathcal{T}ors}

\def\tor{{\operatorname{tors}\nolimits}}
\def\cusp{{\operatorname{cusp}\nolimits}}

\def\Herm{{\operatorname{Herm}\nolimits}}
\def\Hom{\operatorname{Hom}\nolimits}
\def\LHom{\operatorname{LHom}\nolimits}
\def\GL{\operatorname{GL}\nolimits}

\def\End{\operatorname{End}\nolimits}

\def\diag{{\operatorname{diag}\nolimits}}
\def\ord{{\operatorname{ord}\nolimits}}

\def\supp{{\operatorname{supp}\nolimits}}
\def\op{{\operatorname{op}\nolimits}}
\def\Ann{{\operatorname{Ann}\nolimits}}
\def\Spec{\operatorname{Spec}\nolimits}

\def\Aut{\operatorname{Aut}\nolimits}
\def\Gal{\operatorname{Gal}\nolimits}
\def\Ext{\operatorname{Ext}\nolimits}
\def\Res{\operatorname{Res}\nolimits}

\def\Fr{\operatorname{Fr}\nolimits}

\def\deg{{\operatorname{deg}\nolimits}}
\def\length{{\operatorname{\ell}\nolimits}}
\def\Pic{{\operatorname{Pic}\nolimits}}
\def\rk{\operatorname{rk}\nolimits}
\def\rat{{\operatorname{rat}\nolimits}}

\def\Im{\operatorname{Im}\nolimits}

\def\Id{\operatorname{Id}\nolimits}

\def\tr{\operatorname{tr}\nolimits}

\def\Sym{\operatorname{Sym}\nolimits}

\def\Supp{\operatorname{Supp}\nolimits}
\def\Weil{\operatorname{Weil}\nolimits}
\def\wt{\operatorname{wt}\nolimits}
\def\Mat{\operatorname{Mat}\nolimits}
\def\Eis{\operatorname{Eis}\nolimits}

\def\Cr{\mathcal{C}r}
\def\Frm{\mathcal{F}rm}
\def\3x3{\operatorname{3}\times\operatorname{3}}
\def\QCoh{\mathcal{QC}oh}
\def\SES{\mathcal{SES}}
\def\1{{\bf 1}}
\def\lra{\longrightarrow}
\def\(({(\hskip -1mm (}
\def\)){)\hskip -1mm )}

\begin{document}

\author{Mikhail Kapranov 
 \and  Olivier Schiffmann
  \and Eric Vasserot
  }

\title{The Hall algebra of a curve}



\maketitle

\tableofcontents

 \addtocounter{section}{-1}

\section{Introduction}

Let $X$ be a smooth projective irreducible curve over a finite field $\FF_q$, and let
$\Bun(X)$ be the category of vector bundles on $X$. The {\em Hall algebra}
of $\Bun(X)$, denoted $H$, is an associative non-commutative algebra
whose elements are  finitely supported functions on the
set of isomorphism classes of objects of $\Bun(X)$. The functions take value in a field
$k$ of characteristic 0; in this introduction we assume $k=\CC$. 
The multiplication is
given by counting short exact sequences of bundles.  
In other words, elements of $H$  are
unramified automorphic forms for  all the groups
$GL_r$ over the function field $\FF_q(X)$, and multiplication is given by
the parabolic pseudo-Eisenstein series map. 

\vskip .2cm

This algebra is an object of remarkable depth which exhibits connections with
other areas of mathematics. It
was first considered in \cite{K} where functional equations for Eisenstein series 
were interpreted
as quadratic relations in $H$.  Already in the simplest case $X=\PP^1$
the algebra $H$ is identified with the ``pointwise Borel subalgebra" of the
quantum affine algebra $U_q(\widehat{\mathfrak{sl}}_2)$. 
Next, in the case when $X$ is an elliptic curve, one can identify 
a natural ``spherical" subalgebra in Cherednik's  double affine Hecke algebra
with a certain
subalgebra $H_{\text{sph}}\subset H$ (which is also called spherical)
 \cite{schiffmann-vasserot:macdonald}. There are  further deep relations of
  $H_{\text{sph}}$ in this
 case with Macdonald polynomials and the  Hilbert scheme of the plane
 \cite{schiffmann-vasserot:hilbert}.  Even more recently, an analog of $H_{\text{sph}}$ for
 higher genus curves was studied in \cite{schiffmann-vasserot:highergenus}.

 \vskip .2cm
 
 The goal of this paper
is to give a  description of the full algebra $H$ in terms of automorphic Rankin-Selberg
L-functions associated to unramified cusp eigenforms on (all the groups $GL_r$
over the function field of) $X$. We consider the set of all such forms as a
1-dimensional scheme $\Sigma$ with countably many components
 (similarly to the set of quasi-characters of
the idele group in the classical theory of Tate).  The value at 1 of the Rankin-Selberg
L-function defines then a rational function $\LHom: \Sigma\times\Sigma\to\CC$. 
We use this rational function to construct a {\em Feigin-Odesskii shuffle algebra},
similarly to
\cite{feigin-odesskii:shuffles}. Our main result, Theorem \ref{thm:main},
identifies $H$ with this ``Rankin-Selberg shuffle algebra". 
This extends
the result of \cite{schiffmann-vasserot:highergenus} for the  subalgebra 
in $H$ generated by line bundles. In particular, we embed $H$ into
$\CC_\qcom[\Sym(\Sigma)]$, the space of {\em regular} symmetric functions of
variables from $\Sigma$, i.e., of groups of variables from $\CC^*$,
see Corollary \ref{cor:regularity-of-embedding}. Here the subscript ``$\qcom$"
means that we take the direct sum of the space of regular functions on
individual irreducible components. 

\vskip .2cm

 Applying the Langlands correspondence for the groups $GL_r$ over functional fields
 established by Lafforgue \cite{lafforgue}, we then give a purely Galois-theoretic
 interpretation of the algebra $H$ in Corollary \ref{cor-of-lafforgue}.

\vskip .2cm 

From the point of view of relations, our result means that all the relations in $H$
are governed by the zeroes of the Rankin-Selberg functions (as it is the zeroes of
the defining functions which affect the relations in Feigin-Odesskii algebras). 
 Note that starting from the elliptic case,
  quadratic relations are no longer sufficient to describe $H$ (or even $H_{\text{sph}}$),
  see \cite {burban-schiffmann:elliptic-I, schiffmann:elliptic-II}.
 The situation is parallel to that of quantum affine algebras,
 see \cite{chari-pressley}, where one has to impose certain loop analogs of the
 Serre relations for semisimple Lie algebras. As an application, we give a simple
 proof of the fact that the algebra $H$ for an elliptic curve splits into an
 infinite tensor product of commuting algebras analogous to $H_{\text{sph}}$. This result was 
 also obtained by Fratila \cite{fratila}. 
  
  \vskip .2cm
  
  The importance of desribing the full algebra $H$ is that it comes with a natural
  basis (formed by individual vector bundles) and it is interesting to study the  
  symmetric polynomials on $\Sigma$ corresponding to these bundles.
  For intance, 
   the basis of bundles is obviously orthogonal with respect
  to the orbifold scalar product. The corresponding scalar product on 
  $\CC_\qcom[\Sym(\Sigma)]$ can be found by the  classical ``Maass-Selberg relations"
  (evaluation of the scalar product of two pseudo-Eisenstein series, cf. 
  \cite[\S II.2.1]{moeglin-waldspurger-book}). Algebraically, they give the $L_2$-scalar
  product on  $\CC_\qcom[\Sym(\Sigma)]$ corresponding to some rational function
  weight formed out of the $\LHom$ functions, in a way remindful
  of the scalar products considered by Macdonald \cite{macdonald2}.  So the symmetric polynomials
  associated to the individual bundles, form an orthogonal system,  thus
  presenting an exciting generalization of the
 setup of \cite{macdonald2}.

\vskip .2cm

The paper is organized as follows. In Section 1 we develop the formalism of shuffle algebras
in the generality we need (that  of a scheme with infinitely many components,
such as  the scheme $\Sigma$ of cusp eigenforms). Most earlier treatments worked with defining functions
$c(s,t)$ defined on a 1-dimensional algebraic group such as $\GG_m$ and depending
only on the ratio of variables: $c(s,t) = c(s/t)$. We also develop a formalism
of rational algebras and bialgebras of which  $H$  will
be later shown to give an example. 

Section 2 provides background on the Hall algebra
of  the category of all coherent sheaves on $X$ (not just vector bundles). In particular,
torsion sheaves form the classical unramified Hecke algebra $A$. It is convenient for
us to view the spectrum $W=\Spec(A)$ as a Witt scheme (or, rather, the product of
the classical Witt schemes of \cite{mumford}, one for each point of $X$). The interpretation
of points of a Witt scheme as  power series produces at once the L-factors and the
L-series corresponding to a cusp $A$-eigenform. In addition, the well known
{\em ring structure} on the Witt scheme corresponds to forming the Rankin-Selberg
tensor product L-functions. 

In Section 3 we introduce the scheme $\Sigma$ as a subscheme in $W$, and
the sub-semigroup in $W$ (with respect to the additive structure) generated by $\Sigma$  of $W$
plays  a key role in our construction. We take care to reformulate the basic results of the
automorphic theory (such as, e.g.,  the multiplicity one theorem)
 in a purely algebraic, rather than the more traditional analytic, way.
In particular, we use  the support of a module $M$ over the commutative ring $A$ instead of
the joint spectrum of operators in a Hilbert space. Our modules  $M$  consist of  finitely supported functions
on $\Bun$, so they don't contain actual Hecke eigenforms, 
but the corresponding eigenvalues show up in  $\Supp_A(M)$. 
We formulate our main results in \S 3.4 and
give an application to elliptic curves in \S 3.5.

Finally, Section 4 is devoted to the proof of Theorem \ref{thm:main}. We interpret
the classical results on functional equation and constant terms of Eisenstein
series as saying that $H$, although highly noncommutative,  can nevertheless be seen as a
{\em commutative and cocommutative rational bialgebra} in a certain meromorphic symmetric
monoidal category in the sense  of \cite{soibelman}.  We then use the constant
term (the analog of comultiplication) in order to embed each graded component of
$H$ into the space
of rational functions on a product of several copies of $\Sigma$.

 \vskip .2cm

 M.K.  would like to thank Universities Paris-7 and Paris-13 
  as well as the Max-Planck Institut f\"ur Mathematik in Bonn for hospitality and support during the
  work on this paper. His research was also partially supported by an NSF grant. 

\vfill\eject

\section{Generalities on shuffle algebras.}\label{sec:shuffle}

This section is a reminder on shuffle algebras. 
This material is mostly standard, so we will be sketchy.
An algebra will always be an associative algebra, 
with unit unless otherwise specified.
Let $k$ be a field of characteristic zero.

\subsection{Disjoint union schemes.}
\label{sec:disjoint}
Let $\Sigma$ be a scheme which is a disjoint union of irreducible affine algebraic
varieties over $k$.
Denote by $\pi_0(\Sigma)$ the set of connected components of $\Sigma$. The ring of regular functions
on $\Sigma$ is
$$k[\Sigma]\,\, =\,\, \prod_{S\in\pi_0(\Sigma)}k[S].$$
Note that if $\pi_0(\Sigma)$ is infinite, 
then $\Sigma$ is not an affine scheme, in particular
$\Sigma\neq\Spec(k[\Sigma])$. 
Let $\QCoh(\Sigma)$ be the category of quasicoherent sheaves of 
$\Oc_\Sigma$-modules. For $\Fc\in\QCoh(\Sigma)$ denote
$$\Gamma_\qcom(\Sigma,\Fc)\,\, =\,\, \bigoplus_{S\in\pi_0(\Sigma)}\Gamma(S,\Fc)$$
the space of sections with quasicompact support. For $\Fc=\Oc_\Sigma$ we get the $k$-algebra
$$k_\qcom[\Sigma]\,\, =\,\, \bigoplus_{S\in\pi_0(\Sigma)}k[S].$$
If $\pi_0(\Sigma)$ is infinite, this algebra has no unit, but for any finite number of elements
$a_1, a_2,\dots, a_n\in k_\qcom[\Sigma]$ there is an idempotent $e\in k_\qcom[\Sigma]$ such that
$ea_i=a_i$ for all $i$. The following is straightforward, compare with
\cite{bernstein}, \S 1.1.

\begin{prop} The functor $\Fc\mapsto \Gamma_\qcom(\Sigma,\Fc)$ identifies $\QCoh(\Sigma)$
with the category of $k_\qcom[\Sigma]$-modules $M$ such that
for any finite number of elements
$m_1, m_2,\dots, m_n\in M$ there is an idempotent $e\in k_\qcom[\Sigma]$ such that
$em_i=m_i$ for all $i$. 
\end{prop}

\begin{proof} Left to the reader.
\end{proof}

\vskip3mm

\subsection{Symmetric powers.}\label{section:symmetric-powers}
Since every component of $\Sigma$ is an affine scheme, we have a well-defined 
symmetric power scheme
$\Sym^n(\Sigma)=\Sigma^n/S_n$ for $n\geqslant 0$. It is clear that
$$k_\qcom[\Sym^n(\Sigma)]\,=\,(k_\qcom[\Sigma]^{\otimes n})^{S_n}$$
is the symmetric power of the $k$-vector space $k_\qcom[\Sigma]$, with the obvious $k$-algebra
structure. 
The component decomposition of $\Sym^n(\Sigma)$ is given by
$$\Sym^n(\Sigma)\,\,=\,\,\coprod_{\l}\,\prod_{S\in\pi_0(\Sigma)}\Sym^{\l(S)}(S),$$
where $\l$ runs over the set of maps $\pi_0(\Sigma)\to\NN$ such that
$\sum_S\l(S)=n$.  
We denote
$$\Sym(\Sigma)\,\,=\,\,\coprod_{n\geqslant 0}\Sym^n(\Sigma).$$
It is a commutative monoid in the category of schemes 
with respect to the operation
$$p:\Sym(\Sigma)\times\Sym(\Sigma)\to\Sym(\Sigma)$$
whose components are the usual symmetrization maps
$$p_{m,n}:\Sym^m(\Sigma)\times\Sym^n(\Sigma)\to\Sym^{m+n}(\Sigma).$$ 
Note that $p$ is an affine and finite morphism of schemes.
As $\Sym(\Sigma)$ is a monoid scheme, 
$\QCoh(\Sym(\Sigma))$ is a monoidal category under convolution.
Explicitly, an object  of $\QCoh(\Sym(\Sigma))$ is a sequence $\Fc=(\Fc_n)$ with
$\Fc_n\in\QCoh(\Sym^n(\Sigma))$,
 and for two such objects $\Fc$ and $\Gc$ their convolution is given by
$$\Fc\circledast\Gc=p_*(\Fc\boxtimes\Gc),\quad
(\Fc\circledast\Gc)_n=\bigoplus_{i+j=n}(p_{ij})_*(\Fc_i\boxtimes\Gc_j).$$
The unit object $\1$ of $\QCoh(\Sym(\Sigma))$ is equal to 
$k$ on $\Sym^0(\Sigma)=pt$ and to 0 elsewhere.
Note that $p$ being an affine morphism, we have
$$\Gamma_\qcom(\Sym(\Sigma),\Fc\circledast\Gc)\,\,=\,\,\Gamma_\qcom(\Sym(\Sigma),\Fc)\otimes_k
\Gamma_\qcom(\Sym(\Sigma),\Gc).$$
Next, let 
$$\sigma:\Sym(\Sigma)\times\Sym(\Sigma)\to\Sym(\Sigma)\times\Sym(\Sigma)$$
be the permutation. We have $p\sigma=p$. Thus 
$\QCoh(\Sym(\Sigma))$ is a symmetric monoidal category
with the symmetry 
$$R_{\Fc,\Gc}=p_*(P_{\Fc,\Gc}):\,\, \Fc\*\Gc=p_*(\Fc\boxtimes\Gc)\longrightarrow
(p\sigma)_*(\Gc\boxtimes\Fc)=\Gc\*\Fc,$$
where $P_{\Fc,\Gc}$ is the obvious isomorphism
$\Fc\boxtimes\Gc\to\sigma_*(\Gc\boxtimes\Fc)$.

\vskip3mm

\subsection{Rational sections.}\label{sec:rational-sections}
For any component  $S\subset\Sigma$ 
(which is an irreducible algebraic variety over $k$), 
we denote by $k(S)$ the field of rational functions on $S$ and by
$S_\rat=\Spec(k(S))$ the generic point of $S$. 
We view $S_\rat$ as an open subscheme of $S$ with the
embedding map $j_S:S_\rat\to S$.
For a quasicoherent sheaf $\Fc$ on $S$ let 
$\Fc_\rat=(j_S)_*j_S^*(\Fc)$ be the 
{\it sheaf of rational sections of $\Fc$ over $S$}. 
We extend this notation by defining $\Sigma_\rat$ and $\Fc_\rat$, for 
$\Fc\in\QCoh(\Sigma)$, in the obvious way. In particular we have the rings
$$k(\Sigma)=\Gamma(\Sigma,\Oc_\rat)=\prod_{S\in\pi_0(\Sigma)}k(S),\quad
k_\qcom(\Sigma)=\Gamma_\qcom(\Sigma,\Oc_\rat)=\bigoplus_{S\in\pi_0(\Sigma)}k(S).$$
For $\Fc,\Gc\in\QCoh(\Sigma)$, a morphism
$\Fc\to\Gc_\rat$ will be called a {\it rational morphism from $\Fc$ to $\Gc$}.

\vskip3mm

\subsection{Rational braidings.}
Let $c=c(s,t)$ be an invertible element of the ring $k(\Sigma\times\Sigma)$.
Thus $c$ is the datum, for each components $S,T\subset\Sigma$, 
of a non-zero rational function in
$k(S\times T)$. We extend $c$ to an invertible element in
$k(\Sym(\Sigma)\times\Sym(\Sigma))$ by putting
\begin{equation}\label{c-multiplicativity}
c\biggl(\sum_sn_s\cdot s,\sum_tm_t\cdot t\biggr)
\,\,=\,\, \prod_{s,t}c(s,t)^{n_sm_t}.
\end{equation}
For $\Fc,\Gc\in\QCoh(\Sym(\Sigma))$ we define a rational morphism
$$R_{\Fc,\Gc}^c:\Fc\*\Gc\to(\Gc\*\Fc)_\rat$$
to be the composition of $R_{\Fc,\Gc}:\Fc\*\Gc\to\Gc\*\Fc$
with the rational morphism $p_*(c_{\Fc,\Gc})$ where
$c_{\Fc,\Gc}:\Fc\boxtimes\Gc\to(\Fc\boxtimes\Gc)_\rat$
is the multiplication by $c$.

\begin{prop}
 The morphisms $R^c_{\Fc,\Gc}$ are natural in $\Fc,\Gc$ and they satisfy the rational analogs
of the braiding axioms, i.e., we have
$$\gathered
R_{\Ec\*\Fc,\Gc}^c=(R_{\Ec,\Gc}^c\*\Id_\Fc)\circ(\Id_\Ec\*R_{\Fc,\Gc}^c),\quad
R_{\Ec,\Fc\*\Gc}^c=(\Id_\Fc \* R_{\Ec,\Gc}^c)\circ(R_{\Ec,\Fc}^c \*\Id_\Gc),\cr
R^c_{\Fc,\1}=R^c_{\1,\Fc}=\Id_\Fc.
\endgathered$$
\end{prop}

\begin{proof} Left to the reader.\end{proof}

\vskip3mm

\subsection{Bialgebras in braided categories.}
\label{sec:braided-bialgebras}
Let $\Mc$ be a monoidal category with multiplication $\otimes$ and unit 
object $\1$.
An {\it algebra} in $\Mc$ is an object $A$ with morphisms
$\mu:A\otimes A\to A$, 
$e:{\1}\to A$ satisfying the usual associativity and unit axioms.
A {\it coalgebra} in $\Mc$ is an object $C$ with morphisms
$\Delta:C\to C\otimes C$, $\epsilon:C\to {\1}$ 
satisfying the usual coassociativity and counit axioms.
For $\Mc=\Vect$, the category of $k$-vector spaces,  an algebra (resp.~a coalgebra) in $\Mc$
is the same as a $k$-algebra 
(resp. ~$k$-coalgebra) in the usual sense.
Assume that $\Mc$ has a braiding 
$R=(R_{X,Y}:X\otimes Y\to Y\otimes X)$. Then the tensor product
$A\otimes B$ of two algebras is again an 
algebra with respect to the multiplication
$$\xymatrix{
A\otimes B\otimes A\otimes B\ar[r]^-{\Id\otimes R\otimes \Id}&A\otimes A\otimes B\otimes B
\ar[r]^-{\mu\otimes\mu}&A\otimes B}$$
and to the unit $e\otimes e$. Similarly, the tensor product of two coalgebras is again a coalgebra with comultiplication
$$\xymatrix{
C\otimes D\ar[r]^-{\Delta\otimes\Delta}&C\otimes C\otimes D\otimes D
\ar[r]^-{\Id\otimes R\otimes \Id}&C\otimes D\otimes C\otimes D}$$
and counit $\epsilon\otimes\epsilon$.
The following is well-known, see e.g.,\cite{majid-braided, takeuchi}.

\begin{prop}
\label{prop:braided-bialgebras}
Let $A$ be an algebra in $\Mc$ which 
is also a coalgebra in $\Mc$. Then $\mu$ is a morphism
of coalgebras if and only if $\Delta$ is a morphism of algebras. Indeed, both conditions are equivalent
to the commutativity of the diagram
$$\xymatrix{
A\otimes A\otimes A\otimes A\ar[r]^-{\Id\otimes R \otimes \Id}&
A\otimes A\otimes A\otimes A
\ar[r]^-{\mu\otimes\mu}&A\otimes A\cr
A\otimes A\ar[u]_-{\Delta\otimes\Delta}\ar[rr]^-\mu&&A\ar[u]^\Delta.}$$
\end{prop}

\vskip3mm

\subsection{Rational algebras and coalgebras.}
As $\QCoh(\Sym(\Sigma))$ is monoidal we can speak
of algebras and coalgebras in it.  We now modify the constructions in Section 
\ref{sec:braided-bialgebras} 
using the rational braiding $R^c$.

\begin{df}
A {\it rational algebra} in $\QCoh(\Sym(\Sigma))$
is  an object $\Ac$ 
with morphisms
$\mu : \Ac\*\Ac\to\Ac_\rat$ and $e:\1\to\Ac$
satisfying the usual axioms, but at the level of rational sections.
A {\it rational coalgebra} in
$\QCoh(\Sym(\Sigma))$ is an object $\Cc$ with morphisms
$\Delta:\Cc\to(\Cc\*\Cc)_\rat$ and
$\epsilon:\Cc\to \1$
satisfying the usual axioms at the level of rational sections. 
\end{df}

If $\Ac$ is an algebra in $\QCoh(\Sym(\Sigma))$ then 
$A=\Gamma_\qcom(\Sym(\Sigma),\Ac)$ is a $k$-algebra.
If $\Ac$ is a rational algebra in $\QCoh(\Sym(\Sigma))$ then 
$A_\rat=\Gamma_\qcom(\Sym(\Sigma),\Ac_\rat)$ is a $k$-algebra.
If $\Cc$ is a coalgebra in $\QCoh(\Sym(\Sigma))$ then $C=\Gamma_c(\Sym(\Sigma),\Cc)$
is a coalgebra. On the other hand, if $\Cc$ is a rational coalgebra in $\QCoh(\Sym(\Sigma))$,
 then 
$C_\rat=\Gamma_\qcom(\Sym(\Sigma),\Cc_\rat)$
is, in general, not a coalgebra because
\hfill\break $\Gamma_\qcom(\Sym(\Sigma),(\Cc\*\Cc)_\rat)$ is usually bigger than
$C_\rat\otimes_k C_\rat$.

\vskip .3cm

{\em Rational morphisms} of rational algebras or coalgebras are defined in the 
obvious way. Given two rational algebras $\Ac$, $\Bc$ in $\QCoh(\Sym(\Sigma))$, 
their product $\Ac\*\Bc$ is made into a rational algebra in the same way as in 
Section \ref{sec:braided-bialgebras}. 
Note that $\Ac\*\Bc$ will be a rational algebra even if $\Ac$, $\Bc$ 
are genuine algebras, because $R_{\Bc,\Ac}^c$ is a rational morphism. 
Dually, given two rational coalgebras $\Cc$, $\Dc$, their product $\Cc\*\Dc$ 
is made into a rational coalgebra in the same way as in Section
 \ref{sec:braided-bialgebras}. 

\begin{prop} Let $\Ac$ be a rational algebra in $\QCoh(\Sym(\Sigma))$ which is also a rational coalgebra.
Then $\mu$ is a rational morphism of rational coalgebras if and only if $\Delta$ is a rational morphism
of rational algebras.
\end{prop}

\begin{proof} Use a diagram similar to that in Proposition
\ref{prop:braided-bialgebras} but whose arrows are now
rational morphisms.\end{proof}

\begin{df} A {\it rational bialgebra} is a rational algebra and coalgebra
in $\QCoh(\Sym(\Sigma))$ such that $\mu$ is a rational morphism of rational coalgebras.
\end{df}

\vskip3mm

\subsection{The shuffle product on the tensor algebra.}
 Let
$\Pc\in\QCoh(\Sym(\Sigma))$ and $\Pc^{\*n}=\Pc\*\Pc\*\cdots\*\Pc$ ($n$ times). 
We denote 
$$\Tc(\Pc) \,\,=\,\,\ \bigoplus_{n\geqslant 0}\Pc^{\*n}$$
the {\em tensor algebra} of $\Pc$. 
We will sometimes write $\Tc^n(\Pc)$ for the direct summand $\Pc^{\*n}$.
We now define a multiplication (shuffle product) on $\Tc(\Pc)$
which is different from the obvious tensor multiplication. 

\vskip .2cm

For a permutation $\sigma\in S_n$ we
denote by $\ell(\sigma)$ its length. 
We  have a rational map $R^c_\sigma:\Pc^{\*n}\to(\Pc^{\*n})_\rat$. 
These maps are uniquely determined  by the following properties
\begin{itemize}
\item for $\sigma=(i,i+1)$ we have
$R^c_{(i,i+1)}=\Id_{\Pc^{\*(i-1)}}\*R^c_{\Pc,\Pc}\* \Id_{\Pc^{\*(n-i-1)}},$
\item for $\ell(\sigma\tau)=\ell(\sigma)+\ell(\tau)$ we have
$R^c_{\sigma\tau}=R^c_\sigma\circ R^c_\tau$.
\end{itemize}

Let $n=r+s$. A permutation
is called a {\it $(r,s)$-shuffle} if $\sigma(i)<\sigma(j)$ whenever $1\leqslant i<j\leqslant r$ or
$r+1\leqslant i<j\leqslant n$. 
Let $Sh_{r,s}\subset S_n$ be the set of $(r,s)$-shuffles.

For $n=r+s$ we define the {\it shuffle product}
$$\mu_{r,s}=\sum_{\sigma\in Sh_{r,s}}R_\sigma^c:\,\,\,\Pc^{\*r}\*\Pc^{\*s}\lra\Pc^{\*n}_\rat.$$

\begin{prop} The shuffle product $\mu=\sum_{r,s}\mu_{r,s}$
is associative, and so
makes $\Tc(\Pc)$ into a rational
algebra in $\QCoh(\Sym(\Sigma))$.
\end{prop}

\begin{proof} To see associativity, we
 need to compare the two rational morphisms
$$\mu_{r+s,t}\circ(\mu_{r,s}\*1), \,\,\,\mu_{r, s+t}\circ(1\*\mu_{s,t}):\,\,
\Pc^{\*r}\*\Pc^{\*s}\*\Pc^{\*t}\longrightarrow \Pc^{\*(r+s+t)},$$
for each $r,s,t\geqslant 0$. For this, call an $(r,s,t)$-{\em shuffle}
a permutation $\sigma\in S_{r+s+t}$ such that $\sigma(i)<\sigma(j)$ whenever
$$
1\leqslant i<j\leqslant r, \,\,\,{\text{or}}\,\,\, r+1\leqslant i<j\leqslant r+s,
\,\,\,{\text{or}}\,\,\, r+s+1\leqslant i<j\leqslant r+s+t.
$$
Using the braiding property of $R^c$, we see that both rational morphisms above
are equal to $\sum_\sigma R_\sigma^c$, where $\sigma$ runs over the $(r,s,t)$-shuffles.
\end{proof}

\noindent Next, for $n=r+s$ let
$\Delta_{r,s}:\Pc^{\*n}\to\Pc^{\*r}\*\Pc^{\*s}$ be the identity map.
Define a morphism 
$\Delta:\Tc(\Pc)\to \Tc(\Pc)\*\Tc(\Pc)$ by $\Delta=\sum_{r,s}\Delta_{r,s}$.
\begin{prop} The morphism $\Delta$ is coassociative. 
Together with the shuffle product it makes $\Tc(\Pc)$
into a rational bialgebra in $\QCoh(\Sym(\Sigma))$.
\end{prop}
\begin{proof} To see coassociativity, note that both morphisms
$$(\Delta\*1)\circ\Delta,\, (1\*\Delta)\circ\Delta:\,\, \Tc(\Pc)\longrightarrow
 \Tc(\Pc)\* \Tc(\Pc)* \Tc(\Pc)$$
have components
$\Delta_{r,s,t}: \Pc^{\*(r+s+t)}\to \Pc^{\*r}\*\Pc^{\*s}\*\Pc^{\*t}$,
which are the identity maps.  
To see compatibility of $\mu$ and $\Delta$, consider the $(r,s)$-graded
component of the pentagonal diagram in Proposition \ref{prop:braided-bialgebras}
for $\Ac=\Tc(\Pc)$. This component consists of the 2-arrow path
$$\Pc^{\*r}\*\Pc^{\*s} \stackrel{\mu_{r,s}}{\longrightarrow} \Pc_{\rat}^{\*(r+s)}
\stackrel{\sum \Delta_{n_1, n_2}}{\longrightarrow} \bigoplus_{n_1+n_2=r+s}
(\Pc^{\*n_1}\*\Pc^{\*n_2})_{\rat}$$
and the 3-arrow path  
$$
\Pc^{\*r}\*\Pc^{\*s} \stackrel{\sum \Delta_{r_1,r_2}\*\Delta_{s_1, s_2}}{\lra}
\bigoplus_{\substack{r_1+r_2=r\\ s_1+s_2=s}} \Pc^{\*r_1}\*\Pc^{\*r_2}\*\Pc^{\*s_1}\*\Pc^{\*s_2}
\lra
$$
$$\stackrel{\sum 1\*R_{\Pc^{\*r_2}, \Pc^{\*s_1}}\*1}{\lra} 
\bigoplus(\Pc^{\*r_1}\*\Pc^{\*s_1}\*\Pc^{\*r_2}\*\Pc^{\*s_2})_{\rat}\lra
$$
$$\stackrel{\sum \mu_{r_1,s_1}\*\mu_{r_2,s_2}}{\lra} 
\bigoplus_{n_1+n_2=r+s}
(\Pc^{\*n_1}\*\Pc^{\*n_2})_{\rat}.$$
Fix $n_1, n_2$ with $n_1+n_2=r+s$ and look at the two rational morphisms
$\Pc^{\*r}\*\Pc^{\*s}\to \Pc^{\*n_1}\*\Pc^{\*n_2}$ represented by the
composition of arrows of each path and then projection to the $(n_1, n_2)$-summand. 
Each of these two rational morphisms is a certain sum of operators of the form
$R^c_\sigma$, $\sigma\in S_n$, $n=r+s=n_1+n_2$. 
We claim that the summands in the two sums can be identified. 
This is an elementary combinatorial verification which we sketch briefly.

In the 2-arrow path, the arrow $\mu_{r,s}$ is a sum over the $(r,s)$-shuffles;
we view a shuffle $\sigma$ as a word formed by rearrangement of letters
$a_1, ..., a_r$, $b_1, ..., b_s$ such that the order of the $a$'s, as well
as the order of the $b$'s, is preserved. The following arrow, $\Delta_{n_1, n_2}$
gives only one summand which consists of partitioning our word $\sigma$
into two consecutive segments $\sigma_1$ and $\sigma_2$, of lengths
$n_1$ and $n_2$. Since $\sigma$ is a shuffle, $\sigma_1$ involves some
initial segment of the $a$'s, say $a_1, ..., a_{r_1}$ and some initial
segment of the $b$'s, say $b_1, ..., b_{s_1}$, so $r_1+s_1=n_1$.
 Moreover, $\sigma_1$
is an $(r_1, s_1)$-shuffle. Similarly, $\sigma_2$ is an $(r_2, s_2)$-shuffle
for $r_2+s_2=n_2$. We also have $r_1+r_2=r$ and $s_1+s_2=s$. The pair
of an $(r_1,s_1)$-shuffle $\sigma_1$ and an $(r_2, s_2)$-shuffle $\sigma_2$
gives a summand in the composition of the 3-arrow path on the diagram.
It remains to verify that this establishes a bijection
between the two sets of summands and that the corresponding summands
are equal in virtue of the braiding axioms.  We leave this to the reader. 
\end{proof}

\begin{rem} The above arguments are quite general. In particular, they are applicable
to any additive braided monoidal category $(\Mc, \otimes, \oplus, R)$
and to any object $\Pc$ on $\Mc$. In this case $\Tc(\Pc)=\bigoplus_{n\geqslant 0}
\Pc^{\otimes n}$ is a bialgebra in $\Mc$ with respect to $\mu$ and
$\Delta$ defined as above. The case when $\Mc$
is the category of modules over a triangular Hopf algebra, is well known
\cite{rosso}. From this point of view, our case corresponds to
the more general framework of {\em meromorphic braided categories} 
\cite{soibelman},
of which $\bigl(\QCoh(\Sym(\Sigma)), \*,\oplus,  R^c\bigr)$ is an example.

\end{rem}

\vskip3mm

\subsection{The shuffle algebra.} Let $\Pc\in\QCoh(\Sym(\Sigma))$. 
Set
$$\Sc h^n(\Pc)\,\,=\,\,
\Im\,\Bigl\{\sum_{\sigma\in S_n}R^c_\sigma:\,\,\Pc^{\*n}\lra\Pc^{\*n}_\rat\Bigr\}.$$
This is a quasicoherent subsheaf of  the quasicoherent sheaf
$(\Pc^{\*n})_\rat$ over $\Sym(\Sigma)$.
Note that $\sum_{\sigma\in S_n}R^c_\sigma$ is the $n$-fold multiplication 
$\mu^{(n-1)}:\Tc(\Pc)^{\*n}\to \Tc(\Pc)$ restricted to $\Pc^{\*n}=\Tc^1(\Pc)^{\*n}$. 
Write
$$\Sc h(\Pc)\,\,=\,\,\bigoplus_{n\geqslant 0}\Sc h^n(\Pc)
\,\,\subset\,\, \Tc(\Pc)_\rat.$$

\begin{prop} 
$(a)$ $\Sc h(\Pc)$ is closed under the multiplication in $\Tc(\Pc)_\rat$.
It is an algebra, not
just a rational algebra, in the category $\QCoh(\Sym(\Sigma))$.

$(b)$ The comultiplication $\Delta:\,\Tc(\Pc)_\rat\to(\Tc(\Pc)\*\Tc(\Pc))_\rat$ takes 
$\Sc h(\Pc)$ into
$(\Sc h(\Pc)\*\Sc h(\Pc))_\rat$, thus making 
$\Sc h(\Pc)$ into a rational bialgebra.
\end{prop}

\begin{proof} 
$(a)$ Since $\Tc(\Pc)$ is associative, the iterated multiplications satisfy 
$$\mu\bigl(\mu^{(r-1)}(\Pc^{\*r})\*\mu^{(s-1)}(\Pc^{\*s})\bigr)\,\,= \,\,
\mu^{(r+s-1)}(\Pc^{\*(r+s)}),$$
whence the statement.

$(b)$ The monoidal structure $\*$ being the composition of the tensor product
(over $k$)
and the pushdown under the affine morphism $p$, it is right exact in 
both arguments. Thus the image  of the
$\*$-product of two morphisms coincides with the 
$\*$-product of the images. In particular, $\Sc h^r(\Pc)\*\Sc h^s(\Pc)$ 
is the image of the map
$$\mu^{(r-1)}\*\mu^{(s-1)}:\,\,\Pc^{\*r}\*\Pc^{\*s}\lra(\Pc^{\*r}\*\Pc^{\*s})_\rat.$$
Let now $n=r+s$ and let us prove that
$$\Delta_{r,s}:\Pc^{\*n}_\rat\to (\Pc^{\*r}\*\Pc^{\*s})_\rat$$
(which is the identity morphism) takes
$\Sc h^n(\Pc)=\Im(\mu^{(n-1)})$ to
$\Im(\mu^{(r-1)}\*\mu^{(s-1)})_\rat$.
This follows from the commutativity of the diagram
$$\xymatrix{
(\Pc^{\*r}\*\Pc^{\*s})_\rat\ar[rr]^-{\mu^{(r-1)}\*\mu^{(s-1)}}
&&(\Pc^{\*r}\*\Pc^{\*s})_\rat\cr
\Pc^{\*n}\ar[rr]^-{\mu^{(n-1)}}
\ar[u]^-{\nabla_{r,s}}
&&\Pc^{\*n}_\rat
\ar[u]_-{\Delta_{r,s}},
}$$
where $\nabla_{r,s}=\sum_{\sigma\in\Sc h_{r,s}}R_\sigma$.
\end{proof}

\begin{df} The rational bialgebra $\Sc h(\Pc)=\Sc h_c(\Pc)$ in
 $\QCoh(\Sym(\Sigma))$
is called the {\it shuffle algebra} generated by $\Pc$.
 
\end{df}

\noindent Since the multiplication in $\Sc h(\Pc)$ is a regular map, 
the vector space 
$$Sh(\Pc)=\Gamma_\qcom(\Sym(\Sigma),\Sc h(\Pc))$$
is a $k$-algebra. This $k$-algebra is also called the shuffle algebra.

\begin{ex} 
\label{ex:classical-shuffle}
(a) Suppose that $c$ is identically equal to 1. Then $R^c=R$ is
the standard permutation. The braided monoidal category 
$\bigl(\QCoh(\Sym(\Sigma)), \*,  R^c\bigr)$ is symmetric,
and $\Sc h(\Pc) = \Sc(\Pc)$ is the symmetric
algebra of $\Pc$ in this category. The $k$-algebra
$Sh(\Pc)$ is  the usual symmetric algebra of the vector
space $\Gamma_\qcom(\Sigma, \Pc)$; in particular, it is commutative.  

\vskip .2cm

(b) Let $c$ be arbitrary. Since $\Sigma=\Sym^1(\Sigma),$ we can view
$\Oc_\Sigma$ as an object of $\QCoh(\Sym(\Sigma))$
whose sheaf of rational sections over $\Sym(\Sigma)$ is 
identified with $\Oc_{\Sigma_\rat}$.
For $\Pc=\Oc_\Sigma$ and $n\geqslant 0$ we have
$\Pc^{\*n}=\Oc_{\Sigma^n}$ viewed as a sheaf over
$\Sym^n(\Sigma)$ (extended by zero to the whole of $\Sym(\Sigma)$). Write
$Sh(\Sigma,c)=Sh(\Oc_\Sigma)$. We have
$$
\Gamma_\qcom(\Sym(\Sigma),\Pc^{\*n})=k_\qcom[\Sigma^n],\quad
\Gamma_\qcom(\Sym(\Sigma),\Pc^{\*n}_\rat)=k_\qcom(\Sigma^n),
$$
the spaces of regular (resp.~rational) functions $\phi(t_1,\dots,t_n)$ of $n$
 variables from $\Sigma$
which are $\neq 0$ only on finitely many components. 
For $n=r+s$ the shuffle product
$\mu_{r,s}:\Pc^{\*r}\*\Pc^{\*r}\to\Pc_\rat^{\*n}$
yields a map
$$k_\qcom(\Sigma^r)\times k_\qcom(\Sigma^s)\to k_\qcom(\Sigma^n)$$ 
which takes 
$(\phi,\psi)$ into the rational function
$$\mu_{r,s}(\phi,\psi)(t_1,\dots,t_n)=
\sum_{\sigma\in Sh_{r,s}}\Bigl[\prod_{i,j}c(t_i,t_j)\Bigr]\phi(t_{\sigma(1)},\dots,t_{\sigma(r)})
\psi(t_{\sigma(r+1)},\dots,t_{\sigma(n)}),$$
where the product runs over the pairs $(i,j)$ with $i<j$ and $\sigma(i)>\sigma(j)$.
The algebra $Sh(\Pc)$ is then
the subalgebra of
$\bigoplus_{n\geqslant 0}k_\qcom(\Sigma^n)$
generated by the space $k_\qcom[\Sigma]$.
 
\end{ex}

\begin{rem} More generally, the shuffle algebra can be defined for any object
$\Pc$ of any abelian meromorphic braided category $\Mc$ in the sense of 
\cite{soibelman}. 
An interesting example of such an $\Mc$ can be constructed
for any local non-archimedean field $F$. This is  the category of sequences
$(V_n)_{n\geqslant 0}$ where $V_n$ is an admissible representation of $GL_n(F)$.
The monoidal structure is given by parabolic induction, and the rational
braiding is given by the principal series intertwiners, cf. 
\cite{soibelman}. (See also \cite{joyal-street} for the case of finite fields.)
Let $C_0(F^\times)$ 
be the space of compactly supported locally constant functions
on $F^\times=\GG_m(F)$, and let $\Pc$ be the sequence consisting of 
$C_0(F^\times)$
in degree 1 and of 0 in other degrees. Then the component $\Sc h^n(\Pc)$
of the categorical shuffle algebra
$\Sc h(\Pc)$ is  the Schwartz space of the basic affine space
for $GL_n(F)$  as constructed by Braverman and Kazhdan 
\cite{braverman-kazhdan}.

\end{rem}

\vskip3mm

\subsection{The antisymmetric and coboundary cases.}
\label{sec:coboundary}
Assume 
that $c$ is antisymmetric, i.e.,
$c(s,t)c(t,s)=1$. Then the rational braiding $R^c$ is a (rational)
symmetry, i.e., $R^c_{\Gc, \Fc}R^c_{\Fc, \Gc}=\Id$ for any $\Fc, \Gc$.
The shuffle algebra $\Sc h(\Pc)$ is then a rational analogue of the
construction of the symmetric algebra of an object in a symmetric
monoidal category. In particular, we have the following fact
whose proof is straightforward and left to the reader.

\begin{prop}\label{prop:M-commutativity-shuffle}
 Let $\Ac=\Sc h(\Pc)$, and
denote $M=R^c_{\Ac,\Ac}: \Ac^{\*2}\to\Ac^{\*2}_\rat$. Then:

(a) The multiplication $\mu$ in $\Ac$ is $M$-commutative, 
 i.e., $\mu\circ M: \Ac\*\Ac
\to \Ac_\rat$ takes values in $\Ac$ and coincides with $\mu$.

(b) The rational comultiplication $\Delta$ in $\Ac$ is $M$-cocommutative, i.e., 
$\Delta=M\circ\Delta: \, \Ac\to \Ac^{\* 2}_\rat$. \qed
\end{prop}

\vskip .2cm

\begin{ex} Assume that $\Sigma$ is of the form $\Sigma = \bigsqcup_{i\in I} \Sigma_i$
where each $\Sigma_i\simeq\GG_m$ with coordinate $t_i$. A datum of an antisymmetric
$c$ as above is then a datum of non-zero rational functions $c_{ij}(t_i, t_j)$,
$i,j\in I$ with $c_{ij}(t_i,t_j)c_{ji}(t_j,t_i)=1$.  For $i\in I$ and $d\in \ZZ$
denote by $E_{i,d}$
the element $t_i^d\in k[\Sigma_i]$, considered as an element of
the algebra $Sh(\Sigma, c)$. Consider the formal generating series
\begin{equation}\label{functional-equation-shuffle1}
E_i(t) \,\,=\,\,\sum_{d\in\ZZ} E_{i,d} t^d\,\,\in\,\,Sh(\Sigma, c)[[t, t^{-1}]].
\end{equation}
Proposition \ref{prop:M-commutativity-shuffle} implies that the $E_i(t)$
satisfy the  quadratic commutation relations
\begin{equation}\label{functional-equation-shuffle}
E_i(t) E_j(s) \,\,=\,\,c_{ij}(t,s) E_j(s) E_i(t).
\end{equation}
They are to be understood in the following sense: write $c_{ij}(t,s)$
as a ratio of Laurent polynomials $P_{ij}(t,s)/Q_{ij}(t,s)$
and compare the coefficients at each power $t^as^b$ in
the equation
$$Q_{ij}(t,s) E_i(t) E_j(s) \,\,=\,\,P_{ij}(t,s) E_j(s) E_i(t)$$
which is a formal consequence of \eqref{functional-equation-shuffle}.

\end{ex}

\vskip .2cm

Note that an antisymmetric $c$ as above can be seen as a 1-cocycle of
the group $\ZZ/2$ with coefficients in the multiplicative
group $k(\Sigma\times\Sigma)^\times$, on which it acts by permutation. 
Representing this cocycle as a coboundary amounts to realizing $c$ as
multiplicative antisymmetrization of some invertible rational function 
$\l=\l(s,t)\in k(\Sigma\times\Sigma)^\times$, i.e., in the form
\begin{equation}\label{eq:coboundary}
c(s,t)=\l(s,t)^{-1}\,\l(t,s).
\end{equation}
For instance, if $\Sigma$ is irreducible and of positive dimension,
then such realization is always possible in virtue of Hilbert's
Theorem 90 applied to the field extension $k(\Sigma^2)/k(\Sym^2(\Sigma))$
with Galois group $\ZZ/2$. So we refer to the case 
\eqref{eq:coboundary} as the  {\em coboundary case}. 
As in \eqref{c-multiplicativity}, we extend $\lambda$ to a rational function
on $\Sym(\Sigma)\times\Sym(\Sigma)$ by multiplicativity. 
In the coboundary case there is an alternative realization of $\Sc h(\Pc)$,
motivated by the following.

\begin{ex}\label{shuffle-symmetric-regular}
Suppose that $\l(s,t)\in k[\Sigma\times\Sigma]^\times$ is an invertible
regular function. In this case  $\bigl(\QCoh(\Sym(\Sigma)), \*, R^c\bigr)$
is a genuine symmetric monoidal category. Moreover, it is  equivalent
(as a symmetric monoidal category) to $\bigl(\QCoh(\Sym(\Sigma)), \*, R\bigr)$, 
where $R=R^1$ is the standard permutation. Indeed, by definition, an equivalence 
should consist of a functor $\Phi$ plus natural isomorphisms
$$\phi_{\Fc, \Gc}: \, \Phi(\Fc\*\Gc)\lra \Phi(\Fc)\*\Phi(\Gc)$$
which take the braiding $R^c$ to $R$. 
We take $\Phi=\Id$ and $\phi_{\Fc, \Gc}$ to be the multiplication by $\l^{-1}$. 
 
This implies that for any $\Pc$ 
we have an isomorphism $\Psi: \Sc(\Pc)\to \Sc h(\Pc)$
where $\Sc(\Pc)$ is the usual
symmetric algebra of $\Pc$ defined using the symmetry $R$. It is defined,
on the level of global sections,  by
$$\Psi(a_1\cdot ... \cdot a_n)\,\,=\,\,
\sum_{\sigma\in S_n} \biggl[\prod_{i<j}\lambda(t_{\sigma(i)}, t_{\sigma(j)})^{-1}\biggr]
a_{\sigma(1)}\otimes ... \otimes a_{\sigma(n)}.$$
Note that $\Psi$ is an isomorphism of objects but not of algebras. Indeed,
$\Sc h(\Pc)$ is commutative with respect to the symmetry $R^c$ but does
not have to be commutative in the usual sense (symmetry $R$). 
So using the identification $\Psi$, we get a new 
product on $\Sc(\Pc)$,  referred to as the {\em symmetric shuffle product}. 
The construction below is obtained by extracting the formula for this
product from the structure of $\Psi$ and extending it to the case when
$\lambda$ is rational.

\end{ex}

\vskip .2cm

Assuming $\l$ rational, we define a rational morphism
$$\xi_{m,n}: \Sc^m(\Pc)\*\Sc^n(\Pc)\lra \Sc^{m+n}(\Pc)_{\rat}, \quad m,n\geqslant 0$$
where $\Sc^m(\Pc)$ is the usual symmetric power of $\Pc$ (defined using the
symmetry $R=R^1$). At the level of global sections it is given by
\begin{equation}\label{eq:shuffle-product-symmetric}
\xi_{m,n}(a\otimes b) \,\,=\,\,\frac{1}{m! n!}\operatorname{Symm}\biggl[ a\otimes b
 \prod_{\substack{1\leqslant i\leqslant m\\
1\leqslant j\leqslant n}} \lambda(s_i, t_j)\biggr].
\end{equation}
Here $\operatorname{Symm}$ means symmetrization over the symmetric group $S_{m+n}$
and $\l^{-1}(s_i, t_j)$ is regarded as a rational function on 
$\Sym(\Sigma)^m\times\Sym(\Sigma)^n$ depending on the $i$th coordinate 
of the first factor and the $j$th coordinate of the second factor.

\begin{prop}\label{prop: symmetric-shuffle-algebra}
 (a) The $\xi_{m,n}$ are associative and make $\Sc(\Pc)$
into a rational associative algebra in 
$\bigl(\QCoh(\Sym(\Sigma)), \*)$.

(b) The shuffle algebra $\Sc h(\Pc)$ (defined using the braiding $R^c$)
is isomorphic with the subalgebra in $\Sc(\Pc)_{\rat}$ generated by 
$\Pc\subset\Sc^1(\Pc)_{\rat}$. 
\end{prop}

\begin{proof} Use the rational analog of the isomorphism $\Psi$
from Example \ref{shuffle-symmetric-regular} and verify that it takes
the multiplication $\xi_{m,n}$ to the shuffle product $\mu_{m,n}$. 
\end{proof}

\begin{ex}\label{ex: symmetric-shuffle}
$\,$ As in  Example \ref{ex:classical-shuffle}(b), let us take
$\Pc=\Oc_\Sigma$ considered as an object of $\QCoh(\Sym(\Sigma))$.
Then  \eqref{eq:shuffle-product-symmetric} defines a multiplication on
$\operatorname{SR}:= \bigoplus_{n\geqslant 0} k_\qcom(\Sigma^n)^{S_n}$, 
the space of symmetric
rational functions of arguments in $\Sigma$.
We denote by $SSh(\Sigma, \l)$ and call the
{\em symmetric shuffle algebra} associated to $\Sigma$ and $\l$
the subalgebra in $\operatorname{SR}$ generated by $k_\qcom[\Sigma]$.
Proposition \ref{prop: symmetric-shuffle-algebra} implies that
$SSh(\Sigma, \lambda)$ is isomorphic to 
$Sh(\Sigma, c)$.  
\end{ex}

\begin{prop}\label{prop:symm-shuffle-regular}
 In the situation of Example \ref{ex: symmetric-shuffle}, assume that
each component of $\Sigma$ is a smooth curve over $k$. Assume also that $\lambda$
is regular everywhere except for, possibly, first order poles on the (components of the)
diagonal $\Delta\subset\Sigma\times\Sigma$. Then 
\[
SSh(\Sigma,\lambda) \,\,\subset \,\,\bigoplus_n k_\qcom[\Sigma^n]^{S_n}
\]
is contained in the space of regular symmetric functions of variable in $\Sigma$
(with quasicompact support). Moreover, for each $n$ the image of the degree $n$ component
of $SSh(\Sigma, \lambda)$ is an ideal in the ring $k_\qcom[\Sigma^n]^{S_n}$. 
\end{prop}

\begin{proof}
Consider the symmetrization morphism
\[
p_n: \Sigma^n\lra\Sym^n(\Sigma).
\]
The direct image (trace) of a rational function $F$ under $p_n$ is simply the average
$p_{n*}(F)=\sum_{\sigma\in S_n} \sigma(F)$. 
As $\Sigma$ is a disjoint union of curves, the relative dualizing sheaf of $p$
is the sheaf of functions with at most first order poles along the diagonals $\{t_i=t_j\}$. 
This means that if $F$ is such a function, then $p_{n*}F$, the average of $F$ under the symmetric
group, is a regular function on $\Sym^n(\Sigma)$. 

Now, let $f_1, ..., f_n\in k_\qcom[\Sigma]$. Their product in $SSh(\Sigma, \lambda)$
is 
\[
p_{n*} \biggl( \biggl[ \prod_{i<j}\lambda(t_i, t_j)\biggr] f_1(t_1) \cdot f_n(t_n)\biggr),
\]
i.e., $p_{n*}$ of a function belonging to the dualizing sheaf, so, by the above,
  the product is a regular function. To see that the image of the product map
\[
k_\qcom[\Sigma]^{\otimes n}\to k_\qcom[\Sigma^n]^{S_n}
\]
 is an ideal, we notice that this map is linear over the ring of the symmetric functions acting
on the source and target. 
\end{proof}

\vskip3mm

\subsection{Generalized shuffle algebras.}
\label{sec:generalizedshuffle}
Let $\frak S$ be a $\Bbb N$-graded semigroup scheme. This means that 
$\frak S=\coprod_{n\geqslant 0}\frak S_n$
is a scheme which is a disjoint union of  algebraic varieties of finite type over $k$ with
$\frak S_0=\Spec(k)$, and that there is an addition morphism
$$p:\frak S\times\frak S\to\frak S$$
which is graded and such that the obvious inclusion
$$\Spec(k)=\frak S_0\to\frak S$$ is a unit. We'll assume that $p$ is \emph{locally a finite morphism}.
This means that for each connected component $S$ of $\frak S$ the restriction of $p$ to
$p^{-1}(S)$ is finite. Now, fix an invertible element 
$c$ in $k_{\text{qcom}}(\frak S\times\frak S)$ which is a rational
bihomomorphism $\frak S\times\frak S\to\Bbb A^1_k$. Then $k_{\text{qcomp}}(\frak S)$ is an algebra for the 
shuffle product 
$$\mu_{r,s}:k_{\text{qcomp}}(\frak S^r)\times k_{\text{qcomp}}(\frak S^r)\to k_{\text{qcomp}}(\frak S^{r+s})$$
which takes $(\phi,\psi)$ into the rational function 
$$\mu_{r,s}(\phi,\psi)=p_*\big((\phi\boxtimes\psi)\cdot c\big).$$
The associativity follows from the fact that $c$ is a rational bihomomorphism.
The generalized shuffle algebra $Sh(\frak S,c)$ associated with $\frak S$ and $c$ is then the subalgebra of
$k_\qcom(\frak S)$
generated by the space $k_\qcom[\frak S^1]$.

\vfill\eject

\section{The Hall algebra of a curve.}

\subsection{Orbifolds} By an {\it orbifold} we mean a groupoid $\Mc$ 
such that, first, $\Mc$ is {\it essentially small}, i.e., the isomorphism
 classes of $\Mc$ form a set
$\underline{\Mc}$ , and, second, for any object
$x\in \Mc$ the group $\Aut(x)$ is finite. For an orbifold $\Mc$ let $\Fc(\Mc)$ be the space of isomorphism
invariant functions $f:\operatorname{Ob}(\Mc)\to k$
(which we can view as functions on the set $\underline{\Mc}$);   let 
$\Fc_0(\Mc)\subset\Fc(\Mc)$ be the subspace of functions with 
finite support. For an object $x\in\Mc$ we denote by $1_x\in\Fc_0(\Mc)$ the
characteristic function of (the isomorphism class of) $x$. 

\vskip .2cm

The space $\Fc(\Mc)$ is identified with the algebraic dual of $\Fc_0(\Mc)$ 
via the orbifold scalar product
\begin{equation}\label{scalarproduct}
(f,g)=\sum_{x\in\underline{\Mc}}\frac{f(x)g(x)}{|\Aut(x)|}.
\end{equation}
If $k=\CC$ we have the positive definite Hermitian scalar product
\begin{equation}
\label{hermitianproduct}
(f,g)_{\Herm}=\sum_{x\in\underline{\Mc}}\frac{f(x)\overline{g(x)}}{|\Aut(x)|}.
\end{equation}
A functor of orbifolds $\phi:\Mc\to\Nc$ defines the {\em inverse image map}
$$\phi^*:\Fc(\Nc)\to\Fc(\Mc),\quad (\phi^*g)(x)=g(\phi(x))$$
and an \textit{orbifold direct image map}
$$\phi_*:\Fc_0(\Mc)\to\Fc_0(\Nc),\quad 
\phi_*(1_x)=\frac{|\Aut(\phi(x))|}{|\Aut(x)|}1_{\phi(x)}.$$

We say that $\phi$ is
{\it proper}, if for any $x\in\Nc$ the preimage $\phi^{-1}(x)$
consists of finitely many isomorphism classes. 
In this case
$\phi^*$ takes $\Fc_0(\Nc)$ into $\Fc_0(\Mc)$, and $\phi_*$ extends to a map $\Fc(\Mc)\to\Fc(\Nc)$. Moreover, the maps $\phi_*$, $\phi^*$ are adjoint with respect to (\ref{scalarproduct}). Given functors of orbifolds
$$\xymatrix{\Nc'\ar[r]^v&\Nc&\Mc\ar[l]_\varphi},$$
the {\it fiber product orbifold}
$\Nc'\times_\Nc\Mc$ is the category such that
\begin{itemize}
\item an object is a triple
$(n',m,\alpha)$ with $n'\in\Nc'$, $m\in\Mc$ and $\alpha$ is an isomorphism
$v(n')\to\varphi(m)$ in $\Mc$, 
\item a morphism
$(n'_1,m_1,\alpha_1)\to(n'_2,m_2,\alpha_2)$ is a pair of morphisms
$\beta:n'_1\to n'_2$, $\gamma:m_1\to m_2$ such that the 
following square in $\Nc$ commutes
$$\xymatrix{
v(n'_1)\ar[r]^{\alpha_1}\ar[d]_{v(\beta)}&\phi(m_1)\ar[d]^{\phi(\gamma)}\cr
v(n'_2)\ar[r]^{\alpha_2}&\phi(m_2).
}$$
\end{itemize}
A square of orbifolds and their functors
\begin{equation}\label{square}
\begin{split}
\xymatrix{\Mc'\ar[r]^u\ar[d]_{\phi'}&\Mc\ar[d]^\phi\cr
\Nc'\ar[r]^v&\Nc}
\end{split}
\end{equation}
is called {\it homotopy commutative} if there is a natural transformation
$T:\phi u\Rightarrow v\phi'$ of functors $\Mc'\to\Nc$. In this case $T$ induces a functor
$\epsilon_T:\Mc'\to\Nc'\times_\Nc\Mc$. We say that the square is {\it homotopy cartesian}
if $\epsilon_T$ is an equivalence.

\begin{prop} If in a homotopy cartesian square (\ref{square}) the arrows are proper then we have the base change property
$u_*(\phi')^*=\phi^*v_*:\Fc_0(\Nc')\to\Fc_0(\Mc)$.\qed
\end{prop}

\vskip3mm


\subsection{Hall algebras}
An abelian category $\Ac$ is called 
{\it finitary} if it is essentially small and if for any objects $A,B$
the group $\Ext^i_\Ac(A,B)$ is finite and equal to 0 for almost all $i$. For a finitary abelian category
$\Ac$ we have an orbifold $\Mc(\Ac)$ with the same objects as $\Ac$ and
morphisms being the isomorphisms in $\Ac$. We set
\begin{equation}
\label{euler1}
\langle A,B\rangle\,\,=\,\,\sqrt{\prod_{i\geqslant 0}|\Ext^i_\Ac(A,B)|^{(-1)^i}}, \quad
\((A,B\)) \,=\, \langle A,B\rangle\cdot \langle B,A\rangle.
\end{equation}
Let $[\Ac]$ be the Grothendieck group of $\Ac$. For $A\in\Ac$ let $[A]\in [\Ac]$ be its class.
We have the $\ZZ$-bilinear forms
\begin{equation}
\label{euler2}
\gathered
\langle\bullet,\bullet\rangle, \,\,\,
\((\bullet,\bullet\)): [\Ac]\times[\Ac]\to\RR^\times,\\
\langle [A],[B]\rangle=\langle A,B\rangle,\quad
\((\alpha,\beta\)) = \langle \alpha,\beta\rangle\cdot
\langle\beta, \alpha\rangle,
\endgathered
\end{equation}
called, respectively, the {\em Euler} and the {\em Cartan} form. 
Assume that $k$ contains all the square roots 
appearing in the values of the Euler form.
 Set $H(\Ac)=\Fc_0(\Mc(\Ac))$. It is a $k$-algebra, 
called the {\it Hall algebra of $\Ac$}. The product is given by
\begin{equation}
\label{product}
(f*g)(C)\,\,=\,\,\sum_{A\subset C}\,\langle C/A,A\rangle\cdot f(A)\cdot g(C/A),
\end{equation}
where the sum is over all subobjects $A'$ in $C$.
Alternatively, on the basis vectors the multiplication has the form
\begin{equation}\label{eq:g-ABC}
1_A * 1_B \,\,=\,\, \langle B,A\rangle \sum_C g_{AB}^C \cdot 1_C,
\end{equation}
where $g_{AB}^C$ is the number of subobjects $A\subset C$ such that
$A'\simeq A$ and $C/A'\simeq B$. 
Note that $g_{AB}^C$ is finite and that, for fixed $A$ and $B$, $g^C_{AB}=0$ 
for all but finitely many $C$, up to isomorphism.
Let $\SES(\Ac)$ be the orbifold whose objects are all short exact sequences in $\Ac$
\begin{equation}
\label{eq:SES}
0\to A\to C\to B\to 0,
\end{equation}
and morphisms are isomorphisms of such sequences. We have three functors
$$p_1,p,p_2:\SES(\Ac)\to\Ac$$
associating to a sequence its three terms. We get a diagram of functors, with
$(p_1,p_2)$ being proper
\begin{equation}\label{eq:SES-Hall}
\xymatrix{\Mc(\Ac)\times\Mc(\Ac)&\SES(\Ac)\ar[l]_-{(p_1,p_2)}\ar[r]^p&\Mc(\Ac).}
\end{equation}
Note that $H(\Ac)\otimes H(\Ac)=\Fc_0(\Mc(\Ac)\times\Mc(\Ac))$.
The following is straightforward. 
\begin{prop}
\label{prop:SES-orbifold-Hall}
The multiplication $\mu$ in $H(\Ac)$ is identified with the map
$$\Fc_0(\Mc(\Ac)\times\Mc(\Ac))\to\Fc_0(\Mc(\Ac)),\quad 
f\mapsto p_*\bigl((p_1,p_2)^*(f\cdot\langle B,A\rangle)\bigr),$$
where $\langle B,A\rangle$ is the function $\Mc(\Ac)\times\Mc(\Ac)\to k,$ 
$(A,B)\mapsto\langle B,A\rangle$. \qed
\end{prop}

The $k$-algebra $H(\Ac)$ is $[\Ac]$-graded 
\begin{equation}\label{eq:grading-Hall}
H(\Ac)=\bigoplus_{\alpha\in [\Ac]}H^{(\alpha)}(\Ac)
\end{equation}
where $H^{(\alpha)}(\Ac)=\{f\in H(\Ac)\,;\,f(A)=0\,\text{unless}\, [A]=\alpha\}$.
If $\Ec\subset\Ac$ is an {\it exact subcategory}, 
i.e., a full subcategory closed under extensions, then we have a
$k$-subalgebra $H(\Ec)\subset H(\Ac)$ formed by the functions $f$ such that
$\supp(f)\subset\Mc(\Ec)$. Let $\Ac^{\op}$ be the opposite category of $\Ac$. 
If $\Ec\subset\Ac$ is as above, then $\Ec^\op$ is an exact subcategory in $\Ac^\op$, and we have
$H(\Ec^\op)=H(\Ec)^\op$, the $k$-algebra opposite to $H(\Ec)$. 
By a {\it perfect duality} on a category $\Ec$
we mean an equivalence of categories $D:\Ec\to\Ec^\op$ 
such that $D^2\simeq 1_\Ec$.

\begin{prop}\label{prop:Hall-involution}
(a) Let $\Ec$ be an exact subcategory in a finitary abelian category $\Ac$, 
and $D:\Ec\to\Ec^\op$ be a perfect duality. 
Then the operator $f\mapsto f^*$, $f^*(E)=f(D(E))$ is a $k$-linear 
anti-involution on the $k$-algebra $H(\Ec)$.

(b) If $k=\CC$ the operator $f\mapsto f^\star$, 
$f^\star(E)=\overline{f(D(E))}$ is a $\CC$-antilinear anti-involution on $H(\Ec)$.
\end{prop}

The space $\widehat H(\Ac)=\Fc(\Mc(\Ac))$ is, in general, not a $k$-algebra. 
It is the algebraic dual of $H(\Ac)$ by the orbifold scalar product. 
Similarly, let  $ H(\Ac)\widehat\otimes H(\Ac)=\Fc(\Mc(\Ac)\times\Mc(\Ac))$
 be the algebraic dual of 
$H(\Ac)\otimes H(\Ac)$. The Hall multiplication on $H(\Ac)$ gives, 
by dualization, the map
\begin{equation}
\label{eq:Hall-comultiplication}
\Delta=\Delta_\Ac:\widehat H(\Ac)\to H(\Ac)\widehat\otimes  H(\Ac),\quad
(\Delta(f),g\otimes h)=(f,g*h).
\end{equation}
In terms of the diagram (\ref{eq:SES-Hall}) we have $\Delta(f)=(p_1,p_2)_*(p^*f\cdot 
\langle B,A\rangle )$. Note also that
\begin{equation}
\label{eq:Hall-comultiplication-explicit}
\Delta(1_C)=\sum_{A\subset C}\langle C/A,A\rangle\,{|\Aut(A)|\cdot|\Aut(C/A)|\over |\Aut(C)|}\,
1_A\otimes 1_{C/A},
\end{equation}
where the sum is over all subobjects of $A$.
We will say that an abelian category $\Ac$ is {\it cofinitary},
 if each object of $\Ac$ has only
finitely many subobjects. If $\Ac$ is finitary and cofinitary,
 then both functors in (\ref{eq:SES-Hall}) are proper, and 
$\Delta$ restricts to the Hall algebra
$$\Delta:H(\Ac)\to H(\Ac)\otimes H(\Ac).$$
The Cartan form on $[\Ac]$ makes
the category $\Vect_{[\Ac]}$ of $[\Ac]$-graded $k$-vector spaces
$E=\bigoplus_{\alpha\in\Mc(\Ac)}E^{(\alpha)}$ 
into a braided monoidal category.  The braiding is given by
$$R_{E,F}:E\otimes F\to F\otimes E,\quad
 e\otimes f\mapsto \((\deg(e),\deg(f)\))\cdot f\otimes e,$$
where $\deg(e),\deg(f)\in[\Ac]$ are the degrees of  $e,f$.
In particular, one can speak about bialgebras in $\Vect_{[\Ac]}$
with respect to this braiding.
An abelian category is called 
{\it hereditary} if it has homological dimension at most 1.

\begin{thm}[Green \cite{green}] \label{thm:green} 
Let $\Ac$ be an abelian category which is finitary, 
cofinitary and hereditary. Then 
$H(\Ac)$ is a braided bialgebra in $\Vect_{[\Ac]}$, i.e., 
$\Delta$ is a morphism of $k$-algebras, where the multiplication 
in $H(\Ac)\otimes H(\Ac)$ is given by
\begin{equation}
\label{eq:twisted-tensor-product-Hall}
\bigl(1_A\otimes 1_B\bigr)\,\bigl(1_C\otimes 1_D\bigr)
\,\,=\,\,
\(( B,C\)) \cdot \bigl(1_A * 1_C\bigr)\otimes\bigl(1_B * 1_D\bigr).
\end{equation}
\end{thm}

\vskip 3mm


\subsection{The Hall algebra of a curve} 
>From now on, let $X$ be a smooth connected projective curve over $\FF_q$.
For any closed point $x\in X$ the field $\FF_q(x)$ is a finite extension 
of $\FF_q$. Let $\deg(x)$ be its degree. 
Let $\Coh(X)$ be the category of coherent sheaves on $X$.
This is an abelian category which is finitary and hereditary, 
but not cofinitary. Consider the following full subcategories:
\begin{itemize}
\item $\Bun(X)$ is the exact category of vector bundles on $X$;
we write $\Bunn(X)$ for the set of isomorphism classes of objects of $\Bun(X)$.
\item $\Tor(X)$ is the abelian category of torsion sheaves, it is cofinitary,
\item $\Tor_x(X)$ is the abelian category of torsion sheaves supported at $x$.
\end{itemize}
Note that we have a decomposition into a direct sum of categories
$$\Tor(X)\,\,=\,\,\bigoplus_{x\in X}\Tor_x(X).$$
Let $K(X)=[\Coh(X)]$. The Euler form on $K(X)$ 
takes values in $\QQ(\sqrt{q})$. 
Assume that  $\sqrt{q}\in k$. We consider the following Hall algebras
$$H_\coh=H(\Coh(X)),\quad
H=H(\Bun(X)),\quad
A=H(\Tor(X)),\quad
A_x=H(\Tor_x(X)).$$
We denote by $p_\bun:H_\coh\to H$ the projection given by 
$$p_\bun(1_\Fc)=1_\Fc,\,\, \text{if}\  \Fc\ \text{is a bundle},
\quad
p_\bun(1_\Fc)=0,\,\, \text{otherwise}.$$
The rank and the degree of vector bundles yield a map
$$(\rk,\deg):K(X)\to\ZZ^2.$$
The Euler form on $K(X)$ factors through 
$\ZZ^2$ and is given by
\begin{equation}
\label{eq:euler-form-K(X)}
\log_q\langle(r,d),(r',d')\rangle={1\over 2}\bigl(
rd'-r'd+(1-g_X)rr'\bigr),
\end{equation}
where $g_X$ is the genus of $X$. We have  bigradings
\begin{equation}\label{eq:rank-degree-bigrading}
H_\coh=\bigoplus_{(r,d)\in \ZZ_+^2}H_\coh^{(r,d)},
\quad H=\bigoplus_{(r,d)\in \ZZ_+^2}H^{(r,d)},
\end{equation}
where $\ZZ_+^2=\{(r,d)\;;\; (r,d) \geqslant (0,0)\}$,
with the inequalities understood with respect to
the lexicographic order on $\ZZ^2$.
The bigrading on the $k$-subalgebra $H$ is the induced one.
We write
$$H_\coh^{(r)}=\bigoplus_{d\in \ZZ}H_\coh^{(r,d)},
\quad H^{(r)}=\bigoplus_{d\in \ZZ}H^{(r,d)}.$$

There is a perfect duality on $\Bun(X)$ given by
$\Fc\mapsto\Fc^*$ (the dual vector bundle).
So, by Proposition \ref{prop:Hall-involution} 
the $k$-algebra $H$ has an involutive antiautomorphism
\begin{equation}
\label{involution}
H\to H,\quad f\mapsto f^*,\quad
f^*(\Fc)={f(\Fc^*)}.
\end{equation}
%


\vskip3mm

\subsection{The comultiplication}

If $\Mc$ is an orbifold and $\varsigma: \operatorname{Ob}(\Mc) \to \ZZ^l$ 
is any map, we let $\Fc_{\varsigma}(\Mc)$ be the set of isomorphism invariant functions
 $f: \operatorname{Ob}(\Mc) \to k$ for which $\varsigma(supp(f))$ is a finite set. For any
$n \geqslant 1$, consider the map
$$\varsigma: \operatorname{Ob}(\Mc(Coh(X)^{\times n}) \to \ZZ^2,
 \quad
  (\Fc_1, \ldots, \Fc_n) \mapsto 
  \big(\sum_i \text{rk}\;\Fc_i, \sum_i \text{deg}\;\Fc_i\big)$$
and set
 ${H}^{\widetilde{\otimes} n}_{\coh}=\Fc_\varsigma(\Mc(Coh(X)^{\times n})$. 
 We will abbreviate $\widetilde{H}_{coh}=H_{coh}^{\widetilde{\otimes} 1}$. For any $n$ we have a bigrading
$${H}^{\widetilde{\otimes} n}_{\coh}=\bigoplus_{(r,d) \in \ZZ_+^2} \big({H}^{\widetilde{\otimes} n}_{\coh}\big)^{(r,d)}$$
and ${H}^{\widetilde{\otimes} n}_{\coh}$ is the \textit{graded} algebraic dual of $H^{\otimes n}_{\coh}$.
There is an obvious chain of (strict) inclusions $H_{\coh}^{\otimes n} \subset {H}^{\widetilde{\otimes}n}_{\coh} \subset \widehat{H}^{\widehat{\otimes}n}_{\coh}$.
We equip $H^{\otimes n}_{\coh}$ with the following twisted multiplication
\begin{equation*}
\label{eq:twisted-product-Hcoh}
\bigl(1_{\Ec_1}\otimes\cdots\otimes 1_{\Ec_n}\bigr)*
\bigl(1_{\Fc_1}\otimes\cdots\otimes 1_{\Fc_n}\bigr)=
\Bigl(\prod_{i>j}\((\Ec_i,\Fc_j\)) \Bigr)\,
\bigl(1_{\Ec_1}*1_{\Fc_1}\bigr)\otimes\cdots\otimes
\bigl(1_{\Ec_n}*1_{\Fc_n}\bigr).
\end{equation*}
With respect to this product $H^{\otimes n}_\coh$ is a $(\ZZ^2_+)^n$-graded
$k$-algebra.

\begin{prop}
\label{prop:def-coprod-completion}
(a) The $k$-algebra structure on $H^{\otimes n}_{\coh}$ extends to an associative $k$-algebra structure on
${H}^{\widetilde{\otimes} n}_{\coh}$,

(b) For any $i=1, \ldots, n-1$ the map $Id^{\otimes i-1} \otimes \Delta_{\coh} \otimes Id^{\otimes n-1-i}$ takes
${H}^{\widetilde{\otimes}n-1}_{\coh}$ into ${H}^{\widetilde{\otimes} n}_{\coh}$,

(c) The iterated comultiplication $\Delta^{(n-1)}: \widetilde{H}_{\coh} \to {H}^{\widetilde{\otimes} n}_{\coh}$ is a morphism of $k$-algebras.
\end{prop}
\begin{proof}
Let $n \geqslant 1$, and let $(r,d), (r',d') \in \ZZ^2_+$. To show that the
 multiplication map
  $(H^{\widetilde{\otimes} n})^{(r,d)} \otimes (H^{\widetilde{\otimes n}})^{(r',d')}
   \to (H^{\widetilde{\otimes} n})^{(r+r',d+d')}$ 
   is well-defined, we need to establish the following fact~: 
   for any $n$-tuple of coherent sheaves $(\Fc_1, \ldots, \Fc_n)$ satisfying
$\sum_i \text{rk}\;\Fc_i=r+r', \sum_i \text{deg}\;\Fc_i=d+d'$ , the number of isomorphism classes of $n$-tuples of subsheaves $\Gc_1 \subset \Fc_1, \ldots, \Gc_n \subset \Fc_n$ satisfying $\sum_i \text{rk}\;\Gc_i=r, \sum_i \text{deg}\;\Gc_i=d$, is finite. This in turn may easily be deduced by induction from the following two facts~:

\vspace{.05in}

(i) for any coherent sheaf $\Fc$ and any fixed $r \leqslant \text{rk}\;\Fc$ the possible degrees of subsheaves of
$\Fc$ of rank $r$ is bounded above,

(ii) for any coherent sheaf $\Fc$ and any fixed $(r,d) \in \ZZ^2_+$ satisfying $r \leqslant \text{rk}\;\Fc$, the number of subsheaves of $\Fc$ of rank $r$ and degree $d$ is finite.

\vspace{.05in}

This proves (a). Statement (b) is clear since $Id^{\otimes i-1} \otimes \Delta_{\coh} \otimes Id^{n-1-i}$ takes
$H_{\coh}^{\widehat{\otimes} n-1}$ to $H_{\coh}^{\widehat{\otimes} n}$ and preserves the total $\ZZ^2_+$-grading. Finally, we prove (c). Because of the coassociativity of $\Delta_\coh$, it is enough to verify
this for the case $n=2$. Green's proof of Theorem \ref{thm:green} extends to this case, even though $Coh(X)$ is not coinitary-- see e.g. \cite{schiffmann:lectures} for details.  
\end{proof}

Let $H^{\widehat\otimes n}\subset H^{\widehat\otimes n}_\coh$, resp.  $ H^{\widetilde\otimes n}\subset H^{\widetilde\otimes n}_\coh$
be the subspace of 
functions supported on $\Bun(X)^n$, and let
$p_n: H^{\widehat\otimes n}_\coh\to H^{\widehat\otimes n}$, $p_n: H^{\widetilde\otimes n}_\coh\to H^{\widetilde\otimes n}$ be the obvious projections.
We denote by
\begin{equation}\label{2.38}
\Delta^{(n-1)} = p_n\circ\Delta_\coh^{(n-1)} :H\lra H^{\widetilde\otimes n}.
\end{equation}
the composition of $\Delta_\coh^{(n-1)}$ and the projection on the
bundle part. Note that $\Delta^{(n-1)}$ is not an algebra homomorphism.


\vskip3mm

\subsection{The Hecke algebra}

The Euler form is identically equal to 1 on $[\Tor(X)]$.
 Since $\Tor(X)$ is cofinitary, 
we see that $A$ is a $k$-bialgebra in the usual sense. So is each $A_x$, $x\in X$. 
We have an identification of $k$-bialgebras
$A=\bigotimes_{x\in X}A_x$ (restricted tensor product). There is a perfect duality on $\Tor(X)$ given by
\begin{equation}\label{E:dual-torsion}
\Fc\mapsto\Ec xt^1_{\Oc_X}(\Fc,\Oc_X).
\end{equation} 
Since $\Fc\simeq\Ec xt^1_{\Oc_X}(\Fc,\Oc_X)$ (not canonically), 
the $k$-algebras $A$ is commutative. Since the comultiplication is dual to the multiplication
with respect to the orbifold scalar product on $A$, we deduce that $A$ is cocommutative as well. 
The same holds for the $k$-bialgebra $A_x$ for any $x \in X$. 

One can in fact be much more precise. The completion $\widehat\Oc_{X,x}$ of the local ring of $x$
is a complete discrete valuation ring with residue field $\FF_{q_x}$,
$q_x=q^{\deg(x)}$. Let $\Oc_x$ be the skyscraper sheaf at $x$ with stalk
$\FF_{q_x}$. The $k$-algebra $A_x$ can be seen as 
the ``classical" Hall algebra
of the category of finite $\widehat\Oc_{X,x}$-modules
\cite{macdonald}. 
In particular the elements 
$$b_{x,r}\,=\, q_x^{\frac{r(r-1)}{2}}1_{\Oc_x^{\oplus r}},\quad r\geqslant 1,$$
are free polynomial generators of $A_x$, and their coproduct is given by
\begin{equation}\label{E:coprod-er}
\Delta(b_{x,r})=\sum_{i=0}^rb_{x,i}\otimes b_{x,r-i},\quad b_{x,0}=1.
\end{equation}
The counit is given by $\epsilon(1_{\mathcal{F}})=0$ if $\mathcal{F}\neq 0$.

The algebra $A$ is called the {\it Hecke algebra}.
For $\Fc\in\Tor(X)$ we have the {\it Hecke operator}
\begin{equation}
\label{eq:heckeop}
T_\Fc:\widehat H\to\widehat H,\ (T_\Fc f)(V)=\sum_{V'}\langle\Fc,V'\rangle\,f(V'),
\end{equation}
where the sum runs over all subsheaves (necessarily locally free) 
$V'\subset V$ such that $V/V'$ 
is isomorphic to $\Fc$. Here we view $\widehat H$ as the space of all 
functions on the set of isomorphism classes in $\Bun(X)$.
We also define the {\it dual Hecke operator}
\begin{equation}
\label{eq:dualheckeop}
T_\Fc^*:\widehat H\to\widehat H,\ (T_\Fc^* f)(V)=
\sum_{U}\langle\Fc,V\rangle\,f(U),
\end{equation}
where $U$ runs over overbundles of $V$ such that $U/V\simeq\Fc$, taken modulo
isomorphisms identical on $V$. Observe that
\begin{equation}\label{E:hecke-counit}
 T_{\mathcal{F}}(1)=T^*_{\mathcal{F}}(1)=\epsilon(1_{\mathcal{F}})1.
\end{equation}

\begin{prop}
\label{prop:properties-heckeop}
$(a)$ Both $T_\Fc$ and $T_\Fc^*$ preserve $H$ and $\widetilde{H}$.

$(b)$ $T_\Fc^*$ is adjoint of $T_\Fc$ w.r.t. the orbifold scalar product
\eqref{scalarproduct}.
If $k=\CC$, then $T_\Fc^*$ is also adjoint of 
$T_\Fc$ w.r.t. the Hermitian  scalar product \eqref{hermitianproduct}. 

$(c)$ For $f\in H$ we have
$T_\Fc(f)=p_\bun(f*1_\Fc).$

$(d)$ For $f\in H$ we have
$T_\Fc^*(f)=(T_{{\Fc}^*}(f^*))^*$, 
where $f^*$ is as in (\ref{involution}), and where $\Fc^*$ is the dual in the sense
of (\ref{E:dual-torsion}).

$(e)$ We have
$T_{\Fc_1}T_{\Fc_2}=\sum_\Fc g_{\Fc_2,\Fc_1}^\Fc T_\Fc$,
and so $\widehat H$, $\widetilde{H}$ and $H$ are right $A$-modules. 
The subspaces $\widehat H^{(r)}$, $\widetilde{H}^{(r)}$ and $H^{(r)}$ are submodules.
By $(d)$ the same statement is true for $T_\Fc^*$. This yields left
$A$-modules structures on 
$\widehat H$, $\widetilde{H}$, $H$, $\widehat H^{(r)}$, $\widetilde{H}^{(r)}$, $H^{(r)}$.
\end{prop}

\begin{proof}
This proposition is well-known. 
Let us give a proof for the comfort of the reader.
Consider the action maps
\begin{equation}
\label{action1}
\begin{split}
\gamma^*:A\otimes\widehat H\to\widehat H,\quad \gamma(1_\Fc\otimes f)=T^*_\Fc(f),\cr
\gamma:\widehat H\otimes A\to\widehat H,\quad \gamma(f\otimes 1_\Fc)=T_\Fc(f).
\end{split}
\end{equation}
Let $\Sc$ be the orbifold of short exact sequences
\begin{equation}\label{eq:hecke-orbifold}
0\to V'\to V\to \Fc\to 0
\end{equation}
with $V,V'\in \Bun(X)$ and $\Fc\in\Tor(X)$, with obvious projections
$$p,p':\Sc\to\Mc(\Bun(X)),\quad
q:\Sc\to\Mc(\Tor(X)).$$
Then $\gamma$ is given by
\begin{equation}
\label{eq:hecke-action-orbifold}
\gamma(f)=p'_*\bigl((p,q)^*(f)\cdot\langle\Fc,V'\rangle\bigr),
\end{equation}
where we regard $f\in  \widehat H \otimes A$ as a function on
$\Mc(\Bun(X)) \times \Mc(\Tor(X))$.
Since $(p,q)$ is proper, we have $\gamma( H\otimes A)\subset H$.
Similarly $\gamma^*$ is given by
\begin{equation}
\label{eq:dual-hecke-action-orbifold}
\gamma^*(f)=p_*\bigl((q,p')^*(f)\cdot\langle\Fc,V\rangle\bigr),
\end{equation}
which implies $(a)$, $(b)$.
Next, $(c)$ is obvious from comparing (\ref{eq:dual-hecke-action-orbifold}) 
with the formula for the Hall multiplication, 
which involves an orbifold similar to $\Sc$ but without
the requirement that $V\in\Bun(X)$, see Proposition
\ref{prop:SES-orbifold-Hall}. To see part $(d)$, notice that sending
a short exact sequence (\ref{eq:hecke-orbifold}) into
$$0\to V^*\to V'^*\to \Ec xt^1(\Fc,\Oc_X)\to 0$$
defines a  perfect duality on the orbifold $\Sc$, interchanging the projections
$p$, $p'$ and composing $q$ with the duality (\ref{E:dual-torsion}).
Finally, $(e)$ follows from associativity of the Hall multiplication, together with the obvious identity
\begin{equation}\label{E:bundles}
p_{\text{bun}}(f * g)=p_{\text{bun}}(p_{\text{bun}}(f) * g).
\end{equation}
\end{proof}

The left (resp.~right) $A$-module structure on $H$ yields left (resp.~right)
$A$-module structures on $H\otimes H$ and $H \widetilde{\otimes} H$ via the comultiplication 
$\Delta:A\to A\otimes A$.

\begin{prop}
\label{prop:hall-mult-hecke-comult}

(a) The multiplication $\mu:H\otimes H\to H$ is a morphism of left 
(as well as right) $A$-modules.

(b) The comultiplication $\Delta: H \to H \widetilde{\otimes} H$ is a morphism of left (as well as right)
$A$-modules.

\end{prop}
\begin{proof} Let us first prove (a). By 
Proposition~\ref{prop:properties-heckeop} (d)
 it suffices to consider the right $A$-module structure.
Let $V,V'\in \Bun(X)$ and let $\Fc\in \Tor(X)$. By 
Proposition~\ref{prop:properties-heckeop} (c) we have 
$T_{\Fc} ( 1_{V'} * 1_V)=p_{\text{bun}}(1_{V'}*1_V*1_{\Fc})$.
Using (\ref{E:bundles}) it suffices to check the following formula~:
\begin{equation}
\label{eq:hall-product-cross2}
1_V*1_\Fc=\sum_{\Fc_1\subset\Fc}\frac{|\Aut(\Fc_1)| | \Aut(\Fc_2)|}{|\Aut(\Fc)|}1_{\Fc_1}*T^*_{\Fc_2}(1_V),
\quad \Fc_2=\Fc/\Fc_1. 
\end{equation}
This may be proved along the lines of \cite{schiffmann-vasserot:highergenus}, Proposition 1.2.
We give here a different proof, using orbifolds.
Both sides of (\ref{eq:hall-product-cross2}) are linear operators
$\lambda,\rho:H\otimes A
\to H_\coh.$
Let $\Cc$ be the orbifold of 
short exact sequences
\begin{equation}
\label{seq1}
0\to V\to\Ec\to\Fc\to 0,\end{equation}
with 
$V\in\Bun(X),$
$\Ec\in\Coh(X),$
$\Fc\in\Tor(X)$ and the
obvious projections
$$\xymatrix{
\Mc(\Bun(X))\times\Mc(\Tor(X))
&\Cc\ar[l]_-{(p_V,p_\Fc)}\ar[r]^-{p_\Ec}&\Mc(\Coh(X))}.$$
Then we have
$$\lambda(f)=(p_\Ec)_*\bigl((p_V,p_\Fc)^*(f)\,\langle\Fc,V\rangle\bigr).$$
Next, let $\Dc$ be the orbifolds of diagrams
\begin{equation}
\label{seq2}
\begin{split}
\xymatrix{
&&0&0&\cr
0\ar[r]& V\ar[r]&E\ar[r]\ar[u]&\Fc_2\ar[r]\ar[u]&0\cr
&&\Ec\ar[u]&\Fc\ar[u]&\cr
&&&\Fc_1\ar[u]\ar[ul]&\cr
&&&0\ar[u]&0\ar[ul]
}
\end{split}
\end{equation}
with
$V,E\in\Bun(X),$
$\Ec\in\Coh(X),$
$\Fc,\Fc_1,\Fc_2\in\Tor(X)$
and the obvious projections
$$\xymatrix{
\Mc(\Bun(X))\times\Mc(\Tor(X))
&\Dc\ar[l]_-{(\pi_V,\pi_\Fc)}\ar[r]^-{\pi_\Ec}&\Mc(\Coh(X))}.$$
Then we have
$$\rho(f)=(\pi_\Ec)_*\bigl((\pi_V,\pi_\Fc)^*(f)\,\langle\Fc_2,V\rangle\,
\langle E,\Fc_1\rangle\bigr).$$
We now construct a functor $\varphi:\Cc\to\Dc$ 
giving a commutative diagram of orbifolds
$$\xymatrix{
&\Cc\ar[dl]_-{(p_V,p_\Fc)}\ar[dr]^-{p_\Ec}\ar[dd]^-\varphi &\cr
\Mc(\Bun(X))\times\Mc(\Tor(X))&&\Mc(\Coh(X))\cr
&\Dc\ar[ul]^-{(\pi_V,\pi_\Fc)}\ar[ur]_-{\pi_\Ec}&\cr
}$$
by associating to a sequence as in (\ref{seq1}) the diagram as in
(\ref{seq2}) where
$\Fc_1=\Ec_\tor$ is the maximal torsion subsheaf in $\Ec$
and the surjection $\Ec\to\Fc$ identifies $\Fc_1$ with a subsheaf in $\Fc$.
The statement follows from the

\begin{lemma} We have $\varphi_*(1)=\langle E,\Fc_1\rangle^2$.
Hence $\varphi_*(\langle\Fc,V\rangle)=
\langle\Fc_2,V\rangle\,\langle E,\Fc_1\rangle$.
\end{lemma}

\begin{proof}
It is enough to prove the first equality, as the second
is obtained by using bi-multiplicativity of the Euler form.
Note that $\varphi$ is bijective on isomorphism classes 
and injective on morphisms.
On the other hand, for any short exact sequence $C$ as in (\ref{seq1})
we have
$$|\Aut_\Dc\varphi(C)|=|\Hom(E,\Fc_1)|\cdot|\Aut_\Cc(C)|,$$
as we can have automorphisms of $\Ec=E\oplus\Fc_1$ sending $E$ to $\Fc_1$. Note
finally that 
$|\Hom(E,\Fc_1)|=\langle E,\Fc_1\rangle^2$, since $\Ext^1(E,\Fc_1)=0$.
This proves the lemma, and statement (a).
\end{proof}
Statement (b) follows from (a) by duality. Namely, using the notation in (\ref{action1}), it is enough to prove that
\begin{equation}
\label{eq:hecke-comult-compatibility-tildeH}
\Delta(\gamma^*(a\otimes f))=
\sum\gamma^*(a_1\otimes f_1)\otimes\gamma^*(a_2\otimes f_2),
\end{equation}
where $a\in A$, $f\in H$ and $\Delta(a)=\sum a_1\otimes a_2$,
$\Delta(f)=\sum f_1\otimes f_2.$
To prove this, we take the scalar product with an arbitrary 
$x\otimes y\in H\otimes H$. Propositions 
\ref{prop:properties-heckeop}, \ref{prop:hall-mult-hecke-comult}
 yield
$$\aligned
\bigl(\Delta(\gamma^*(a\otimes f)),x\otimes y\bigr)&=
\bigl(\gamma^*(a\otimes f),xy)\bigr)\cr
&=\bigl(f,\gamma(xy\otimes a)\bigr)\cr
&=\sum\bigl(f,\gamma(x\otimes a_1)
\gamma(y\otimes a_2)\bigr)\cr
&=\sum\bigl(\Delta(f),\gamma(x\otimes a_1)
\otimes\gamma(y\otimes a_2)\bigr)\cr
&=\sum\bigl(\gamma^*(a_1\otimes f_1)\otimes\gamma^*(a_2\otimes f_2)
,x\otimes y\bigr). 
\endaligned $$
\end{proof}

\vskip3mm


\subsection{The $k$-algebra $H_\coh$ as a cross-product}


Let $B$ be a Hopf $k$-algebra, $R$ be a $k$-algebra which is a right $B$-module
with action $(r,b)\mapsto r\triangleleft b$ such that the multiplication
in $R$ is a morphism of right $B$-modules. Then the 
{\it cross-product $k$-algebra}
$B\ltimes R$ is the vector space $B\otimes R$ with the multiplication
\begin{equation}
\label{eq:cross-product-def}
(b\otimes r)(b'\otimes r')=\sum bb'_1\otimes(r\triangleleft b'_2)r',\quad
\Delta(b')=\sum b'_1\otimes b'_2.
\end{equation}
See \cite{majid}, \S 6.1, for background. 
We  apply this to $B=A$, $R=H$ and the action given by
\begin{equation}
\label{action}
f\triangleleft 1_\Fc=T_\Fc(f),\quad
\Fc\in\Tor(X).
\end{equation}

\begin{prop}\label{prop:Hcoh-cross-product}
The cross-product $k$-algebra $A\ltimes H$ is isomorphic  to  $H_\coh$.
\end{prop}

\begin{proof} By (\ref{eq:hall-product-cross2}) the map $\iota:A \ltimes H \to H_{\coh}, a \otimes x \mapsto a * x$
is an algebra homomorphism.
Since every coherent sheaf $\Fc$ on $X$ has a unique torsion subsheaf
$\Fc_\tor$ such that the quotient is locally free, 
the multiplication in $H_\coh$ yields an isomorphism of vector spaces
\begin{equation}
\label{eq:A-otimes-H-Hcoh}
A\otimes H\lra H_\coh
\end{equation}
i.e., $\iota$ is an isomorphism.
\end{proof}

\vskip 3mm


\subsection{The local Witt scheme}
For $x\in X$ we set $W_x=\Spec(A_x)$. 
We call $W_x$ the {\it local Witt scheme at $x$}.
For a character ($k$-algebra homomorphism) $\chi:A_x\to R$
we write 
\begin{equation}\label{eq:polynomial-B}
B_\chi(t)=1+\sum_{r\geqslant 1}b_r(\chi)\,t^r\in R[[t]],\quad
b_r(\chi)=\chi(b_{x,r}).
\end{equation}
This yields an isomorphism $\Hom(A_x,R)=1+tR[[t]]$.
As $A_x$ is a cocommutative Hopf $k$-algebra, 
$W_x$ is a commutative affine group scheme. The group structure comes from the 
comultiplication $\Delta$ in $A_x$, which is given by (\ref{E:coprod-er})

\begin{prop}
The group scheme $W_x$ is isomorphic to the {\it classical Witt scheme}, 
i.e., for any commutative $k$-algebra $R$ the group of $R$ points $W_x(R)$ 
is identified with $\Hom(A_x,R)=1+tR[[t]]$.\qed
\end{prop}

\noindent 
See \cite{mumford} for background on the Witt scheme.
Following the tradition,  we denote the group operation in $W_x$ 
additively, by the symbol
$\boxplus$.

The following is rather standard. Its proof is left to the reader.

\begin{lemma}
\label{lem:witt-scheme-points}
Let $R$ be a commutative Hopf $k$-algebra, so that $G=\Spec(R)$ 
is an affine group scheme.

$(a)$ Group $k$-schemes morphisms $G\to W$ are in bijection with series
$\phi(t)\in 1+tR[[t]]$ which are group-like, i.e., such that
$\Delta\phi(t)=\phi(t)\otimes\phi(t)$.

$(b)$ Group $k$-schemes bihomomorphisms $G\times G\to W$ 
are in bijection with series
$\Phi(t)\in 1+t(R\otimes R)[[t]]$ such that
the following holds in $(R\otimes R\otimes R)[[t]]$
$$\gathered
(1\otimes\Delta)\Phi(t)=\Phi_{12}(t)\Phi_{13}(t),\quad
(\Delta\otimes 1)\Phi(t)=\Phi_{13}(t)\Phi_{23}(t).
\endgathered$$ \qed
\end{lemma}

We say that a character $\chi:A_x\to R$
has rank $\leqslant r$ (resp.~rank $r$) if $B_\chi(t)$
is a polynomial of degree $\leqslant r$ (resp.~of degree $r$ with invertible 
highest coefficient). Rank $\leqslant r$ (resp.~rank $r$) characters form a 
closed (resp.~locally closed) subscheme $W_x^{\leqslant r}$ (resp.~$W_x^r$)
in $W_x$ of the form
$$W_x^r=\Spec(A_x^{(r)}),\quad
A_x^{(r)}=k[b_1,\dots,b_r,b_r^{-1}]$$
with $\overline{W^r_x}=W_x^{\leqslant r}=\Spec(k[b_1,\dots,b_r]).$
We also introduce the ramified covering
$$\widetilde W_x^r=\Spec(k[\lambda_1^{\pm 1},\dots,\lambda_r^{\pm 1}])
\to W_x^r,\quad
(\lambda_1,\dots,\lambda_r)\mapsto B(t)=\prod_i(1-\l_it).$$
For a rank $r$ character $\chi$ let
$\{\l_1(\chi), \ldots, \l_r(\chi)\}$ be the corresponding roots of $B_\chi(t)$ 
(defined up to permutation). We have an involution $\chi\mapsto\chi^*$
on $W_x^r$ defined by 
$$\l_i(\chi^*)=\l_i(\chi)^{-1},\quad
b_i(\chi^*)=b_{r-i}(\chi)b_{r}(\chi)^{-1}.$$
For a character $\chi:A_x\to k$ the {\it Euler factor} corresponding to $\chi$
is defined as the inverse of the corresponding series
\begin{equation}
\label{eulerfactor}
L(\chi;t)={1\over B_\chi(t)}=\prod_{i=1}^r{1\over 1-\lambda_i(\chi)t}.
\end{equation}
Here the last equality holds if $\chi$ is of rank $r$.
The $k$-algebra $A_x$ is graded by $\deg(b_{x,n})=n$. This grading yields
a $\GG_m$-action on $W_x$. Since $\Delta$ is homogeneous, this action is 
by group-scheme automorphisms. 
For $\Fc\in\Tor_x(X)$ we denote by $\length(\Fc)$ its length. 
Note that $\l\in k^\times$ takes a character
$\chi:A_x\to k$ into the character
$$\chi\lambda^{\length}:1_\Fc\mapsto\chi(1_\Fc)\,\lambda^{\length(\Fc)},$$
and that we have
\begin{equation}
\label{2.5}
B_{\chi\lambda^{\length}}(t)=B_\chi(\lambda t),\quad
L(\chi\lambda^{\length},t)=L(\chi,\lambda t).
\end{equation}

\vskip3mm


\subsection{The global Witt scheme}
We define the {\it global Witt scheme} of $X$ as the 
(infinite) product of the local ones
$$W_X=\prod_{x\in X}W_x=\Spec(A).$$
This is an commutative affine group scheme with respect to the componentwise operation.
We have a locally closed subscheme
$$W_X^r=\prod_{x\in X}W_x^r$$
of rank $r$ characters, with the involution 
$$\chi=(\chi_x)_{x\in X}\mapsto\chi^*=(\chi_x^*)_{x\in X}.$$
We also set
$$A^{(r)}=\prod_{x \in X} A_x^{(r)}$$
so that $W_X^r = \Spec(A^{(r)})$.

For a character $\chi:A\to k$ we define the $L$-series
$$L(\chi,t)=\prod_{x\in X}L(\chi_x,t^{\deg(x)})\in 1+tk[[t]].$$
This infinite product converges in $1 + tk[[t]]$ since the number of points of $X$ of any fixed degree 
is finite.
We consider the $\GG_m$-action on $W_X$ corresponding
to the grading of $A$ by the degree. For $\Fc\in\Tor_x(X)$ we have
$$\deg(\Fc)=d_x\,\length(\Fc).$$
So our action of $\GG_m$ is the product of actions on the $W_x$, where
$\l\in\GG_m$ acts on $W_x$ via the action of $\l^{d_x}$ in the previous sense.
On the level of characters, $\l$ takes a character $\chi$ into
$$\l^\deg\chi:1_\Fc\mapsto\l^{\deg(\Fc)}\chi(1_\Fc),\quad 
\Fc\in\Tor(X).$$
On the level of $L$-series we have
$$L(\l^\deg\chi,t)=L(\chi,\lambda t).$$

\vskip3mm


\subsection{The local Witt scheme as a ring scheme}

As well known \cite{mumford}, 
the Witt scheme $W_x$ is not just a group, but a ring scheme.
Let $\boxtimes$ be the multiplication in $W_x$. 
  The  morphism 
$\boxtimes: W_x\times W_x\to W_x$ is uniquely determined by its restrictions
$$\boxtimes: W_x^{\leqslant r}\times W_x^{\leqslant s}\to W_x^{\leqslant rs},$$ 
which are given by
$$\Bigl(\prod_{i=1}^r(1-\lambda_it)\Bigr)\boxtimes
\Bigl(\prod_{j=1}^s(1-\mu_jt)\Bigr)=
\prod_{i=1}^r\prod_{j=1}^s(1-\lambda_i\mu_jt).$$
The neutral elements for $\boxplus$, $\boxtimes$ are given by the 
series ${\bf 0}= 1$ and  $\1 = 1-t$ respectively. The subscheme
$W^1_x\subset W_x$ is a subgroup for $\boxtimes$, 
which is isomorphic to $\GG_m$.
It is formed by polynomials of the form
$$[[\l]]_x=1-\lambda t\in W^1_x,\quad
[[\lambda]]_x\boxtimes[[\mu]]_x=[[\lambda\mu]]_x,
\quad\lambda,\mu\in\GG_m.$$
The action by Witt multiplication of $W^1_x$ on $W_x$ 
is the action considered in (\ref{2.5})
$$[[\lambda]]_x\boxtimes B(t)=B(\lambda t).$$
Let us now describe $\boxtimes$ explicitly.
%
%
%
Consider the following series
$$\Psi(t)=\sum_{\Fc\in\Tor_x(X)}t^{\length(\Fc)}\,|\Aut(\Fc)|\,
(1_\Fc\otimes 1_\Fc)\in 1+t(A_x\otimes A_x)[[t]].$$

\begin{prop}\label{prop:local-product-witt}
(a) The series $\Psi(t)$ defines a bihomomorphism
$$U:W_x\times W_x\to W_x.$$

(b) The bihomomorphism $U$ is $W_x$-bilinear, i.e.,
$$U(\chi\boxtimes\chi',\chi'')=U(\chi',\chi\boxtimes\chi'')=
\chi\boxtimes U(\chi',\chi''),$$
so $U(\chi,\chi')=\chi\boxtimes\chi'\boxtimes\varkappa_x$ for a $\ZZ$-point
$\varkappa_x$ of $W_x$.

$(c)$ The point $\varkappa_x$ is given by $B_{\varkappa_x}(t)=(1+t)/(1+q_xt).$
\end{prop}

\begin{proof}
For part $(a)$ we must check that the series $\Psi(t)$ 
satisfies the conditions of Lemma 
\ref{lem:witt-scheme-points}. This is a consequence of the following lemma, which is proved as in \cite[(4.2)]{K}.

\begin{lemma} For a commutative $k$-algebra $R$ and a $k$-algebra homomorphism 
$\chi:A_x\to R$, the series
$$\Psi_\chi(t)=\sum_{\Fc\in\Tor_x(X)}t^{\length(\Fc)}|\Aut(\Fc)|\,
\chi(\Fc)\,1_\Fc$$
is a group-like element in $(R\otimes A_x)[[t]]$. \qed
\end{lemma}

\noindent To see part $(b)$, it is enough to assume that $\chi\in W_x^{\leqslant r}$ for
some $r$, as the union of the $W_x^{\leqslant r}$ is Zariski dense in $W_x$.
by (a), we reduce to the case when $\chi = [[\lambda]]_x\in W^1_x=\GG_m$. 
Since the action of $[[\lambda]]_x$ by $\boxtimes$ in $W_x$ is the same as the
action $B(t)\mapsto B(\l t)$ on power series and since the latter action 
corresponds to the grading of $A_x$, our statement reduces to the (obvious)
fact that the $n$th coefficient of $\Psi(t)$ has degree exactly $n$ in
each of the tensor variables. 

We now establish part (c). The point $\varkappa_x$ can be found as $U(\1, \1)$.
Since $\1$ is the series $1-t$, we have
$$B_{\varkappa_x}(t) \,\,=\,\,\sum_{\Fc} \alpha(\Fc)\cdot |\Aut(\Fc)|\cdot
t^{\length(\Fc)},$$
where $\alpha: A_x\to \QQ$ is the character sending $b_{x,1} = 1_{\Oc_x}$
to $(-1)$ and 
$b_{x,n}= 1_{\Oc_x^{\oplus n}}$ to $0$, if $n\geqslant 2$. Now, modulo
the ideal generated by the $1_{\Oc_x^{\oplus n}}$, $n\geqslant 2$, we have
$1_\Fc\cong 0$ for any torsion sheaf not of the form $\Oc_X/{\mathfrak m}_x^n$
for some $n\geqslant 0$, and, moreover, we have
 $1_{\Oc_X/{\mathfrak m}_x^n}\cong (1_{\Oc_x})^{*n}$
(the $n$th power in the Hall algebra). This implies
\begin{equation*}\begin{split} B_{\varkappa_x}(t)\,\,=\,\,\sum_{n=0}^\infty (-1)^n \cdot
|\Aut( \Oc_X/{\mathfrak m}_x^n)|\cdot t^n\,\,=\cr
=\,\,1+\sum_{n=1}^\infty (-1)^n (q_x^n-q_x^{n-1})t^n \,\,=\,\,{1+t\over
1+q_xt},
\end{split}
\end{equation*}
as claimed.
\end{proof}

\begin{df} For two characters $\chi,\chi'$ of $A_x$ the 
{\it Rankin-Selberg tensor product Euler factor} is the Euler factor  
$L(\chi\boxtimes\chi',t)$, see (\ref{eulerfactor}). 
If $\chi$ is of finite rank the {\it Rankin-Selberg $\LHom$
Euler factor} is
$$\LHom(\chi,\chi';t)=L(\chi^*\boxtimes\chi';t).$$
\end{df}

\begin{rem}\label{rem:regularL}
For any $r$, the map $(\chi,\chi') \mapsto \text{LHom}(\chi,\chi';t)$ is a regular function on $W_x^r \times W_x$
with values in $1+ tk[[t]]$.
\end{rem}

\vskip 3mm


\subsection{The global Witt scheme as a ring scheme}
\label{sec:global-witt-scheme-ring}
The global Witt scheme $W_X=\prod_{x\in X}W_x$ is a ring scheme for the componentwise product, 
still denoted $\boxtimes$.
We consider the embedding 
$$\GG_m\to W_X,\quad \l\mapsto[[\l]]=([[\l^{\deg(x)}]]_x)_{x\in X}.$$
The $\boxtimes$-multiplication by $\GG_m$ under this embedding
is the action on $W_X$ corresponding to the grading by the degree.
Given two characters $\chi=(\chi_x)$ and $\chi'=(\chi'_x)$ of $A$, we define their Rankin-Selberg tensor product $L$-series by 
$$L(\chi\boxtimes\chi';t)=\prod_{x\in X}L(\chi_x\boxtimes\chi'_x;t^{\deg(x)}),$$
and, when $\chi$ is of finite rank, the LHom-series by
$$\LHom(\chi,\chi';t)=\prod_{x\in X}\LHom(\chi_x,\chi'_x;t^{\deg(x)}).$$
Let us introduce one  extra copy $W_{\operatorname {abs}}$ of the classical Witt scheme,
which we do not identify with any of the $W_x$. 
Lemma
\ref{lem:witt-scheme-points} is applicable to morphisms to $W_{\operatorname {abs}}$. 
We have the map
\begin{equation}
\begin{split}
\operatorname{SP}: W_X\times W_X\lra W_{\operatorname {abs}}, \hskip 5cm \cr
\operatorname{SP}\bigl( (B_x(t))_{x\in X}, \, (B'_x(t))_{x\in X}\bigr) \,\,= 
\,\, \prod_{x\in X} (B_x\boxtimes B'_x)(t^{\deg(x)}).
\end{split}
\end{equation}
called the {\em Witt scalar product}.
By multiplying, over all $x\in X$, the statements of
Proposition \ref{prop:local-product-witt}, we get
\begin{prop} (a) 
Let $\chi:A\to k$ be any character. Then the series
$$\Psi_\chi(t)=\sum_{\Fc\in\Tor(X)}t^{\deg(\Fc)}\cdot |\Aut(\Fc)|\cdot
\chi(\Fc)\cdot 1_\Fc\,\,\in\,\, A[[t]]$$
is group-like.
 The group-scheme homomorphism $W_X\to W_{\operatorname{abs}}$ 
corresponding to $\Psi_\chi$ by Lemma
\ref{lem:witt-scheme-points} takes $\chi'$ to 
$\varkappa\boxtimes\operatorname{SP}(
\chi, \chi')$,
where $\varkappa$ is the $k$-point of $W_X$ corresponding to the series 
$$B_\varkappa(t)={\zeta_X(-t)\over\zeta_X(-qt)}.$$ 
\end{prop}

\vfill\eject

\section{The automorphic picture}

\subsection{The adelic interpretation of $H$ and $A$}
As before, let $X$ be a smooth connected projective curve over $\FF_q$.
We denote by $K=\FF_q(X)$ be the field of rational functions on $X$.
Let $k$ be a field of characteristic zero containing $\sqrt{q}$.
 
Let $K_x$ be the completion of $K$
at $x\in X$. Let $\Aen\subset\prod'_{x\in X}K_x$ be the ring of ad\`eles of $X$ and let 
$\widehat\Oc=\prod'_{x\in X}\widehat\Oc_{x}$ be the subring of integer ad\`eles. Here as usual $\prod'$ stands for the
restricted product. For an ad\`ele
$a=(a_x)$ we write
$$\ord(a)=\sum_{x\in X}\ord(a_x)\,\deg(x).$$
Recall that the set of isomorphism classes of $\Bun_r(X)$ is
$${\Bunn}_r(X)\,\,\simeq \,\,\GL_r(K)\setminus\GL_r(\Aen)/\GL_r(\widehat\Oc).$$
For $g\in\GL_r(\Aen)$ let $E_g$ be the corresponding vector bundle. Its degree is
$\deg(E_g)=\ord(\det(g))$. For $g\in\GL_r(\Aen)$, $h\in\GL_r(\Aen)$ we abbreviate the Euler form
$$\langle g,h\rangle=\langle E_g,E_h\rangle.$$
Recall that $H$ and $A$
denote the Hall algebra of $\Bun(X)$ and $\Tor(X)$. Thus
$\widehat H^{(r)}$ is the space of 
{\it unramified automorphic forms of $\GL_r(\Aen)$}, i.e., of functions 
$f:\GL_r(\Aen)\to k$ which are left $\GL_r(K)$-invariant and right $\GL_r(\widehat\Oc)$-invariant.
The subspace $H^{(r)}$ consists of functions with compact support modulo $\GL_r(K)$.
The right action of $A$ on $\widehat H^{(r)}$ factors
through the homomorphism $A\to A^{(r)}$, with 
$A^{(r)}$ being identified with the convolution algebra of functions on
$\GL_r(\widehat\Oc)\setminus\GL_r(\Aen)/\GL_r(\widehat\Oc)$ with compact support.
For $i\leqslant r$ the element $1_{\Oc_x^{\oplus i}}\in A$ corresponds to the double coset of the matrix
$$\diag(\pi_x,\dots,\pi_x,1,\dots,1)$$
where $\pi_x\in K_x$, a local parameter at $x$, is counted $i$ times.
The comultiplication
$\Delta:\widehat H\to \widehat H^{\widehat\otimes 2}$ is interpreted as the constant term of automorphic forms.
More precisely, the component
\begin{equation}\label{eq:delta-r-s}
\Delta_{r,s}:\widehat H^{(n)}\lra\widehat H^{(r)}\widehat\otimes\widehat H^{(s)},\quad
n=r+s,
\end{equation}
is given by the integral
$$\Delta_{r,s}(f)(g,h)\,\,=\,\, \langle g,h\rangle\int f\begin{pmatrix} g&u\cr0&h\end{pmatrix}
 \,du,$$
where $u$ runs over $N_{r,s}(K)\setminus N_{r,s}(\Aen)$, 
$N_{r,s}$ is the unipotent subgroup, and
$du$ is the Haar measure on $N_{r,s}(\Aen)$ such that the volume of 
$N_{r,s}(K)\setminus N_{r,s}(\Aen)$
is equal to 1.

\begin{df}
An automorphic form $f\in\widehat H^{(r)}$  is  called {\it cuspidal} if $\Delta_{s,t}(f)=0$
for any $s,t>0$ or, equivalently, if $f$ is {\it primitive}, i.e., if $\Delta(f)=f\otimes 1+1\otimes f.$
\end{df}

\vskip 3mm


\subsection{Cusp eigenforms as points of a scheme}
Let $\widehat H_\cusp$, resp.  $\widehat H_\cusp^{(r)}$, resp. $H_\cusp^{(r,d)}$
 be the space of primitive elements in $\widehat H$,
resp. $\widehat H^{(r)}$, resp. $\widehat H^{(r,d)}$.  Let
$$H_\cusp\,\,=\,\, \bigoplus_{r\geqslant 0} H_\cusp^{(r)}, \quad H_\cusp^{(r)}\,\,=\,\,
\bigoplus_{d\in\ZZ} H_\cusp^{(r,d)}$$
be  the  intersection of $\widehat H_\cusp$ with $H$, i.e., the space of cuspidal
functions with finite support. It is a $k$-vector space graded by the rank
and further by the degree.
By Proposition \ref{prop:properties-heckeop} 
each $\widehat H^{(r)}$ is an $A$-module, and $H^{(r)}$ is a submodule, 
the action factoring through $A\to A^{(r)}$. 

\begin{lemma}\label{lem:cusp-forms-module-support}
(a) The subspaces $\widehat H^{(r)}_\cusp\subset \widehat H^{(r)}$
as well as $ H^{(r)}_\cusp\subset  H^{(r)}$
 are  left and right $A$-submodules.

(b)  For each $r,d$ the space $\widehat H^{(r,d)}_\cusp$ coincides
with $H^{(r,d)}_\cusp$, i.e., each cuspidal function on $\Bun_{(r,d)}(X)$ has
finite support. 

(c) Each space $H^{(r,d)}_\cusp$ is finite dimensional.
\end{lemma}

\begin{proof}
The first claim follows from Proposition~\ref{prop:hall-mult-hecke-comult} and (\ref{E:hecke-counit}), because
for $\Delta(f)=f\otimes 1+1\otimes f$ we have
$$\aligned
\Delta(\gamma^*(a\otimes f))
&=
\sum\gamma^*(a_1\otimes 1)\otimes\gamma^*(a_2\otimes f)+
\sum\gamma^*(a_1\otimes f)\otimes\gamma^*(a_2\otimes 1)\cr
&=
\epsilon(a_1)\otimes\gamma^*(a_2\otimes f)+
\gamma^*(a_1\otimes f)\otimes \epsilon(a_2)\cr
&=
1\otimes\gamma^*(a\otimes f)+
\gamma^*(a\otimes f)\otimes 1.
\endaligned$$
The second  and third claims are consequences of the following more
precise fact: there is a finite subset $C\subset \Bunn_{(r,d)}(X)$
such that each $f\in \widehat H^{(r,d)}_\cusp$
vanishes outside $C$. This fact is a particular case of
 \cite{moeglin-waldspurger-book}, Corollary I.2.9. 
\end{proof}

Let $\Hc^{(r)}$, $\Hc^{(r)}_\cusp$ 
be the quasicoherent sheaves on $W_X^r\subset W_X$ 
corresponding to $H^{(r)}$, $H^{(r)}_\cusp$
respectively.
Thus, we have
$$H^{(r)}=\Gamma(W_X^r,\Hc^{(r)}),\quad 
H^{(r)}_\cusp=\Gamma(W_X^r,\Hc_\cusp^{(r)}).$$
Let $\Sigma^{(r)}\subset W_X^r$ be the scheme-theoretic support of 
$H^{(r)}_\cusp$, 
i.e., we set
$$\Sigma^{(r)}=\Spec\bigl(A^{(r)}/\Ann(H_\cusp^{(r)})\bigr),$$
where $\Ann$ denote the annihilator of a module. 
Then we have also
$$H^{(r)}_\cusp=\Gamma(\Sigma^{(r)},\Hc_\cusp^{(r)}).$$
As $H^{(r)}_\cusp=\bigoplus_{d\in\ZZ} H^{(r,d)}_\cusp$
has the grading by the degree, we see that $\Sigma^{(r)}$ 
is invariant under the $\GG_m$-action on $W_X$, and $\Hc^{(r)}_\cusp$
is a $\GG_m$-equivariant quasicoherent sheaf.

\vskip .2cm

Let $R$ be a commutative $k$-algebra. By an {\it $R$-valued cusp eigenform} 
on $\Bun_r(X)$
we mean a non zero function $f: \Bunn_r(X)\to R$ such that
\begin{itemize}
\item $f$ is primitive for $\Delta$,
\item there is a character $\chi_f\in W_X(R)$ such that $T_a(f)=\chi_f(a)\,f$ 
for all $a\in A$.
\end{itemize}

\begin{rem}\label{rem:chi-pi}  
A character of the form $\chi_f$ as above gives rise to a character 
${}^\pi\chi_f: \Pic(X)\to R^\times$ 
such that for any line bundle $L$ on $X$ we have
$$f(V\otimes L)={}^\pi\chi_f(L)\,f(V),\quad V\in\Bunn_r(X).$$
Indeed, the Hecke operator $T_{x,r}=T_{\Oc_x^{\oplus r}}$ sends a function 
$f:\Bunn_r(X)\to R$
to the function $V\mapsto f(V\otimes_{\Oc_X} \Oc_X(x)^*)$ where $\Oc_X(x)$ is the 
line bundle formed by 
functions with at most first order pole at $x$. 
\end{rem}

The following is an algebraic reformulation of the strong 
multiplicity one theorem for unramified cusp
forms on $\GL_r$ which says, see \cite{shalika}, that the joint spectrum of 
the Hecke operators 
on the space of such 
forms is simple.  

\begin{prop}
\label{prop:alg-mult-one}
(a) The sheaf $\Hc^{(r)}_\cusp$ is a free sheaf of $\Oc_{\Sigma^{(r)}}$-modules of rank 1.
 
(b) For any field extension $R$ of $k$, the $R$-valued cusp eigenforms on $\Bun_r(X)$, 
considered modulo  scalar multiples, are in bijection with 
$\Sigma^{(r)}(R)$. The bijection takes an eigenform $f$ to the character $\chi_f$.

(c) The scheme  $\Sigma^{(r)}$ is an algebraic variety over $k$.

(d) The extension of scalars $\Sigma^{(r)}\otimes_k\bar k$ 
is a union of finitely many $\GG_m$-orbits.
\end{prop}

\begin{proof} 
 Consider the restricted dual
$$Q\,\,=\,\,\bigoplus_{d\in\ZZ} (H^{(r,d)}_\cusp)^*.$$
Its left $A$-module structure is dual to the right $A$-module structure
on $H^{(r)}_\cusp$. Let
$f_{\un}: \Bunn_r(X)\to Q$
be the map whose value at a vector bundle $V$ of degree $d$
is the functional on $H^{(r,d)}_\cusp$ given by evaluation at $V$. 
The following is then obvious. 

\begin{lemma} (a) The function $f_\un$ is cuspidal, i.e. we have 
$$\Delta(f_{\un})=f_{\un} \otimes 1 + 1 \otimes f_{\un} \in Q \otimes \widehat{H} \widehat{\otimes} \widehat{H}.$$
It commutes with the $A$-action, i.e., for each $a\in A$ and $V\in\Bunn_r(X)$ we have
$$(T_a f_\un)(V)\,=\, a\cdot (f_\un(V)),$$
where the action on the right is given by the $A$-module structure on $Q$. 

(b) For any $A$-module $M$ and any cuspidal function $f:\Bunn_r(X)\to M$
which commutes with the $A$-action as above, there is a unique
morphism of $A$-modules $\phi: Q\to M$ such that $f=\phi\circ f_\un$. 
\qed
\end{lemma}

 We call $f_\un$ the 
{\em universal cusp form}. Consider a particular case of (b) when $M$ is an
commutative $A$-algebra, i.e., a 
commutative $k$-algebra $R$ equipped with an algebra homomorphism
$\chi: A\to R$. We get:

\begin{cor}\label{cor:cusp-eigenforms}
 Let $R$ be a commutative $k$-algebra and let $\chi: A\to R$
be a character. The
set of $R$-valued cusp eigenforms  with character $\chi$
is identified with $\Hom_A(Q, R)$. \qed
\end{cor}

We now prove part (a) of Proposition \ref{prop:alg-mult-one}. This
part just means that $H^{(r)}_\cusp$ is a cyclic $A^{(r)}$-module. This
is equivalent to saying that $Q$ is a cyclic $A^{(r)}$-module. 
Choose a point $x\in X$, of degree $d_x$. Then the Hecke
operator $T_{x,r}$ is invertible in $A^{(r)}$ and the subalgebra 
in $A^{(r)}$ generated by this operator and its inverse is isomorphic to
the Laurent polynomial ring
$k[z^{\pm 1}]$. Note that tensoring with
$\Oc_X(x)$ identifies $\Bunn_{r,d}(X)$ with $\Bunn_{r, d+rd_x}(X)$. So Remark
\ref{rem:chi-pi} together with Lemma \ref{lem:cusp-forms-module-support}(c)
show that $H^{(r)}_\cusp$ and  $Q$  are  finitely generated 
$k[z^{\pm 1}]$-modules.  Let $\Qc$ be the coherent
sheaf on $\GG'_m=\Spec\, k[z^{\pm 1}]$ corresponding to $Q$. As $Q$
is a graded module over $k[z^{\pm 1}]$, the sheaf $\Qc$ is $\GG'_m$-equivariant.
This implies that $\Qc$ is a vector bundle. The algebra $A^{(r)}$
acts on $\Qc$ by bundle endomorphisms. In this situation the claim that $Q$
is cyclic over $A^{(r)}$ is equivalent to the claim that the fiber
of $\Qc$ at some point $z_0\in\GG'_m$ is a cyclic $A^{(r)}$-module. 
We take $z_0=1$ and take  $\Qc_1^*$, the dual space to the fiber. 
To prove that $\Qc_1^*$ is a cyclic $A^{(r)}$-module, it is enough to assume
that $k=\CC$, the assumption we make until the end of the proof of (a). 
Similarly to Corollary \ref{cor:cusp-eigenforms},
we see that  
$\Qc_1^*=\Hom_{\CC[z^{\pm 1}]}(Q, \CC)$ is identified with the 
(finite-dimensional) space of cusp forms
$f: \Bunn_r(X)\to \CC$ such that $f(V\otimes\Oc_X(x))=f(V)$
for each vector bundle $V$. In this case the classical form of the
multiplicity one theorem,  see \cite{shalika}, Thm. 5.5,
implies that $\Qc_1^*$ is a direct
sum of 1-dimensional invariant subspaces on which $A^{(r)}$ acts
by distinct characters. This, in turn,  implies that $\Qc_1^*$ is cyclic,
establishing part (a).

Let us now prove part (b). Suppose $R/k$ is a field extension and we have  
a character 
$\chi\in\Sigma^{(r)}(R)$. It makes $R$ into an $A^{(r)}$-algebra,
denote it $R_\chi$. In particular, $R_\chi$ is an $A^{(r)}$-module.
By Corollary \ref{cor:cusp-eigenforms}, nonzero 
cusp eigenforms with this character are the same as nonzero
morphisms of $A^{(r)}$-modules
$Q\to R_\chi$. By the proof of (a) above
$Q$ is cyclic, so $Q\simeq k[\Sigma^{(r)}]$ as an $A$-module
by definition of $\Sigma^{(r)}$.
Now, nonzero morphisms of $A^{(r)}$-modules
$k[\Sigma^{(r)}]\to R_\chi$ 
form a torsor over $R^\times$.

As for parts (c) and (d), 
it is enough to establish them for $k=\CC$. In this case
they follow from the description of $\Qc$ as a $\GG'_m$-equivariant
vector bundle on $\GG'_m$ with $A^{(r)}$-action as given in the proof of (a):
the fiber of $\Qc$ over the point $z=1$, and therefore over
any other point, splits into a direct sum of $A^{(r)}$-eigenspaces with distinct
characters. 
 
Proposition \ref{prop:alg-mult-one} is proved. 
\end{proof}



\subsection{Rankin-Selberg L-functions}
\label{rankin-selberg}
The disjoint union $\Sigma=\coprod_r\Sigma^{(r)}$ of the 
algebraic varieties $\Sigma^{(r)}$ is a scheme
of the kind considered in Section \ref{sec:disjoint}. Since $\Sigma^{(r)}$ is closed in 
$W_X^r \subset W_X$, we have a morphism of schemes
\begin{equation}
\label{eq:embedding-sigma}
\alpha: \Sigma\lra W_X. 
\end{equation}
Proposition \ref{prop:alg-mult-one} implies that this
morphism is a {\it categorical monomorphism}, i.e., 
an injection on $R$-points for any commutative $k$-algebra $R$. 
For $\chi\in \Sigma^{(r)}(R)$, $\chi'\in\Sigma^{(s)}(R)$
we have the Rankin-Selberg $L$-series, see Section \ref{sec:global-witt-scheme-ring},
$$\LHom(\chi,\chi';t)=L(\chi^*\boxtimes\chi';t)\in 1+tR[[t]].$$
If $R$ is a field, then $\chi$, $\chi'$ correspond to some cusp eigenforms $f$, $g$.
Therefore the properties of Rankin-Selberg $L$-functions give the following.

\begin{prop}\label{prop:rankin-selberg}
Let $R$ be a field, and $\chi\in\Sigma^{(r)}(R)$,  $\chi'\in\Sigma^{(s)}(R)$ as above. 

(a) The series $\LHom(\chi,\chi';t)$ 
represents a rational function in $t$ satisfying the functional equation
$$\gathered
\LHom(\chi',\chi; 1/qt)=\eps_{\chi,\chi'}\,(q^{1/2}t)^{2(1-g_X)rs}
\LHom(\chi,\chi'; t),\\
\eps_{\chi,\chi'}={}^\pi(\chi^*\boxtimes\chi')(\Omega^1_X).
\endgathered$$

(b) If  $r\neq s$,
then $\LHom(\chi,\chi';t)$
is a polynomial of degree $2(g_X-1)rs$.  
\end{prop}

\begin{proof} It is enough to assume that $R=\CC$. In this case the
proof  is a matter of comparing our conventions and notations
to those of the papers \cite{JPS:rankin-selberg} \cite{jacquet-shalika:rankin-selberg}
dealing with the general theory of
Rankin-Selberg convolutions. However, these papers
do not specifically emphasize the case of a function field and  rational,
rather than just meromorphic, nature of L-functions. A summary of this case,
based, among other sources, on the the preprint \cite{PS-maryland}, can
be found  in \cite{lafforgue}, Appendice B.

More precisely, Th\'eor\`eme B9 of \cite{lafforgue}
implies the rationality in part (a) as well as the general form of the functional
equation, with a monomial $\varepsilon(\chi, \chi'; t)$ determined by $\chi$
and $\chi'$. The identification of this monomial with 
$\eps_{\chi,\chi'}\,(q^{1/2}t)^{2(1-g_X)rs}$ is implied by local calculations
(Lemme B1 of \cite{lafforgue}) which realize $\varepsilon(\chi, \chi'; t)$
as a product of $rs$ factors, each determined by one ``Hecke eigenvalue"
for $\chi$ and $\chi'$. Each of these factors, in turn, is
identified with the local epsilon-factor of the Tate theory corresponding
to an unramified character of $\Aen^*$ and  some chosen
 additive character $\psi: \Aen/K\to \CC^*$. The identification 
 $\eps_{\chi, \chi'}={}^\pi(\chi^*\boxtimes\chi')(\Omega^1_X)$ is obtained,
 as in \cite{deligne},
by taking $\psi(a) = \psi_0\biggl(\sum_{x\in X} \tr_{\FF_{q_x}/\FF_q}\Res_x(a_x\omega)\biggr)$,
where $\omega$ is a rational section of $\Omega^1_X$, and $\psi_0:\FF_q\to \CC^*$
is a nontrivial character. Further, the fact that $\LHom(\chi,\chi'; t)$
is a polynomial for  $r\neq s$, follows from part (iii) of \cite{lafforgue},
Th\'eor\`eme B9. 
The exact value of the degree of this polynomial is found
from the functional equation in (a). 
\end{proof}

From Remark~\ref{rem:regularL}, we deduce

\begin{cor} The correspondence
$$(\chi,\chi')\mapsto \LHom(\chi,\chi'):=\LHom(\chi,\chi';1)$$
descends to a rational function on 
$\Sigma^{(r)}\times\Sigma^{(s)}$ defined over $k$.
\end{cor}

\noindent 
By combining these functions for all $r,s$, 
we get a rational function $\LHom\in k(\Sigma\times\Sigma)$.
The dependence on $t$ can be recovered using the $\GG_m$-action 
in either argument
$$\LHom(t^{-\deg}\chi,\chi')=\LHom(\chi,t^\deg\chi')=\LHom(\chi,\chi';t).$$

We introduce a rational function $c\in k(\Sigma\times\Sigma)$ by
\begin{equation}
\label{eq:rankin-selberg-c}
c(\chi,\chi')=q^{rs(1-g_X)}{\LHom(\chi,\chi')\over \LHom(q^\deg\chi,\chi')}.
\end{equation}
The functional equation for the $\LHom$-functions 
implies that $c$ is antisymmetric
$$c(\chi',\chi)=c(\chi,\chi')^{-1}.$$

\begin{df} A {\it theta characteristic} on $X$ is a line bundle $\Theta$ 
(defined over $\FF_q$) such that $\Theta^{\otimes 2}=\Omega^1_X$.
\end{df}

\noindent 
If the curve $X$ has a theta characteristic then
$c$ has a coboundary representation
\begin{equation}
\label{lambda}
c(\chi,\chi')=\lambda(\chi,\chi')\,\lambda(\chi',\chi)^{-1},
\quad
\lambda(\chi,\chi')=\theta_{\chi,\chi'}\LHom(\chi,\chi'),
\end{equation}
where $\theta_{\chi,\chi'}$ is the value of ${}^\pi(\chi^*\boxtimes\chi')$ 
on the class of $\Theta$ in $\Pic(X)$.

One can obtain other coboundary representations by multiplying $\lambda$
with rational functions on $\Sigma\times\Sigma$, symmetric under permutation.
We will use one such particular representation.
Note that the function
\begin{equation}
f(t) = t^{-1}(1-qt)(1-q^{-1}t)
\end{equation}
satisfies $f(t^{-1})=f(t)$. Define now a rational function $\tb\in k(\Sigma\times\Sigma)^*$
by
\begin{equation}\label{eq:tb-automorphic}
\tb(\chi, \chi') = 
\begin{cases}
1, \quad \text{if}\quad \chi, \chi'\text{ lie in different $\GG_m$-orbits};\cr
(\chi':\chi)\in\GG_m,\text{ if $\chi, \chi'$ lie in the same $\GG_m$-orbit}.
\end{cases}
\end{equation}
Put
\begin{equation}\label{eq:widetilde-lambda}
\widetilde\lambda(\chi, \chi') \,\,=\,\, f\bigl(\tb(\chi, \chi')\bigr) \cdot\lambda(\chi,\chi'). 
\end{equation}
Then
\[
c(\chi, \chi')\,\,=\,\,\widetilde\lambda(\chi,\chi')\widetilde\lambda(\chi',\chi)^{-1},
\]
but as we will see later, $\widetilde\lambda$ has better singularity behavior on $\Sigma\times\Sigma$. 

\vskip 3mm


\subsection{The main theorem.}
Recall the scheme
$\Sigma$
and the rational function $c\in k(\Sigma\times\Sigma)$
in (\ref{eq:rankin-selberg-c}).
We can now formulate our main result.

\begin{thm}\label{thm:main}
 The $k$-algebra $H$ is isomorphic to the shuffle $k$-algebra 
$Sh(\Sigma,c)$. If $X$ has a theta-characteristic, 
then $H$ is isomorphic to the symmetric shuffle 
$k$-algebra $SSh(\Sigma,\lambda)$.
\end{thm}

The proof will be given in Section \ref{sec:spectral-and-proof}.
We now reformulate Theorem \ref{thm:main} by using the Langlands correspondence
for $GL_r$ over function fields as established by L. Lafforgue 
\cite{lafforgue}. This
correspondence uses $l$-adic local systems, so we fix a prime $l$ not dividing $q$.

We denote by $\LS(X)$ the
category of lisse sheaves of
$\overline{\QQ}_l$-vector spaces on (the \'etale topology of) $X$, see \cite{milne}.
 We refer to objects of this category as {\em local systems} on $X$. 
 Let $\LS^{(r)}(X)\subset \LS(X)$ be the subcategory of 
 local systems of rank $r$.
For such a local system $\Lc$ and a point $x\in X$ we denote by
$\Fr(x, \Lc): \Lc_x\to\Lc_x$ the action of the Frobenius
$\Fr_x\in\Gal(\overline{\FF}_{q_x}/\FF_{q_x})$ on the stalk of $\Lc$ at $x$.
The L-function of $\Lc$ is defined by the product
\begin{equation}\label{eq:L-function}
L(\Lc,t)\,\,=\prod_{x\in X} {1\over\det(1-\Fr(x,\Lc)\cdot t^{\deg(x)})}
\,\,\in\,\,\overline{\QQ}_l[[t]].
\end{equation}
It is known that $L(\Lc, t)$ is a rational function satisfying the
functional equation 
\begin{equation}\label{E:functional-equation1}
 L(\Lc^\vee, 1/qt) \,=\,\epsilon_\Lc (q^{1/2}t)^{(2-2g_X)r} L(\Lc, t),
\end{equation}
where $\Lc^\vee$ is the dual local system, and
$$\epsilon_\Lc\cdot  q^{(1-g_X)r} \,\,=\,\,\det\bigl(\Fr: H^\bullet(X\otimes\overline\FF_q, \Lc)
\to H^\bullet(X\otimes\overline\FF_q, \Lc)\bigr).$$
It was shown by Deligne \cite{deligne} that 
$$\epsilon_\Lc \,=\,\prod_{x\in X} \det(\Fr(x,\Lc))^{c_x},$$
where $\sum c_x\cdot x$ is any divisor representing $\Omega^1_X\in\Pic(X)$. 
It is also known that if $\Lc$ is irreducible, then $L(\Lc, t)$
is a polynomial.

Let $D^{(r)}_0$ be the set of isomorphism classes of irreducible local systems
$\Lc\in\LS^{(r)}(X)$ such that $\det(\Lc)$ has finite order, i.e., 
$\det(\Lc)^{\otimes d}$ is a trivial local system for some $d\geqslant 1$.
Let also $\Sigma^{(r)}_0(\overline{\QQ}_l)\subset\Sigma^{(r)}(\overline\QQ_l)$
consist of those automorphic characters $\chi: A\to\overline\QQ_l$ for which
the corresponding character ${}^\pi\chi: \Pic(X)\to\overline\QQ_l^{\, \times}$ is
of finite order. Lafforgue's result is as follows (\cite{lafforgue}, Th\'eor\`eme VI.9).

\begin{thm}\label{thm:lafforgue}
 There exist bijections
$$\Sigma^{(r)}_0(\overline\QQ_l)\lra D^{(r)}_0, \quad \chi\mapsto \Lc_\chi$$
for each $r\geqslant 1$ such that for any $\chi\in \Sigma^{(r)}_0(\overline\QQ_l)$,
$\chi'\in\Sigma^{(s)}_0(\overline\QQ_l)$ we have
$$\LHom (\chi', \chi, t)\,\,=\,\,L(\underline{\Hom}(\Lc_{\chi'}, \Lc_{\chi}), t).$$\qed
\end{thm}

In particular, by taking $\chi'$ to be the point $\1$ of the Witt scheme, we have
$L(\chi,t)\,=\,L(\Lc_{\chi}, t)$
for each $r$ and each $\chi\in\Sigma^{(r)}_0(\overline\QQ_l)$. 

\begin{prop}\label{prop:D-0-intersects}
The subset $\Sigma^{(r)}_0(\overline\QQ_l)\subset \Sigma^{(r)}(\overline\QQ_l)$
intersects any orbit of the $\GG_m$-action.
\end{prop}

\begin{proof}
Indeed, fix a line bundle $L$ on $X$ of
non-zero degree. Then $\Pic(X)/L^\ZZ$ is a finite abelian group. Therefore
a homomorphism $\phi: \Pic(X)\to\overline\QQ_l^{\, \times}$ has a finite order if and 
only if $\phi(L)$ is a root of unity. So given $\chi\in \Sigma^{(r)}(\overline\QQ_l)$,
taking $\lambda = {}\pi\chi(L)^{-1/deg(L)}$, we have $(\lambda^{\deg}\cdot \chi)^\pi(L) =1$,
and therefore $\lambda^\deg\cdot\chi\in \Sigma^{(r)}_0(\overline\QQ_l)$. \end{proof}

\vskip .2cm

We now extend $D_0^{(r)}$ to a variety over $\overline\QQ_l$ similar to $\Sigma^{(r)}$,
which is a disjoint union of $\GG_m$-orbits. For this, choose a point
$x\in X$ and recall that $\LS^{(r)}(X)$
is equivalent to the category of continuous representations $\pi_1^{\et}(X,x)
\to GL_r(\overline\QQ_l)$, where $\pi^{\et}_1(X,x)$ is the \'etale fundamental group 
 of $X$ with base point $x$.We use the same notation for the corresponding 
 objects of the two categories. Recall further that we have a surjective 
 homomorphism
$$\pi^{\et}_1(X,x)\buildrel d\over \lra \Gal(\overline\FF_q/\FF_q)\,=\,\widehat\ZZ.$$
and the {\em unramified Weil group} of $X$ is defined to be
$$\pi_1^{\Weil}(X,x)\,\,=\,\,d^{-1}(\ZZ)\, \subset 
 \,
\pi_1^{\et}(X,x).$$
  Note that for each $x\in X$ we have a well-defined conjugacy class
$\{\Fr_x\}\subset\pi_1^{\Weil}(X,x)$, with $d(\{\Fr_x\})= \deg(x)$. So for
any $r$-dimensional representation $V$ of $\pi_1^{\Weil}(X,x)$ over
$\overline\QQ_l$ we define the L-series
$L(V, t)$ in the same way as in \eqref{eq:L-function}. 

Unlike local systems, representations of $\pi_1^{\Weil}(X,x)$ admit
{\em fractional Tate twists}. That is, for each such representation 
$(V, \rho_V)$ with
$\rho_V:\pi_1^{\Weil}(X,x)\to GL(V)$ 
and each $\lambda\in\overline\QQ_l^{\, \times}$ we have a new representation
$V(\lambda)$,
identified with $V$ as a vector space,  with $\rho_{V(\lambda)}$ taking $g$ into $\rho_V(g)\cdot\lambda^{d(g)}$.
It follows that
\begin{equation}\label{eq:L-twist}
L(V(\lambda), t) = L(V, \lambda t).
\end{equation}
Let now $D^{(r)}$ be the set of isomorphism classes of irreducible
$r$-dimensional representations of $\pi_1^{\Weil}(X,x)$ over $\overline\QQ_l$
which are of the form $V(\lambda)$ where $V\in D_0^{(r)}$ and
$\lambda\in\overline\QQ_l^{\,\times}$. It follows from 
Theorem \ref{thm:lafforgue} and Proposition \ref{prop:D-0-intersects}
that $D^{(r)}$ is a finite union of orbits of the 
$\overline\QQ_l^{\,\times}$-action
by fractional Tate twists. We consider $D^{(r)}$ as an algebraic
variety over $\overline\QQ_l$, the corresponding disjoint union of copies of $\GG_m$.
The properties of L-functions of \'etale local systems together with 
\eqref{eq:L-twist} imply that the correspondence
$$(V,W)\,\mapsto\, L(\underline{\Hom}(V,W))\,:=\,
L(\underline{\Hom}(V,W),1)$$
define a rational functon $L\underline{\Hom}$ on 
$D^{(r)}\times D^{(s)}$ for any $r,s\geqslant 1$.
 
We set $D=\coprod_{r\geqslant 1} D^{(r)}$. This is a scheme of the type 
considered in Section \ref{sec:disjoint}, and we have a rational function
$L\underline{\Hom}\in \overline\QQ_l(D\times D)^\times$. 

\begin{rem}\label{rem:on-singularities-LHom}
The function $L\underline{\Hom}$ is regular everywhere except for the
first order poles on the diagonal and the $q$-shifted diagonal
\[
\Delta \,=\,\{(\chi, \chi)\}, \quad \Delta_q\,=\,\{ (\chi, q^{-\deg}\cdot\chi)\} \quad\subset\quad D\times D.
\]
This follows from the fact that $L(\Lc,t)$ is a polynomial for an irreducible local system $\Lc$
not of the form $\overline{\underline\QQ_l}(\lambda)$, while 
\[
L(\overline{\underline\QQ_l}, t) \,\,=\,\,\zeta_X(t)\,\,=\,\,{ P(t)\over (1-t)(1-qt)}, 
\]
with $P(t)$ a polynomial of degree $2g_X$.
\end{rem}

Lafforgue's theorem
can be reformulated as follows.

\begin{thm} Let $k$ be a subfield of $\overline\QQ_l$
containing $\sqrt{q}$. There is an isomorphism 
$\Sigma\otimes_k\overline\QQ_l\to D$
of schemes over $\overline\QQ_l$
commuting with the $\GG_m$-action and sending the rational function
$\LHom\in  k(\Sigma\times\Sigma)$ to $L\underline\Hom$.  \qed
\end{thm}

Defining a rational function
$$\delta\in\overline\QQ_l(D\times D), \quad \delta(V,W) \,=\,
q^{(1-g_X)rs}{ L\underline\Hom(V,W)\over L\underline\Hom(V(q),W)},$$
we see that 
\begin{equation}\label{E:functional-equation2}
 \delta(V,W) =\delta(W, V)^{-1}.
\end{equation}
If $X$ has a theta-characteristic
$\Theta$, then
$$\delta(V,W) = \mu(V,W) \mu(W, V)^{-1},$$
where
$$\mu(V,W) \,\,=\,\,q^{(1-g_X)rs} \eta_{\underline\Hom(V,W)}(\Theta)
L\Hom(V,W)$$
and $\eta_{\underline\Hom(V,W)}$ is the character of $\Pic(X)$ corresponding to
the determinant of the local system $\underline\Hom(V,W)$. 
 Define further the rational function $\tb\in k(D\times D)$
simularly to \eqref{eq:tb-automorphic} by
\begin{equation}\label{eq:tb-galois}
\tb(V, W) = 
\begin{cases}
1, \quad \text{if}\quad V, W\text{ lie in different $\GG_m$-orbits};\cr
\lambda \in\GG_m,\text{ if $W=V(\lambda) $ lie in the same $\GG_m$-orbit},
\end{cases}
\end{equation} and put
\[
\widetilde\mu(V,W) \,\,=\,\, f\bigl(\tb(V,W)\bigr)\cdot\mu(V,W).
\]
By Remark \ref{rem:on-singularities-LHom}, the function $\widetilde\mu$ is regular on $D\times D$
except for first order pole on each component of the diagonal. 

Lafforgue's theorem implies a purely Galois-theoretical realization of the
Hall algebra of vector bundles.

\begin{cor}\label{cor-of-lafforgue}
Let $k=\overline\QQ_l$. The Hall algebra $H$ 
is isomorphic to the shuffle algebra
$Sh_{\overline\QQ_l}(D,\delta)$. If $X$ has a theta-characteristic, then
$H$ is isomorphic to the symmetric shuffle algebra 
${SSh}_{\overline\QQ_l} (D,\widetilde\mu)$.\qed
\end{cor}

Transferring the information about singularities of $L\underline\Hom$
back into the automorphic situation (where the purely automorphic Proposition 
\ref{prop:rankin-selberg} gives a weaker statement), we conclude that:

\begin{cor}\label{cor:regularity-of-embedding}
 Suppose that $X$ has a theta-characteristic. 
The function $\widetilde\lambda\in k(\Sigma\times\Sigma)^*$ defined in
\eqref{eq:widetilde-lambda}, is regular on $\Sigma\times\Sigma$ except for
the first order pole at each component of the diagonal. Therefore, by
Theorem \ref{thm:main} and 
Proposition \ref{prop:symm-shuffle-regular}, we have an embedding
\[
H \,\,\simeq\,\, SSh(\Sigma,\widetilde\lambda) \,\,\subset\,\,\bigoplus_{n\geq 0} 
k_\qcom[\Sigma^n]^{S_n}
\]
 of $H$ into the space of regular symmetric functions on $\Sigma$. The image of $H$
 is the direct sum, over $n\geq 0$, of ideals in the rings $k_\qcom[\Sigma^n]^{S_n}$.
\end{cor}

\begin{rem} The orbifold scalar product on $H$ corresponds, in this embedding,
to the scalar product of pseudo-Eisenstein series which is classically found by
the ``Maass-Selberg relations", cf. \cite[\S II.2.1]{moeglin-waldspurger-book}. In our language
 they correspond to the $L_2$-scalar
product on $\bigoplus_{n\geq 0} 
k_\qcom[\Sigma^n]^{S_n}$,
corresponding to a  measure on some real locus of $\Sym(\Sigma)$
with the weight of the measure found as a product of the $\LHom$-functions, similarly to
the scalar products on symmetric polynomials considered in
\cite{macdonald2}. On the other hand, $H$ has a natural orthogonal basis,
formed by vector bundles themselves. This means that we 
represent vector bundles by orthogonal polynomials.  
  
\end{rem}

\vskip .3mm


\subsection{Example : elliptic curves.}
In this subsection we assume that $X$ is an elliptic curve, i.e., $g_X=1$.
In this case the Hall algebra was studied extensively, see 
\cite{burban-schiffmann:elliptic-I, schiffmann:elliptic-II, schiffmann:drinfeld,
schiffmann-vasserot:macdonald, schiffmann-vasserot:hilbert}. 
These papers emphasized, in particular, the so-called {\em spherical subalgebra}
$H_{\operatorname{sph}}\subset H$ generated by the characteristic functions
of $\Pic^d(X)$, $d\in\ZZ$.
Our results imply that $H$
splits into an infinite tensor product of 
simpler pairwise commuting $k$-algebras, one of which is 
$H_{\operatorname{sph}}$. This is based on the following fact.

\begin{prop}\label{prop:L=1}
Let $\Lc$ be an irreducible local system of $\bar\QQ_l$-vector
spaces on $X$ such that
$\Lc(\l)$ is nontrivial for any $\lambda\in\bar\QQ_l$.
Then $L(\Lc,t)=1$ identically.
\end{prop}

\begin{proof} Denote $\overline X = X\otimes \overline\FF_q$.
The following is a direct consequence of the isomorphism 
$\pi_1^{\et}(\overline X)=\widehat\ZZ\oplus\widehat\ZZ$.

\begin{lemma}\label{lem:H=0}
Let $\Nc$ be an irreducible local system of 
$\bar\QQ_l$-vector spaces on $\overline X$ which is nontrivial.
Then $\rk(\Nc)=1$ and $H^i(\overline X,\Nc)=0$ for all $i$. \qed
\end{lemma}

The lemma implies, in particular, that $\Ext^1(\Nc, \Nc')=0$, 
if $\Nc, \Nc'$ are two
non-isomorphic irreducible local systems on $\overline X$.
As before, let $\Lc$ be irreducible, and let $r=\rk(\Lc)$. 
Denoting by $\overline\Lc$ the local
system on $\overline X$ pulled back from $\Lc$, we then conclude
that
\begin{equation}\label{eq:decomposition-bar-L}
\overline\Lc\,\,\simeq \,\,\bigoplus_{i=0}^{r-1} (\Fr^i)^*\Nc,
\end{equation}
where $\Nc$ is a 1-dimensional local system on $\overline X$ such that
$(\Fr^r)^*\Nc\simeq\Nc$, and $r$ is the minimal number with this property.
If $r>1$, this implies that $H^i(\overline X,\overline\Lc)=0$ for all $i$,
and so $L(\Lc,t)=1$ by the cohomological interpretation of L-functions.
If $r=1$, this means that $\overline\Lc=\Nc$ is such that $\Fr^*(\Nc)\simeq\Nc$.
If $\Nc$ is nontrivial, we conclude as before. If $\Nc = \underline{\overline\QQ_l}$
is trivial, then the descent of $\Nc$ to a local system $\Lc$ on $X$ is given by
a homomorphism $\phi:\widehat\ZZ=\Gal(\overline\FF_q/\FF_q)\to
\overline\QQ_l^{\, \times}$.
If $\lambda$ is the image of 1 under $\phi$, then 
$\Lc =  \underline{\overline\QQ_l}(\lambda)$ is a Tate twist of the trivial local system.
Proposition \ref{prop:L=1} is proved. 
\end{proof}

\begin{cor}\label{cor:LHom(L,M)}
Let $\Lc,\Mc$ be two irreducible local systems on $X$. If $\Mc$ is
not isomorphic to $\Lc(\lambda)$ for any $\lambda\in\overline\QQ_l^{\, \times}$,
then
$L(\underline{\Hom}(\Lc,\Mc),t)=1$ identically. If $\Mc\simeq \Lc(\l)$, then
$$L\bigl(\underline\Hom(\Lc,\Mc),t\bigr) \,\,=\,\,
\prod_{\epsilon\in\sqrt[n]{1}} \zeta_X(\epsilon \lambda t), \quad r=\rk(\Lc).
$$
\end{cor}

\begin{proof} The first statement follows by decomposing $\overline\Lc$ and
$\overline \Mc$ and in \eqref{eq:decomposition-bar-L}, and then applying
Lemma \ref{lem:H=0}. To prove the second statement, it is enough to assume
$\l=1$, i.e., $\Mc=\Lc$. Then, in the notation of 
\eqref{eq:decomposition-bar-L}, we have
\begin{equation}\label{eq:decomposition-End-L}
\underline\Hom(\Lc,\Lc) \,\,\simeq\,\,  \overline\Pc\,\oplus\, \bigoplus_{i=0}^{r-1}
 \underline\End((\Fr^i)^*\Nc) ,
\end{equation}
where $\overline\Pc$ is a local system without any cohomology which therefore
does not contribute to the L-function. Now, each summand $\underline\End((\Fr^i)^*\Nc)$
is a trivial 1-dimensional local system, but these summands are permuted by
the Frobenius in a cyclic way. So, as a module over
$\Gal(\overline\FF_q/\FF_q)=\widehat\ZZ$, the second summand in
\eqref{eq:decomposition-End-L} is identified with the tensor product
of the trivial local system $\underline{\overline\QQ_l}$ and
of  $\overline\QQ_l[\ZZ/r]$, the group algebra of the cyclic group $\ZZ/r$. 
Denoting by $\overline\QQ_l(\epsilon)$ the 1-dimensional representation
of $\widehat\ZZ$, in which $1$ acts by $\epsilon$, we have an isomorphism
of $\overline\QQ_l[\widehat\ZZ]$-modules
$$\overline\QQ_l[\ZZ/r]\,\,=\,\,\bigoplus_{\epsilon\in\sqrt[r]{1}}
\overline\QQ_l(\epsilon),$$
and therefore we conclude that
$$\underline\Hom(\Lc,\Lc) \,\,\simeq\,\,\Pc \,\oplus\,
\bigoplus_{\epsilon\in\sqrt[r]{1}} \underline{\overline\QQ_l}(\epsilon),$$
where $\underline{\overline\QQ_l}(\epsilon)$ is the Tate twist of the trivial
local system by $\epsilon$, and $L(\Pc,t)=1$. This implies the  corollary.
\end{proof}
 
\vskip .2cm
 
For $r\geqslant 1$ we introduce the rational function
of two variables $t,s\in\GG_m$ 
$$c_{X,r}(t,s)\,\,=\,\,\prod_{\epsilon\in\sqrt[r]{1}}
{\zeta_X(\epsilon t/s)\over \zeta_X(\epsilon qt/s)}.
$$
 
\begin{thm}\label{thm:tensor-product}
 Let $k$ be an algebraically closed field of characteristic 0.
Then the algebra $H$ is isomorphic to the infinite tensor product
$$H\,\,\simeq\,\,\bigotimes_{S\in\pi_0(\Sigma)}Sh(\GG_m, c_{X,r(S)}),$$
where $S$ runs over connected components of $\Sigma$ and $r(S)$
is the rank of (the cusp forms corresponding to points of)  $S$. 
\end{thm}
  
\begin{proof}
It is enough to take $k=\overline\QQ_l$. By Corollary \ref{cor-of-lafforgue},
we have that $H=Sh(D,\delta)$, with $\pi_0(D)$ identified with
$\pi_0(\Sigma)$. Further, if $S, S'$ are different components of $D$, then
by Corollary \ref{cor:LHom(L,M)}, the restriction of $\delta$ to $S\times S'$
is identically equal to 1. The restriction of $\delta$ to each $S\times S$
is identified with the function $c_{X,r(S)}$ on $\GG_m\times\GG_m$,
again  by  Corollary \ref{cor:LHom(L,M)}. This implies our statement.
\end{proof}

\begin{ex} If $r(s)=1$, then $Sh(\GG_m, c_{X,r(S)})$ is isomorphic to
the spherical Hall algebra $H_{\operatorname{sph}}$ of 
\cite{burban-schiffmann:elliptic-I}, as shown in
\cite{schiffmann-vasserot:macdonald}. The components of
$\pi_0(\Sigma)$ with $r(S)=1$ correspond to characters of $\Pic(X)$
modulo tensoring with characters of the form $\lambda^\deg$. So their
number is equal to  $|\Pic^0(X)|$.
The tensor product
decomposition of Theorem \ref{thm:tensor-product} contains
therefore $|\Pic^0(X)|$ commuting copies of $H_{\operatorname{sph}}$.
The subalgebra generated by these copies is a particular case of the 
{\em principal Hall 
algebra} defined in \cite{schiffmann-vasserot:highergenus}
for curves of any genus. 

 \end{ex}

\vfill\eject

\section{The spectral decomposition of $H$ and the proof of Theorem
\ref{thm:main}.}\label{sec:spectral-and-proof}

\subsection{Spectral decomposition in geometric terms.}
\label{sec:spectral}

Let us recall the relation between Hall products and
Eisenstein series observed in \cite{K}.  Let $f\in\Sigma^{(r)}(k)$
be a cusp eigenform. By Lemma \ref{lem:cusp-forms-module-support}(b),
the restriction of $f$ to each $\Bunn_{r,d}(X)$ has finite support.
Let
\begin{equation}\label{eq:elements-E-f-d}
E_{f,d}\,\,=\,\,\sum_{V\in\Bunn_{r,d}(X)} f(V)\cdot 1_V\,\,\,\in\,\, H^{(r,d)}
\end{equation}
be the corresponding element of $H$, and
\begin{equation}\label{eq:series-E-f-t}
E_f(t)\,\,=\,\,\sum_{d\in\ZZ} E_{f,d} t^d\,\,\,\in\,\,\,H^{(r)}[[t, t^{-1}]]
\end{equation}
be the formal generating series of the $E_{f,d}$. Let ${\bf f}=(f_1, ..., f_m)$,
$f_i\in\Sigma^{(r_i)}(k)$ be a sequence of cusp eigenforms, and $r=\sum r_i$.
The Hall product
\begin{equation}\label{eq:eisenstein-product}
E_{\bf f}({\bf t}) \,\,=\,\,E_{f_1}(t_1)* \cdots * E_{f_m}(t_m)\,\,\in\,\,
H^{(r)}[[t_1^{\pm 1}, ..., t_m^{\pm 1}]]
\end{equation}
considered as a function on $\Bunn_r(X)$,
  is the Eisenstein series corresponding to
 $f_1, ..., f_m$. In particular,  its value at any rank $r$
 bundle $V$,  is a series in the $t_i^{\pm 1}$. It is known that this series is the expansion of
 a rational function in the region $|t_1|\gg \cdots\gg |t_m|$,
 and that this function (or rather the collection of these functions for all $V$) satisfies a functional equation (recalled in
 \ref{functional-equation-eisenstein} below). If $k=\CC$, by a {\em pseudo-Eisenstein series}
 one  means a contour integral of the form
  $$\int_{|t_i|=\eps_i } E_{\bf f}({\bf t})\phi({\bf t}) d^*{\bf t}$$
 where $\phi$ is a Laurent polynomial and $|\eps_1|\gg \cdots\gg |\eps_m|$,
 see \cite{moeglin-waldspurger-book}, \S II.1.10-11. Thus pseudo-Eisenstein series
 are finite linear combinations of the products $E_{f_1, d_1}*\cdots * E_{f_m, d_m}$. 

\vskip .2cm

Now, the algebraic version of the spectral decomposition theorem for
unramified automorphic forms on $GL_n(\Aen)$ can be formulated
as follows.

\begin{prop}\label{prop:spectral-decomposition}
The subspace $H_{\cusp}\subset H$ generates $H$
as an algebra.
\end{prop}

\begin{proof} This was shown in \cite{K}, Thm. 3.8.4. For convenience of
the reader, we provide the argument here. It is enough to assume that $k=\CC$. Let
$H' \subseteq H$ be the subalgebra generated by $H_{\cusp}$. We will prove by induction on the rank $r$
that ${H'}^{(r)}=H^{(r)}$. This is obvious for $r=1$ since $H^{(1)}_{\cusp}=H^{(1)}$. Let $r >1$ and let us assume that
${H'}^{(s)}=H^{(s)}$ for any $s <r$. Let $d \in \ZZ$ and let $x \in H^{(r,d)}$ be orthogonal to $H'$ with respect
to the Hermitian orbifold pairing (\ref{hermitianproduct}). 
For any $(r_1,d_1),$ $(r_2,d_2)$ with $r_1,r_2 \geqslant 1$ such that 
$r=r_1+r_2,$ $d=d_1+d_2$
we have $$\Delta_{(r_1,d_1),(r_2,d_2)}(x) \in ({H'}^{(r_1)} 
\otimes {H'}^{(r_2)})^{\perp}$$ by the Hopf property of the
pairing. Since ${H'}^{(r_1)}=H^{(r_1)},$ ${H'}^{(r_2)}=H^{(r_2)}$, 
and since the pairing is nondegenerate, it follows that $\Delta_{r_1,r_2}(x)=0$
for all $r_1, r_2 \geqslant 1$. But then $x \in H^{(r)}_{\cusp} \subset H'$. Since the pairing is positive definite we conclude that $x=0$.
We have shown that $({H'}^{(r)})^{\perp}=\{0\}$. This implies that ${H'}^{(r)}$ is dense in $H^{(r)}$ as $H^{(r)}$, equipped with
the pairing (\ref{hermitianproduct}), is a pre-Hilbert space.
This means that for any finite set $F\subset \Bunn_{r,d}(X)$ the coordinate projection
 ${H'}^{(r)}\to\CC^F$ is surjective. The reduction theory, see \cite{stuhler}, shows that for every $\lambda\in\ZZ_+$
we have a finite subset $F_\lambda\subset\Bunn_{r,d}(X)$ formed by bundles $V$
whose Harder-Narasimhan filtration has quotients with slopes decreasing,
at each step, 
by at most $\lambda$. So we can represent any given $f\in H^{(r,d)}$
as $f=\psi_\lambda + \chi_\lambda$, where $\psi_\lambda\in {H'}^{(r)}$,
and $\chi_\lambda$ is supported outside $F_\lambda$. Finally, for large enough
$\lambda$ (depending on the genus $g$ of $X$), a layer $V'$ of the Harder-Narasimhan filtration
of $V\in\Bunn_{r,d}(X)$ with slope decreasing by more than $\lambda$, gives
a short exact sequence 
$$0\to V'\lra V\lra V''\to 0$$
with $\Ext^1(V'',V')=0$. This implies that $V\simeq V'\oplus V''$
and, moreover, $1_V$ is a multiple of $1_{V'}*1_{V''}$.
This means we have $\chi_\lambda\in {H'}^{(r)}$
by induction on the rank of the bundles. 
\end{proof}

\vskip .2cm

Recall from Section \ref{rankin-selberg} the scheme $\Sigma = \bigsqcup_r \Sigma^{(r)}$, with
$$\Sigma^{(r)}\, =\, \Supp_{A^{(r)}}(H^{(r)}_{\cusp})\,\subset\, W_X^r\,\subset\, W_X,$$
and the
monomorphism $\alpha: \Sigma\to W_X$ formed by closed embeddings
of  $\Sigma^{(r)}$ into $W_X^r$.
Next, recall that the coproduct on $A$ yields an abelian group structure
$\boxplus$ on $W_X=\Spec(A)$.  Iterated addition in
this group structure, together
with the morphism  $\alpha$, yields  morphisms
\begin{equation}\label{eq:alpha-r-i}
\alpha_{r_1, ..., r_n}: 
\Sigma^{(r_1)}\times ... \times\Sigma^{(r_n)}
\lra W^{r}_X, \quad r=r_1+\cdots r_n,
\end{equation}
which gives rise to a morphism 
\begin{equation}
\mathfrak{a}: \, \Sym(\Sigma)\lra W_X. 
\end{equation}
Let $\Pi$ be the set of maps $\nu: \ZZ_{\geqslant 1}\to\ZZ_{\geqslant 0}$
such that $\nu(s)=0$ for $s\gg 0$. 
For $\nu\in\Pi$ we denote
\begin{equation}
|\nu| \,=\,\sum_{s\geqslant 1} \nu(s), \quad 
\wt(\nu) \,=\, \sum_{s\geqslant 1} s\nu(s),\quad 
\Sym^\nu(\Sigma):=\prod_{s\geqslant 1}\Sym^{\nu(s)}(\Sigma^{(s)}).
\end{equation}
 Then for $r\geqslant 0$
we have the disjoint union decomposition 
$$
\Sym^m(\Sigma)\,\,=\,\,
\coprod_{|\nu|=m}\Sym^\nu(\Sigma),$$
and the restriction of $\aen$ to $\Sym^\nu(\Sigma)$ defines
a morphism 
$$\aen_\nu: \Sym^\nu(\Sigma)\to W_X^{\wt(\nu)}. $$

\begin{prop}\label{prop:embedding-sym-sigma}
The morphism $\aen$ is injective on $\overline k$-points.
\end{prop}

\begin{proof}  
It is enough to assume that
$k=\CC$. Our statement means that any finite set of 
$\CC$-points
$\sigma_i\in \Sigma^{(r_i)}\subset W_X$ is linearly independent
over $\ZZ$ in the group $W_X(\CC)$. But such independence follows from
the result of  Jacquet and Shalika 
(\cite{jacquet-shalika:rankin-selberg}, Thm. 4.2). 
\end{proof}

 \vskip .3cm
 
Next, we prove that $\aen_\nu$ is a closed morphism, i.e.,
it takes Zariski closed sets into Zariski closed sets.
 Indeed, each $\aen_\nu$ is obtained by descent  of some
$p_{r_1, ..., r_n}$ under the finite morphism
$\prod_i \Sigma^{(r_i)}\to \Sym^\nu(\Sigma)$. So it is
enough to show that each $p_{i_1, ..., i_r}$ is closed.
Let $\Delta_{r_1,\dots,r_n}$ denote the degree $(r_1,\dots,r_n)$ component
$A^{(r)}\to \bigotimes_i A^{(r_i)}$
of the iterated comultiplication map
$\Delta^{(n-1)}.$

\begin{prop}\label{prop:integral}
For any $r_1, ..., r_n$ the Witt addition morphism
$$\operatorname {wadd}_X: \prod_i W_X^{r_i}\lra W_X^{r}, \quad r=r_1+...+r_n,$$
is an integral morphism, i.e., the map
$\Delta_{r_1,\dots,r_n}$
makes $\bigotimes_i A^{(r_i)}$ into an integral ring extension of $A^{(r)}$.
\end{prop}

The proposition implies that $\aen_\nu$ is closed, since any integral morphism
is closed, see e.g., \cite[Proposition 6.1.10]{EGAII}, and
since $\prod_i \Sigma^{(r_i)}$
is a closed subset in the source of $\operatorname{wadd}_X$.

\begin{proof}[Proof of the proposition] Note that $\operatorname{wadd}_X$  
is the Cartesian product, over all 
$x\in X$, of similar morphisms 
$\operatorname {wadd}_x:\,\, \prod_i W_x^{r_i}\to W_x^{r}$.
Each $W_x^{r_i}$ is the scheme of polynomials with
constant term 1 and degree exactly $r_i$, and $\operatorname {wadd}_x$
is given by multiplication of polynomials. So each 
$\operatorname {wadd}_x$ is a finite morphism, and 
$\operatorname {wadd}_X$, being a projective limit of finite
morphisms, is integral.  
\end{proof}
\vskip .2cm

The fact that each $\aen_\nu$ is closed together with 
Proposition \ref{prop:embedding-sym-sigma}
means that
for any component $S\subset \Sym(\Sigma)$ the restriction
$\aen|_S: S\to \aen(S)$ is a regular birational morphism,  bijective on $\overline k$-points.
In particular, the morphism $\aen_\nu$ is a regular birational homomorphism from
$\Sym^\nu(\Sigma)$ to a closed subscheme  $\frak S^\nu$ in $W^{\wt(\nu)}_X$.
We write $$\frak S^r\,\,=\,\,
\coprod_{|\nu|=r}\frak S^\nu,\qquad \frak S\,\,=\,\,
\coprod_{r\geqslant 0}\frak S^r.$$

First, let us quote a few simple properties of the scheme $\frak S$ (the proof is left to the reader).

\begin{prop} We have

(a) $\frak S^1=\Sigma$,

(b) the subscheme $\frak S\subset W_X$ is a sub-semigroup of $W_X,$

(c) the map $\aen$ factors to a morphism of semigroups $\Sym(\Sigma)\to\frak S.$
\end{prop}

Next, we prove the following.

\begin{prop}\label{prop:inclusion-supp-H}
We have an equality of schemes
$$\Supp_{A^{(r)}}(H^{(r)})=\frak S^r.$$
\end{prop}

\begin{proof} 
By Proposition \ref{prop:spectral-decomposition}
it is enough to show
that for $r_1,\dots, r_n$ as in \eqref{eq:alpha-r-i} we have
\begin{equation}\label{eq:support1}
\Supp_{A^{(r)}}\Bigl( H_\cusp^{(r_1)}* \cdots * H_\cusp^{(r_n)}\Bigr)
\,\,\,=\,\,\, \alpha_{r_1, ..., r_n}\Bigl(\Sigma^{(r_1)}\times \cdots
\times\Sigma^{(r_n)}\Bigr).
\end{equation}
By  Proposition \ref{prop:hall-mult-hecke-comult},
the Hall multiplication gives 
 a surjective $A^{(r)}$-module homomorphism
$$H_\cusp^{(r_1)}\otimes \cdots \otimes H_\cusp^{(r_n)}\lra
H_\cusp^{(r_1)}* \cdots * H_\cusp^{(r_n)},$$
where the source is made into a $A^{(r)}$-module
via the map $\Delta_{r_1,\dots,r_n}:A^{(r)}\to \bigotimes_i A^{(r_i)}$
in Proposition \ref{prop:integral} and the target is
regarded as a $A^{(r)}$-submodule of $H^{(r)}$.
Thus we have
\begin{equation}\label{eq:support2}
\begin{gathered}
\Supp_{A^{(r)}}\Bigl( H_\cusp^{(r_1)}* \cdots * H_\cusp^{(r_n)}\Bigr)
\,\,\subset \,\,
\Supp_{A^{(r)}}\Bigl( H_\cusp^{(r_1)}\otimes \cdots \otimes H_\cusp^{(r_n)}\Bigr)\,\,=\cr
=\,\, {\operatorname{wadd}}_X\biggl(\Supp_{A^{(r_1)}\otimes \cdots A^{(r_n)}}\Bigl( H_\cusp^{(r_1)}\otimes \cdots \otimes H_\cusp^{(r_n)}\Bigr)
\biggr),
\end{gathered}
\end{equation}
the last equality following from the fact that the ring
embedding  
$\Delta_{r_1,\dots,r_n}$ corresponds to   the integral
morphism of schemes
$\operatorname {wadd}_X$ from Proposition  \ref{prop:integral}.
Since the right hand sides in \eqref{eq:support1} and
\eqref{eq:support2} coincide,
this proves  the inclusion $\subset$ in \eqref{eq:support1}.

We now prove the reverse inclusion in  \eqref{eq:support1}.
For this we can assume that $k=\CC$.
For a cusp eigenform $f\in\Sigma^{(r)}(\CC)$
let $H_f\subset H_\cusp$ be the span of the $E_{f,d}$, $d\in\ZZ$. 
It is enough to prove that
for any sequence of cusp eigenforms $f_i\in\Sigma^{(r_i)}(\CC)$, $i=1, ..., n$,
we have the exact equality
\begin{equation}\label{eq:support3}
\supp_{A}\bigl(H_{f_1}*\cdots * H_{f_n}\bigr)\,\,\,=\,\,\,
\alpha_{r_1, \cdots r_n}(S_1\times \cdots \times S_n),
\end{equation}
where $S_i$ be the $\GG_m$-orbit of $f_i$ in $W_X$.  
In other words, we have to prove that for any $a\in A$ not
vanishing identically on the RHS of \eqref{eq:support3},
the Hecke operator $T_a$ cannot annihilate each
$E_{f_1, d_1}*\cdots * E_{f_n,d_n}$, i.e., cannot annihilate the
series $E_{f_1}(t_1)*\cdots *E_{f_n}(t_n)$ identically. 
On the other hand, denoting by $\chi_i\in W_X(\CC)$ the character
of $A$ corresponding to the eigenform $f_i$, we see from 
Proposition \ref{prop:hall-mult-hecke-comult} that 
\begin{equation}\label{eq:support4}
T_a \bigl(E_{f_1}(t_1)*\cdots * E_{f_n}(t_n)\bigr)\,\,=\,\,\phi_a(t_1, \ldots, t_n)
E_{f_1}(t_1)*\cdots * E_{f_n}(t_n),
\end{equation}
where the Laurent polynomial $\phi_a(t_1, \ldots, t_n)$ is defined 
as the value of $a$ at the $\CC[t_1^{\pm 1}, \cdots t_n^{\pm 1}]$-point
$(\chi_1t_1^{-\deg})\boxplus \cdots \boxplus (\chi_nt_n^{-\deg})$ of $W_X$. 
Our assumption on $a$ means that $\phi_a$ is not identically zero.
Now, identical vanishing of \eqref{eq:support4} is not
possible because the Eisenstein series represents a nonzero
rational function in $t_1, \ldots, t_n$. Proposition \ref{prop:inclusion-supp-H}
is proved.
\end{proof}

\vskip .2cm

Let $\Hc, \Hc^{(r)}$ be the quasicoherent sheaves on $W_X$ corresponding to
the $A$-modules $H, H^{(r)}$, so that $\Hc=\bigoplus_{r\geqslant 0}\Hc^{(r)}$. 
The above proposition implies that each $\Hc^{(r)}$ is supported on
a finite union of (Zariski closed) components of $\frak S$. Therefore
we can regard $\Hc^{(r)}$ and $\Hc$ as quasicoherent sheaves on
$\frak S$.

 For any connected component $S\subset\frak S$ let $\Hc_S$ be
 the induced sheaf on $S$, and $H(S):=\Gamma(S, \Hc_S)$. 
 Note that Propositions \ref{prop:alg-mult-one} and  
 \ref{prop:spectral-decomposition} imply that $\Hc_S$ is a coherent,
 not just  quasicoherent, sheaf on the algebraic variety $S$. 
 By definition, we have an isomorphism of vector spaces
\begin{equation} \label{eq:component-decomposition-H}
H \,\,=\,\,\Gamma_\qcom(\frak S, \Hc)\,\, =\,\,  \bigoplus_{S\in\pi_0(\frak S)} H(S),
\end{equation}
while $H^{(r)}$ is a similar, finite, sum taken over $S\subset\frak S^r$. 

\begin{prop}\label{prop:decompsupport}
 The decomposition \eqref{eq:component-decomposition-H}
 is a multiplicative grading, i.e.,
 $$\label{eq:S-grading-of-H}
 H(S)*H(S') \,\subset H(p(S\times S'))
 $$
 where $p: \frak S\times\frak S
 \to\frak S$ is the addition morphism. 
 \end{prop}
 
 \begin{proof} Suppose $H(S)\subset H^{(r)}$ and $H(S')\subset H^{(r')}$.
 Then, similarly to
  the proof of Proposition \ref{prop:inclusion-supp-H},
 we have
 $$\Supp_{A^{(r+r')}}\bigl(H(S)*H(S')\bigr)\,\,\subset\,\,
 \Supp_{A^{(r+r')}}\bigl(H(S)\otimes H(S')\bigr) \,\,=\,\,p(S\times S'),
  $$
  as $p$ corresponds to the morphism $\Delta_{r,r'}:
  A^{(r+r')}\to A^{(r)}\otimes A^{(r')}$ which is an integral extension
  of rings.
\end{proof}

  

\vskip .2cm
 
Now, recall (see Section \ref{section:symmetric-powers})
that the category $\QCoh(\frak S)$ has a monoidal
structure $\*$.
The following is then straightforward.
 
\begin{prop} The Hall multiplication on $H$ localizes to a morphism
$\mu: \Hc\*\Hc\to\Hc$ which makes $\Hc$ into an associative algebra
in $(\QCoh(\frak S), \*)$. \qed
\end{prop}

\vskip3mm

\subsection{Rational completion of $H^{\otimes n}$.}
We call the {\em rational completion} of $H$ the vector space
$$H_{\rat} \,\,=\,\,\bigoplus_{S\in\pi_0(\frak S)} k(S)\otimes_{k[S]} H(S)
\,\,=\,\,\Gamma_\qcom(\frak S, \Hc_{\rat}).$$
Here $\Hc_\rat$ is defined as in Section \ref{sec:rational-sections}.
The space $H_\rat$ is an associative algebra, but it is not
a priori identified with any space of functions on ${\Bun(X)}$. We will prove later (see Corollary~\ref{C:torsionfree}) that
the canonical map $H \to H_{\rat}$ is an embedding.
More generally, for $n\geqslant 1$ we denote
$$\begin{gathered} H^{\otimes n}_\rat\,\,=\,\,\Gamma_\qcom(\frak S, (\Hc^{\*n})_\rat)
\,\,=\cr
=\,\,\bigoplus_{S_1, ..., S_n\in\pi_0(\frak S)} 
k(S_1\times \cdots \times S_n)
\otimes_{k[S_1\times\cdots\times S_n]} H(S_1)\otimes \cdots H(S_n).
\end{gathered} $$
The twisted multiplication in $H^{\otimes n}$, see 
\eqref{eq:twisted-product-Hcoh}, gives
$H^{\otimes n}_\rat$ a structure of an associative algebra.

Recall the completions $\widetilde{H}, H^{\widetilde\otimes n}$ of $H, H^{\otimes n}$ defined in Section~2.4. Fix connected 
components $S_1, \ldots, S_n$ of $\frak S$ of weights $r_i=wt(S_i)$ and set
$$\widetilde{H}_{S_1, \ldots, S_n}=\bigoplus_{d \in \ZZ} \widetilde{H}_{S_1, \ldots, S_n}^{(r,d)}$$
where $r=\sum r_i$ and
$$\widetilde{H}_{S_1, \ldots, S_n}^{(r,d)}=\prod_{d_1+ \ldots + d_n=d} H(S_1)^{(r_1,d_1)} \otimes \cdots \otimes H(S_n)^{(r_n,d_n)}.$$
We have $\widetilde{H}_{S_1, \ldots, S_n} \subset H^{\widetilde{\otimes} n}$. The algebra $k[S_1 \times \cdots \times S_n]$ acts on
$\widetilde{H}_{S_1, \ldots, S_n}$, preserving $H(S_1) \otimes \cdots \otimes H(S_n)$. We call an element $u \in \widetilde{H}_{S_1, \ldots, S_n}$
\textit{rational} if there exists $a \in k[S_1 \times \cdots \times S_n]$, $a \neq 0$ such that $au \in H(S_1) \otimes \cdots \otimes H(S_n)$.
We denote by $\widetilde{H}_{S_1, \ldots, S_n, \rat}$ the set of rational elements of $\widetilde{H}_{S_1, \ldots, S_n}$.
There is a natural $k[S_1 \times \cdots \times S_n]$-module map 
$$\widetilde{H}_{S_1, \ldots, S_n,\rat} \to (H(S_1) \otimes \cdots \otimes H(S_n))
\otimes_{k[S_1 \times \cdots \times S_n]} k(S_1 \times \cdots \times S_n).$$
We set 
$$H^{\widetilde{\otimes} n}_{\rat}=\bigoplus_{S_1, \ldots, S_n} \widetilde{H}_{S_1, \ldots, S_n,\rat}.$$
The sum ranges over all $n$-tuples of connected components of $\frak S$.

\vspace{.1in}

\begin{prop} The following hold

(a) $H^{\widetilde{\otimes} n}_{\rat}$ is a subalgebra of $H^{\widetilde{\otimes} n}$,

(b) the natural map $H^{\widetilde{\otimes} n}_{\rat} \to H^{\otimes n}_{\rat}$ is an algebra homomorphism.
 
\end{prop}
\begin{proof}To prove (a) it is enough to show that for any 
$(S_1, \ldots,S_n),$ $ (S'_1, \ldots, S'_n)$ we have
$$\mu(\widetilde{H}_{S_1, \ldots, S_n,\rat} \otimes \widetilde{H}_{S'_1, \ldots, S'_n,\rat}) \subset \widetilde{H}_{S''_1, \ldots, S''_n,\rat},$$
where if $S_i \in \frak S^n$ and $S'_i \in \frak S^m$ then 
$S''_i = p(S_i \times S'_i) \in \frak S^{n+m}$ 
(see Proposition \ref{prop:decompsupport}).
Let $u,u'$ be respective elements of $\widetilde{H}_{S_1, \ldots, S_n,\rat},\widetilde{H}_{S'_1, \ldots, S'_n,\rat}$. 
There exists $a \in k[S_1 \times \cdots \times S_n]$ 
and $a' \in k[S'_1\times \cdots \times S'_n]$ such that 
$$au \in H(S_1) \otimes \cdots \otimes H(S_n), \qquad a'u' \in H(S'_1) \otimes \cdots \otimes H(S'_n).$$
Since $k[S_1 \times \cdots \times S_n] \otimes k[S'_1 \times \cdots \times S'_n]$ is, via $\Delta$, an integral extension
of $k[S''_1 \times \cdots \times S''_n]$, there exists 
$a'' \in k[S''_1 \times \cdots \times S''_n]$ such that
$$\Delta(a'') (u \otimes u') \in \big(H(S_1) \otimes \cdots \otimes H(S_n)\big)
\otimes \big(H(S'_1) \otimes \cdots \otimes H(S'_n)\big).$$
 Since the multiplication map $\mu$ is a morphism of $A$-modules and since 
$H(S_i) * H(S'_i) \subset H(S''_i)$ for all $i$, we obtain 
$$a'' (u * u') \in H(S''_1) \otimes \cdots \otimes H(S''_n)$$
as wanted. Statement (b) is obvious.
\end{proof}

\vskip3mm

\subsection{The Hall algebra as  a braided commutative bialgebra.}
We now give a Hopf-algebraic interpretation of the classical
package of theorems about principal series intertwiners  and constant
terms of Eisenstein series \cite{moeglin-waldspurger-book}.

\begin{thm}\label{thm:hall-as-R-matrix}
 (a)
The comultiplication $\Delta: H\to H^{\widetilde\otimes 2}$ takes
values in $H^{\widetilde\otimes 2}_\rat$ and so gives rise to
a coassociative map
$\Delta: H\to H^{\otimes 2}_\rat$. This map makes the sheaf $\Hc$ into a
rational coalgebra in the monoidal category $(\QCoh(\frak S),\*)$.
Let $\mu: \Hc\*\Hc\to \Hc_\rat$ and 
$\Delta: \, \Hc\to \Hc^{\* 2}_\rat$ denote the multiplication and the rational comultiplication in $\Hc$.

(b) There is an operator $M: H^{\otimes 2}_{rat}\to H^{\otimes 2}_\rat$
with the following properties

(b1) $M$ commutes with the  $A$-action defined via $\Delta: A\to A\otimes A$.
In particular, it gives rise to a rational morphism
$M: \Hc^{\*2}_{rat}\to \Hc^{\*2}_\rat$. 

(b2) $M$ is involutive and satisfies the Yang-Baxter equation
$$
\begin{gathered} 
M\circ M \,=\,\Id: \,\,\Hc^{\*2}_\rat\lra \Hc^{\* 2}_\rat\cr
M_{12}\circ M_{23}\circ M_{12}\, =\,M_{23}\circ M_{12}\circ M_{23}:\,\,
\Hc^{\* 3}\lra \Hc^{\* 3}_\rat.
\end{gathered}
$$

(b3) $\mu$ is $M$-commutative, i.e., $\mu\circ M$ maps into $\Hc$ and is equal to $\mu$.

(b4) $\Delta$ is $M$-cocommutative, i.e., 
$\Delta=M\circ\Delta$.

(b5) $\mu$ and $\Delta$ are $M$-compatible, i.e., $\Delta$
is a homomorphism of algebras, if the multiplication in $\Hc^{\* 2}_\rat$ is given by
$$(\Hc \* \Hc)_\rat \* (\Hc\ * \Hc)_{\rat}
\buildrel 1\otimes M\otimes 1\over\lra (\Hc \* \Hc \*\Hc\* \Hc)_\rat
\buildrel\mu\otimes\mu\over\lra (\Hc \* \Hc)_\rat.$$
\end{thm}

So the situation is quite similar to that of  rational  bialgebras,
see Proposition \ref{prop:M-commutativity-shuffle}, 
except at this stage the operator $M$ is defined on 
just one object $\Hc\*\Hc$ and
is not yet given as a part of a braiding on any ambient category.

The rest of this section is devoted to the proof  of Theorem \ref{thm:hall-as-R-matrix}. 
Note that it is enough to establish all the statements for the case $k=\CC$
which we assume.

We begin with statement (a). Because $H$ is generated by cuspidal elements, it is enough to show that
$\Delta(x) \in H^{\widetilde{\otimes} 2}_{\rat}$ for any 
$x=E_{f_1,d_1} * \cdots * E_{f_m,d_m}$, where $f_i$ is a cusp eigenform
on $\Bun_{r_i}(X)$, and $d_i \in \ZZ$ for $i=1, \ldots, m$. Define $E_\mathbf{f}(\mathbf{t})$ as in (\ref{eq:eisenstein-product}).
Set $r=\sum r_i$ and fix $r',r''$ such that $r'+r''=r$. The element $\Delta_{r',r''}(x)$ is the coefficient of $t_1^{d_1} \cdots t_m^{d_m}$
in the constant term of the Eisenstein series $E_{\mathbf{f}}(\mathbf{t})$ with respect to the standard parabolic subgroup
\begin{equation}\label{eq:parabolic-subgroup}
\quad P_{r',r''} =\begin{pmatrix} GL_{r'} &0\\ \Mat_{r'',r'}& GL_{r''}
\end{pmatrix}.
\end{equation}
Now, the constant term of any Eisenstein series 
with respect to any parabolic subgroup
is given by the classical  formula of Langlands 
(see, e.g., \cite{moeglin-waldspurger-book}, II.1.7), 
which for the case of $P_{r',r''}$
specializes to
\begin{equation}\label{eq:constant-term-eisenstein}
\begin{gathered}
\Delta_{r',r''}(E_{\mathbf{f}}(\mathbf{t})) \,\,=\,\,\sum_{m'+m''=m}
\sum_{\substack{(i_1, ..., i_{m'})\\ (j_1, ..., j_{m''})}}
\prod_{i_\a>j_\beta} q^{r_{i_\a}r_{j_b}(1-g_X)} 
{\LHom(f_{i_\a}, f_{j_\beta}, t_{i_\alpha}/t_{j_\beta})\over 
\LHom(f_{i_\a}, f_{j_\beta}, t_{i_\alpha}/qt_{j_\beta})}\times\cr
\times\bigl( E_{f_{i_1}}(t_{i_1})*\cdots 
*E_{f_{i_{m'}}}(t_{i_{m'}})\bigr) \otimes
\bigl( E_{f_{j_1}}(t_{j_1})* \cdots * E_{f_{j_{m''}}}(t_{j_{m''}})\bigr).
\end{gathered} \end{equation}
Here $(i_1, ..., i_{m'}, j_1, ..., j_{m''})$ run over all permutations of
$\{1,2,...,m\}$ such that $\sum i_\nu = r'$ and $\sum j_\nu = r''$.
The equality in \eqref{eq:constant-term-eisenstein} 
is understood as an equality of formal power series, when
we expand the rational functions on the right hand side
in the domain $|t_1|\gg ... \gg |t_m|$.  There exists a Laurent
polynomial $P(\mathbf{t}) \in \CC[t_1^{\pm 1}, \ldots, t_m^{\pm 1}]$ such that 
$\Delta_{r',r''}(P(\mathbf{t}) E_{\mathbf{f}}(\mathbf{t}))$
is a $\CC[t_{1}^{\pm 1}, \ldots, t_m^{\pm 1}]$-linear combination of the 
$$\bigl( E_{f_{i_1}}(t_{i_1})*\cdots *E_{f_{i_{m'}}}(t_{i_{m'}})\bigr) \otimes
\bigl( E_{f_{j_1}}(t_{j_1})* \cdots * E_{f_{j_{m''}}}(t_{j_{m''}})\bigr).$$
Recall that $E_{\mathbf{f}}(\mathbf{t})$ is an eigenfunction
of the Hecke algebra with character 
$$\chi_{E_{\bf f}({\bf t})} \,\,=\,\,(\chi_{f_1}\cdot 
t_1^{-\deg})\boxplus \cdots \boxplus
(\chi_{f_m}\cdot t_m^{-\deg})$$
(addition in the Witt scheme). We can find 
$a \in \CC[p(S_1 \times \cdots \times S_m)]$
such that $a E_{\mathbf{f}}(\mathbf{t}) =P(\mathbf{t}) E_{\mathbf{f}}(\mathbf{t})$, where $S_i \subset \Sigma^{(r_i)}$ is the connected component
of $f_i$. Considering the coefficient of $t_1^{d_1} \cdots t_m^{d_m}$ and using the fact that $\Delta$ is a morphism of $A$-modules,
we deduce that $a \Delta(x) \in H \otimes H$. Note that the action of $a$ on 
thesummand
$$\bigl( E_{f_{i_1}}(t_{i_1})*\cdots *E_{f_{i_{m'}}}(t_{i_{m'}})\bigr) \otimes
\bigl( E_{f_{j_1}}(t_{j_1})* \cdots * E_{f_{j_{m''}}}(t_{j_{m''}})\bigr)$$
factors through the map 
$$\CC[p(S_1 \times \cdots \times S_m)] \to
\CC[p(S_{i_1} \times \cdots \times S_{i_m'})] \otimes 
\CC[p(S_{j_1} \times \cdots \times S_{j_{m''}})]$$
induced by $\Delta : A \to A \otimes A$.
This implies that $\Delta(x) \in H^{\widetilde\otimes 2}_{\rat}$ as wanted, and proves (a).

\vspace{.1in}

The remainder of this section deals with (b). Before giving the definition of the operator $M$, we
recall a few notations. As before, $\Aen$ and $K$ denote the ring of adeles of and field of rational functions
on the curve $X$. 
Fix $r,s\geqslant 1$ and consider the subgroup
$$\Xi_{r,s} \,\,=\,\,\begin{pmatrix} GL_r(K)&0\\ \Mat_{s,r}(\Aen)& GL_s(K)
\end{pmatrix}\,\,\subset\,\, GL_{r+s}(\Aen).$$
Then the Iwasawa decomposition defines an identification
\begin{equation}\label{eq:double-xi-cosets}
\begin{gathered} \theta_{r,s}: \Xi_{r,s}\backslash GL_{r+s}(\Aen)/GL_{r+s}(\widehat\Oc\bigr)  
\lra {\Bun_r (X)}\times{\Bun_s(X)}\,\,\simeq\cr\simeq \,\,
\bigl(GL_r(K)\backslash GL_r(\Aen)/GL_r(\widehat \Oc)\bigr) \times
\bigl(GL_s(K)\backslash GL_s(\Aen)/GL_s(\widehat \Oc)\bigr)
\end{gathered}
\end{equation}
Using it, we define a linear isomorphism
$$\widehat\epsilon_{r,s}: (H^{\widehat \otimes 2})^{(r,s)} = 
\Fc\bigl({\Bun_r (X)}\times{\Bun_s(X)}\bigr) \lra
\Fc \bigl(\Xi_{r,s}\backslash GL_{r+s}(\Aen)/GL_{r+s}(\widehat\Oc\bigr))$$
by putting 
$$(\widehat\epsilon_{r,s} f)([g]) = \langle E'',E'\rangle f(E', E''),$$
where $g\in GL_{r+s}(\Aen)$, $[g]$ is its double coset in 
the source of 
\eqref{eq:double-xi-cosets}, $E'\in{\Bun_r(X)}$,
$E''\in{\Bun_s(X)}$ are such that $\theta_{r,s}([g]) = (E', E'')$ and $\langle E'',E'\rangle$ is the Euler form.
The restriction of $\widehat\epsilon_{r,s}$ to functions with finite support defines
an isomorphism
$$\epsilon_{r,s}: H^{(r)}\otimes H^{(s)} \lra 
\Fc_0 \bigl(\Xi_{r,s}\backslash GL_{r+s}(\Aen)/GL_{r+s}(\widehat\Oc\bigr)).$$
Recall further (\cite{moeglin-waldspurger-book}, II.1.10), the 
(pseudo-)Eisenstein series map
$$\begin{gathered}\Eis_{r,s}: \,\Fc_0 \bigl(\Xi_{r,s}\backslash GL_{r+s}(\Aen)\bigr)
\lra \Fc_0 \bigl(GL_{r+s}(K)\backslash GL_{r+s}(\Aen)\bigr), \\
\bigl(\Eis_{r,s}(f)\bigr) (g) \,\,=\,\,\sum_{\gamma\in P_{r,s}(K)\backslash GL_{r+s}(K)}
f(\gamma g), \quad \operatorname{where} \quad P_{r,s} =\begin{pmatrix} GL_r &0\\ \Mat_{s,r}& GL_s
\end{pmatrix}.
\end{gathered}
$$
 It is equivariant with respect to the right action of $GL_{r+s}(\Aen)$
and so induces a map between the spaces of invariants under $GL_{r+s}(\widehat \Oc)$.
By comparing the definitions, we see at once that~:
\begin{prop} The composition
$$\begin{gathered} H^{(r)}\otimes H^{(s)}\buildrel \epsilon_{r, s}\over\lra
  \Fc_0\bigl(\Xi_{r,s}\backslash GL_{r+s}(\Aen)\bigr)^{\GL_{r+s}(\widehat\Oc)}
\buildrel \Eis_{r,s}\over\lra \cr \buildrel \Eis_{r,s}\over\lra
\Fc_0 \bigl(GL_{r+s}(K)\backslash GL_{r+s}(\Aen)\bigr)
^{GL_{r+s}(\widehat\Oc)} \,=\, H^{(r+s)}
\end{gathered}
$$
is equal to the Hall multiplication $*$. \qed
\end{prop}

Our operator $M$ is essentially given by the classical {\em principal series intertwiner} 
(\cite{moeglin-waldspurger-book}, II.1.6). Recall that the latter is the $GL_{r+s}(\Aen)$-
equivariant operator
\begin{equation}\label{eq:M-on-representation}
M_{r,s}: \, \Fc_0\bigl(\Xi_{r,s}\backslash GL_{r+s}(\Aen)\bigr)
\lra \Fc\bigl(\Xi_{s,r}\backslash GL_{r+s}(\Aen)\bigr),
\end{equation}
  defined by
\begin{equation}\label{eq:M-operator}
(M_{r,s}f)(g) =\,\,\int_{Z\in\Mat_{r,s}(\Aen)} f \left( \begin{pmatrix}
Z&{\bf 1}_r\\{\bf 1}_s&0\end{pmatrix}\cdot g\right) dZ,
\end{equation}
where $dZ= \prod_{i,j} dz_{ij}$ and $\int_{\Aen/K} dz_{ij}=1$. 
Compare also with the discussion in \cite{joyal-street}
for the case of a finite field.
As above, $M_{r,s}$ induces an operator on $GL_{r+s}(\widehat\Oc)$-equivariant
vectors
\begin{equation}\label{eq:M-on-invariants}
\begin{gathered} M_{r,s}\,: H^{(r)}\otimes H^{(s)} \lra \Fc\left(\Xi_{s,r}\backslash GL_{r+s}(\Aen)/
GL_{r+s}(\widehat\Oc)\right) \,= \\
= \Fc(\Bun_s(X)\times\Bun_r(X)) \,\, =\,\,  H^{(s)}\widehat\otimes H^{(r)}.\end{gathered}
\end{equation}

\begin{prop}\label{prop:MrsA}
 (a) The operator $M_{r,s}$ commutes with the  $A$-action defined via $\Delta: A\to A\otimes A$.

(b) The operator $M_{r,s}$
 takes values in $H^{\widetilde{\otimes} 2}_{\rat}$.
 
\end{prop}

\begin{proof} To prove (a), recall that $A^{(r+s)}$ is the Hecke algebra of 
  the group
 $GL_{r+s}(\Aen)$ by the subgroup $GL_{r+s}(\widehat\Oc)$. The action
   of $A^{(r+s)}$ on $H^{(r)}\widehat\otimes H^{(s)}$
 defined via $\Delta: A^{(r+s)}\to A^{(r)}\otimes A^{(s)}$ coincides 
 with the  standard Hecke algebra action   on the $GL_{r+s}(\widehat\Oc)$-invariant
 subspace in $\Fc(\Xi_{r,s}\backslash GL_{r+s}(\Aen))$.
 Since the operator $M_{r,s}$ in \eqref{eq:M-on-representation}
 is a morphism of $GL_{r+s}(\Aen)$-modules, its restriction
 to the invariant subspace commutes with the Hecke algebra action.

\vspace{.1in}

We now turn to (b). Let us first prove that $M_{r,s}$ takes values in $H^{(s)} \widetilde{\otimes} H^{(r)}$.
For $E\in \Bunn_r(X)$ and $F\in\Bunn_s(X)$ we  write
$$M_{r,s}(1_E\otimes 1_F)\,\,=\,\,\sum_{\substack{F'\in\Bunn_s(X)\\ E'\in\Bunn_r(X)}}
c_{EF}^{F'E'} \,(1_{E'}\otimes 1_{F'}).$$
It is enough to show that $c_{EF}^{F',E'}=0$ unless
\begin{equation}\label{eq:M-laurent-type}
  \deg(F')+\deg(E')=\deg(E)+\deg(F)
 \end{equation}
  Indeed, let $E,F$ correspond to
 $a\in GL_r(\Aen)$ and $b\in\GL_s(\Aen)$ and  let $f$ in \eqref{eq:M-operator}
 correspond to $1_E\otimes 1_F$, i.e., take $f$ to be the characteristic
 function of the double coset of $g_0=\begin{pmatrix} a&0\\0&b\end{pmatrix}\in GL_{r+s}(\Aen)$
 modulo $\Xi_{r,s}$ and $GL_{r+s}(\widehat \Oc)$. 
 For any element $h$ of this double coset we have $\ord\det(h) = \ord\det(a)=\ord\det(b)$. 
 Therefore for any $g$ in \eqref{eq:M-operator} such that 
 $f \left( \begin{pmatrix}
Z&{\bf 1}_r\\{\bf 1}_s&0\end{pmatrix}\cdot g\right)\neq 0$ we have
that $\ord\det(g)=\ord\det(a)+\ord\det(b)$. This proves
 \eqref{eq:M-laurent-type}.

\vskip .2cm

 Let us now recall the standard formulation of the rationality properties  
  satisfied by the $M_{r,s}$-operators. 
Let $f_i\in H^{(r_i)}_\cusp$, $i=1, ..., m$ and $f'_j\in H^{(s_j)}_\cusp$, $j=1,...,n$,
be two sequences of cusp eigenforms, with characters $\chi_i$ and $\chi'_j$ respectively, 
and let $r=\sum r_i$ and $s=\sum s_j$.   Define $E_{\bf f}({\bf t})$, $E_{{\bf f}'}({\bf t}')$  
  as in \eqref{eq:eisenstein-product}. Their values at any bundle
  on $X$ are series in ${\bf t}=(t_1, ..., t_m)$,
  ${\bf t}'=(t'_1, ..., t'_n)$, converging in
  \begin{equation}\label{eq:domain-convergence}
  |t_1|\gg \cdots\gg |t_m|, \quad |t'_1|\gg\cdots \gg |t'_n|
  \end{equation}
  to rational functions. 
  
  \begin{lemma}\label{prop:M-on-eisenstein} 
  For ${\bf t}, {\bf t}'$ in \eqref{eq:domain-convergence} the integral 
  \eqref{eq:M-operator} defining $M_{r,s}(E_{\bf f}({\bf t})\otimes E_{{\bf f}'}({\bf t}'))$
  converges to
  \begin{equation}\label{eq:M-on-eisenstein}
  M_{r,s}(E_{\bf f}({\bf t})\otimes E_{{\bf f}'}({\bf t}') )\,\,=\,\, E_{{\bf f}'}({\bf t}') \otimes E_{\bf f}({\bf t})\cdot
  \prod_{i=1}^m\prod_{j=1}^n
  q^{(1-g_X)r_is_j}{\LHom (\chi'_j, \chi_i; t'_j/t_i)\over \LHom(\chi'_j,\chi_i, t'_j/qt_i)
  }.
  \end{equation}
  \end{lemma} 
  
  \begin{proof} This is standard, we just indicate the main steps. 
  We write ${\bf r}=(r_1, ..., r_m)$, ${\bf s}=(s_1, ..., s_n)$ and
  consider the standard "block-lower triangular"
  parabolic subgroup $P_{{\bf r}, {\bf s}}\subset GL_{r+s}$ corresponding to
  the ordered partition of $r+s$ into $r_1, ..., r_m, s_1, ..., s_n$. Let $U_{{\bf r}, {\bf s}}$
  be the unipotent radical of $P_{{\bf r}, {\bf s}}$ and $L_{{\bf r}, {\bf s}}$ the
  Levi subgroup. We write $\Xi_{{\bf r}, {\bf s}}= L_{{\bf r}, {\bf s}}(K)U_{{\bf r}, {\bf s}}(\Aen)$. 
  We have the similar subgroups $P_{{\bf s}, {\bf r}}$ etc. defined by the ordered partition of
  $r+s$ into $s_1, ..., s_n, r_1, ..., r_n$.  Let $w_{{\bf r}, {\bf s}}$ be the minimal length element
  of  $S_{r+s}$, the Weyl group of $GL_{r+s}$,  transforming $P_{{\bf r}, {\bf s}}$ into $P_{{\bf s}, {\bf r}}$. 
    
  Notice that  $E_{\bf f}({\bf t})\otimes E_{{\bf f}'}({\bf t}')$
  is obtained by Eisenstein series summation from the element
  $$E_{{\bf f}, {\bf f"}}({\bf t}, {\bf t}') \,\,=\,\,\bigotimes_{i=1}^m E_{f_i}(t_i) \otimes\bigotimes_{j=1}^n E_{f'_j}(t_j)$$
  which can be regarded as a function on $\Xi_{{\bf r}, {\bf s}}\backslash GL_{r+s}(\Aen)/GL_{r+s}(\widehat \Oc)$. 
  We have the principal series intertwiner for $GL_{r+s}(\Aen)$ corresponding to $w_{{\bf r}, {\bf s}}$ which gives
  a map
  $$M_{w_{{\bf r}, {\bf s}}}: \Fc(\Xi_{{\bf r}, {\bf s}})\backslash GL_{r+s}(\Aen) \lra 
  \Fc(\Xi_{{\bf s}, {\bf r}})\backslash GL_{r+s}(\Aen).
  $$ 
  It is enough to prove that $M_{w_{{\bf r}, {\bf s}}}(E_{{\bf f}, {\bf f"}}({\bf t}, {\bf t}'))$
  converges in the domain \eqref {eq:domain-convergence} to 
  $E_{{{\bf f}'}, {\bf f}}({\bf t}', {\bf t})$ times the product of rational functions 
  in \eqref{eq:M-on-eisenstein}. 
  This is obtained by splitting $M_{w_{{\bf r}, {\bf s}}}$ first, into 
  the tensor
  product over $x\in X$,  of local intertwiners $M_{w_{{\bf r}, {\bf s}},x}$ for unramified principal series representations of
  $GL_{r+s}(\widehat K_x)$. The claim is that
    each local intertwiner contributes 
  the product over $i,j$ of the ratios of the   Euler factors at $x$
  for the   LHom functions in \eqref{eq:M-on-eisenstein}, times
  a constant coming from comparing the    Haar measure $da$ on $\Aen$ normalized by
  $\int_{\Aen/K}da=1$ and the product of the Haar measures $da_x$ on $\widehat K_x$,
  normalized by $\int_{\widehat\Oc_x} da_x=1$. This constant accounts for the
  power of $q$ in \eqref{eq:M-on-eisenstein}. Now, to prove the claim about each local intertwiner, 
  one represents it as the
   composition of $rs$ intertwiners
  corresponding to simple reflections in $S_{r+s}$. For each simple reflection the calculation
  becomes elementary, 
  reducing to the case of $GL_2(\widehat K_X)$,  and   yields a factor involving one eigenvalue
  of some $\chi_i$ and one eigenvalue of some $\chi'_j$; the product of these factors
  over all the eigenvalues
  gives the ratio of the two Euler factors as in \eqref{eq:M-on-eisenstein}. 
   \end{proof}
   
   \vskip .2cm

By Lemma~\ref{prop:M-on-eisenstein}, $M_{r,s}$ maps $H(S) \otimes H(S')$ to $H(S') \widetilde{\otimes} H(S)$ for any components
$S,S'$ of $\frak S$. Because the functions $\LHom$ are rational, there are Laurent polynomials $P(\mathbf{t},\mathbf{t'}), Q(\mathbf{t},\mathbf{t}')$ 
such that
 $$Q({\bf t}, {\bf t'})\cdot M_{r,s}(E_{\bf f}({\bf t})\otimes E_{{\bf f'}}({\bf t}'))\,\,=\,\, P({\bf t}, {\bf t'})\cdot
  (E_{{\bf f}'}({\bf t}')\otimes E_{{\bf f}}({\bf t}')).$$
Arguing as in the proof of Theorem~\ref{thm:hall-as-R-matrix} (a) above, we see that there exists elements $a \in \CC[p(S_1 \times \cdots \times S_m)]$ and
$b \in \CC[p(S'_1 \times \cdots \times S'_n)]$ such that $$(a \otimes b) (E_{\mathbf{f}}(\mathbf{t}) \otimes E_{\mathbf{f}'}(\mathbf{t}')) = 
Q(\mathbf{t},\mathbf{t}') E_{\mathbf{f}}(\mathbf{t}) \otimes E_{\mathbf{f}'}(\mathbf{t}').$$
Since $\Delta: A \to A \otimes A$ is an integral extension we can find $c \in A^{(r+s)}$ such that
$c(E_{\mathbf{f}}(\mathbf{t}) \otimes E_{\mathbf{f}'}(\mathbf{t}')) = 
Q'(\mathbf{t},\mathbf{t}') E_{\mathbf{f}}(\mathbf{t}) \otimes E_{\mathbf{f}'}(\mathbf{t}')$
for some Laurent polynomial $Q'(\mathbf{t}, \mathbf{t}')$ divisible by $Q(\mathbf{t}, \mathbf{t}')$. Because $M_{r,s}$ commutes with
the action of $A^{(r+s)}$, it follows that
$$c M_{r,s} (E_{\mathbf{f}}(\mathbf{t}) \otimes E_{\mathbf{f}'}(\mathbf{t}')) = 
P'(\mathbf{t},\mathbf{t}') E_{\mathbf{f}'}(\mathbf{t}') \otimes E_{\mathbf{f}}(\mathbf{t})$$
for some nonzero Laurent polynomial $P'(\mathbf{t}, \mathbf{t}')$. But then for any tuple $(d_1, \ldots, d_m, d'_1, \ldots d'_n)$, we have
$$c\, M_{r,s}((E_{f_1,d_1} * \cdots * E_{f_m,d_m}) \otimes (E_{f'_1,d'_1} * \cdots * E_{f'_n,d'_n})) \in H^{(s)} \otimes H^{(r)}$$
which implies that
$$ M_{r,s}((E_{f_1,d_1} * \cdots * E_{f_m,d_m}) \otimes (E_{f'_1,d'_1} * \cdots * E_{f'_n,d'_n})) \in H^{\widetilde{\otimes} 2}_{\rat}.$$
This finishes the proof of Proposition~\ref{prop:MrsA}.
\end{proof}

\vskip .2cm

We are now in position to define the operator $M : H^{\otimes 2} \to H^{\otimes 2}_{\rat}$. On each $H^{(r)} \otimes H^{(s)}$
it is given by the composition $H^{(r)} \otimes H^{(s)} \to H^{\widetilde{\otimes} 2}_{\rat} \to H^{\otimes 2}_{\rat}$.

Property (b1) is established in Proposition~\ref{prop:MrsA} (a).
The property $M\circ M=\Id$ from part (b2) 
follows from (\ref{eq:M-on-eisenstein}) and the functional equations satisfied by the LHom-functions, see
(\ref{E:functional-equation1}), (\ref{E:functional-equation2}). As for the Yang-Baxter equation,
it may be deduced directly from (\ref{eq:M-on-eisenstein}), or from the fact that the two sides, considered as operators
on each  $H^{(r)}\otimes H^{(s)}\otimes H^{(t)}$, are given by
integration over the same domain (appropriate Schubert cell
in the group $GL_{r+s+t}$, represented in two different ways 
as a product of smaller Schubert cells).  This proves (b2).

\vskip .2cm

Let us now prove (b3). As before,
it suffices to compare the values of $\mu$ and $\mu\circ M$ on
elements of $H\otimes H$ which are some Laurent coefficients
of $E_{\bf f}({\bf t})\otimes E_{{\bf f}'}({\bf t}')$. In this case
the statement expresses the functional equation for Eisenstein
series on the group $GL_{r+s}(\Aen)$ with respect to
the element $w_{{\bf  r}, {\bf s}}$ of $S_{r+s}$ described in the
proof of Lemma~\ref{prop:M-on-eisenstein}, namely
\begin{equation}\label{functional-equation-eisenstein}
E_{\bf f}({\bf t})* E_{{\bf f}'}({\bf t}') \,\,=\,\, E_{{\bf f}'}({\bf t}') * E_{\bf f}({\bf t})\cdot
  \prod_{i=1}^m\prod_{j=1}^n
  q^{(1-g_X)r_is_j}{\LHom (\chi'_j, \chi_i; t'_j/t_i)\over \LHom(\chi'_j,\chi_i, t'_j/qt_i)
  }.
\end{equation}

In the same vein, (b4) follows from the functional equation
of the LHom-functions and the formula \eqref{eq:constant-term-eisenstein}. 

Finally, (b5) is another consequence of the formula for
the constant term of an Eisenstein series. Indeed, we need to prove that
for  any $r_1+r_2=r$, any $u\in H^{(r_1)}$, $v\in H^{(r_2)}$,  and any
decomposition $r=r'+r''$ we have
$$\Delta_{r',r''}(u*v)\,\,=\,\,\sum \Delta_{r_{11}, r_{12}}(u) \, \tilde *\, \Delta_{r_{21}, r_{22}}(v),
$$
where $\tilde *$ is the $M$-twisted multiplication on $(H\otimes H)_\rat$ and the sum is
over $r_{ij}\in \ZZ_+$ such that
$$\begin{gathered}
r_{11}+r_{12}=r_1, \quad r_{21}+r_{22}=r_2,\\
r_{11}+r_{21}=r, \quad r_{12}+r_{22}=r'.
\end{gathered}
$$
As before, it is enough to assume that $u$ (resp. $v$) is the product of an initial
(resp. final) segment in some product $E_{f_1, d_1}* \cdots E_{f_m, d_m}$, and in this
case the statement follows from \eqref{eq:M-on-eisenstein} and \eqref{eq:constant-term-eisenstein}.
This finishes the proof of Theorem \ref{thm:hall-as-R-matrix}.

\vspace{.1in}

Let us sketch, for the reader's convenience, an alternative proof of Theorem~\ref{thm:hall-as-R-matrix} which does not
directly make use of the machinery or results of \cite{moeglin-waldspurger-book}. Recall that $H_{\coh}$ is a topological
bialgebra with coproduct $\Delta_{\coh}: H_{\coh} \to H_{\coh}^{\widetilde{\otimes} 2}$. If $f$ is a cuspidal eigenform then 
$$\Delta_{\coh}(E_f(t))=1 \otimes E_f(t) +E_f(t) \otimes  \Psi_f(t) $$
where
\begin{equation}\label{eq:psift}
\Psi_f(t)=\sum_{\mathcal{T}}=\sum_{\mathcal{T}} t^{deg \mathcal{T}}\overline{\chi}_f(\mathcal{T}) |\text{Aut}\;\mathcal{T}| 1_{\mathcal{T}}
\end{equation}
where the sum ranges over all torsion sheaves $\mathcal{T}$ on $X$ (see \cite{K}, Section (3.2)). Observe that
$\Psi_f(t)=\Psi_{\overline{\chi_f}}(t)$ where $\chi_f$ is the caharacter of $A$ associated to $f$, and where $\Psi_{\rho}(t)$ is defined as in Section 2.

Langlands' formula
\eqref{eq:constant-term-eisenstein} for the constant term of Eisenstein series may now be deduced from the following Lemma using
Proposition~\ref{prop:properties-heckeop} and standard properties of bialgebras.

\vspace{.1in}

\begin{lemma}\label{lem:psif-Eis} Let $f, g$ be cuspidal eigenforms. Then
$$E_g(t) * \Psi_f(t') = \frac{\LHom(f,g,t'/t)}{\LHom(f,g, t'/qt)} \Psi_f(t') * E_g(t).$$
\end{lemma}
\begin{proof} This follows from \eqref{eq:psift} and the definition of the $\LHom$ function, see \cite{K}, Section~4.1. Note that
the formulas here differ slightly from those of \cite{K} since we do not twist the coproduct as in \textit{loc. cit.}
\end{proof}

Next, we define the principal intertwiner map as follows~: 
\begin{equation}
 \begin{split}
  M_{r,s}~: H^{(r)} \otimes H^{(s)} &\to H^{(s)} \widetilde{\otimes} H^{(r)}\\
u \otimes v &\mapsto p_2(\Delta_{\coh, 0,r}(u) * \Delta_{\coh,s,0}(v))
 \end{split}
\end{equation}
where $p_2: H^{\widetilde{\otimes} 2}_{\coh} \to H^{\widetilde{\otimes} 2}$ is the projection map, and where $*$ is the twisted
multiplication on $H^{\widetilde{\otimes} 2}_{\coh}$ (see Section~2.4). Formula (\ref{eq:M-on-eisenstein}) is also a formal consequence
of Lemma~\ref{lem:psif-Eis} and standard properties of bialgebras. The remaining statements of Theorem~\ref{thm:hall-as-R-matrix} are deduced,
as above, from \eqref{eq:constant-term-eisenstein} and \eqref{eq:M-on-eisenstein}.

\vskip3mm

\subsection{The Hermitian scalar product and proof of Theorem \ref{thm:main}.}
Consider the grading of $H$ by ``degree of non-cuspidality"
\begin{equation}
\begin{gathered}
H\,=\bigoplus_{n\geqslant 0} H_n, \quad H_n\,=\,\Im\bigl\{\mu^{(n-1)}: 
(H_\cusp)^{\otimes n} \to H\bigr\},\\
H_n=\bigoplus_{S\in \pi_0(\frak S^n)} H(S).
\end{gathered}
\end{equation}
Here $H(S)$ is as in \eqref{eq:component-decomposition-H}. 
In particular, $H_0=k$ and $H_1=H_\cusp$.  For each $n\geqslant 0$
the tensor product $H^{\otimes n}$ has then a natural $\ZZ_+^n$-grading.
 Note that the action of the Hecke
algebra $A$ on $H^{\otimes n}$ preserves this $\ZZ_+^n$-grading and  therefore induces
a $\ZZ_+^n$-grading on $H^{\otimes n}_\rat$. 

\begin{lemma} For any $n\geqslant 0$, consider the $\ZZ_+$-grading 
on $H^{\otimes n}$
by total degree. Then 
the iterated comultiplication
$\Delta^{(n-1)}: H\to H^{\otimes n}_\rat$ preserves the grading.  
\end{lemma}

\begin{proof} It is enough to consider the case $n=2$. Further, it is enough
to assume that $k$ is algebraically closed, so $H$ is generated by the elements
$E_{f,d}\in H_1$ as in \eqref{eq:elements-E-f-d}. So it is enough to prove
that $\Delta(E_{f_1, d_1}*\cdots E_{f_m, d_m})$ lies in the component
of $H^{\otimes 2}_\rat$ of total degree $m$. This follows from
\eqref{eq:constant-term-eisenstein}.
\end{proof}

Let 
$$\omega_n = p_{1,\cdots 1}\circ\Delta^{n-1}: \, H_n \lra (H_1)^{\otimes n}_\rat$$
be the composition of $\Delta_{n-1}$ and the projection on the component
of multidegree $(1,...,1)$. Let
$$\omega \,=\,\bigoplus_{n\geqslant 0}\, \omega_n: \,\,\, H\lra 
\  \bigoplus_n (H_1)^{\otimes n}_\rat,
$$
be the direct sum of the $\omega_n$. 
 Theorem \ref{thm:main} follows from the next fact.

\begin{prop} (a) There is a unique rational bihomomorphism
$\tilde c:\frak S\times\frak S\to\Bbb A^1_k$ whose restriction to $\Sigma\times\Sigma$ is equal
 to the function $c\in k(\Sigma\times\Sigma)^\times$
given by \eqref{eq:rankin-selberg-c}.

(b) The map $\omega$ takes the Hall multiplication in $H$ into the
shuffle multiplication corresponding to $\tilde c$.

(c) $\omega$ is a monomorphism of vector spaces. 
\end{prop}

\begin{proof}  Part (a) is obvious. Indeed, start with $c:\Sigma\times\Sigma\to\Bbb A^1_k$
and extend it to $c:\Sym(\Sigma)\times\Sym(\Sigma)\to\Bbb A^1_k$ by bi-homomorphicity.
Then, let $\tilde c$ be the pushforward of $c$. It is still a rational bihomomorphism.

Part (b) is equivalent to the formula for the constant
term of Eisenstein series with respect to general parabolic
subgroups in $GL_r$. It can be deduced 
 from Theorem \ref{thm:hall-as-R-matrix} (b5). Indeed, 
we assume that $k$ is algebraically closed and look at generators
$E_{f_i, d_i}$ for 
$f_i\in H^{(r_i)}_\cusp$, $i=1,...,m $.  Then the $M$-compatibility
of the product and coproduct in $H$ together with 
\eqref{eq:constant-term-eisenstein} implies the  constant term formula 
in the form
\begin{equation}
\begin{gathered}
\Delta^{(n-1)}\bigl(E_{f_1}(t_1)*\cdots * E_{f_n}(t_n)\bigr) \,\,=\cr
=\sum_{\sigma\in S_n}\biggl[ \prod_{\substack{i<j \\ 
\sigma(i)^{-1} > \sigma^{-1}(j)}}
q^{r_i r_j (1-g_X)} 
{\LHom(f_{j}, f_{i}, t_{j}/t_{i})\over 
\LHom(f_{j}, f_{i}, t_{j}/qt_{i})}\biggr]
\bigl(E_{f_{\sigma(1)}}(t_{\sigma(1)})\otimes \cdots\\
\cdots\otimes E_{f_{\sigma(n)}}(t_{\sigma(n)})\bigr),
\end{gathered}
\end{equation}
which is precisely the formula for the shuffle product on the generators.
Since both the Hall product and the shuffle product are associative,
claim (b) follows.

To see part (c), it is enough to assume $k=\CC$, which we will. Let
$F=\sum f_{i_1}*\cdots * f_{i_n}$ be a nonzero element of $H_n$, 
so $f_{i_\nu}\in H_1=H_\cusp$.
To show that
$\Delta^{(n-1)}(F)$ is nonzero, it is enough to consider it as an element of
$H^{\widetilde\otimes n}\subset\Fc(\Bun(X)^n)$ and to find 
$\phi\in \Fc_0(\Bun(X)^n) =
H^{\otimes n}$ such that $(\Delta^{(n-1)}(F),\phi)_{\Herm} \neq 0$
(orbifold Hermitian product \eqref {hermitianproduct} of functions on 
the orbifold $\Bun(X)^n$).
Let us take $\phi = \sum f_{i_1}\otimes \cdots \otimes f_{i_n}$. Then
the adjointness of the product and coproduct gives
$$(\Delta^{(n-1)}(F), \phi)_{\Herm} \,\,=\,\,(F,F)_{\Herm} \,\, > \,\,0$$
as the Hermitian product \eqref {hermitianproduct} is positive definite. 
\end{proof} 

\vspace{.2cm}
Note that the shuffle algebra $Sh(\Sigma,c)$ is the same as the 
generalized shuffle algebra $Sh(\frak S,\tilde c)$.
\vspace{.2cm}

\begin{cor}\label{C:torsionfree}
For any component $S$ of $\frak S$, the $k[S]$-module $H(S)$ is torsion free.
\end{cor}

\begin{proof} Indeed, if $S=(S_1, \ldots, S_n) \in \pi_0(\frak S^n)$ 
then $\omega_n$ restricts to an embedding
\begin{equation}\label{E:torsionfree}
\omega_n~: H(S) \to \bigoplus_{\sigma} \big( H(S_{\sigma(1)}) \otimes \cdots \otimes H(S_{\sigma(n)})\big)_{\rat}
\end{equation}
where the sum ranges over all permutations $\sigma$ of $\{1, \ldots, n\}$. By proposition~\ref{prop:hall-mult-hecke-comult},
$\omega_n$ is a morphism of $A$-modules, where $A$ acts on $(H_1)^{\otimes n}_{\rat}$ via $\Delta^{(n-1)}$. In particular,
the restriction of $\omega_n$ to $H(S)$ is a morphism of $k[S]$-modules, where $k[S]$ acts on the right hand side of (\ref{E:torsionfree})
via the embedding 
$$k[S] \to \bigoplus_{\sigma} k[S_{\sigma(1)}] \otimes \cdots 
\otimes k[S_{\sigma(n)}]$$ which is the composition
of $\Delta^{(n-1)}$ and the projections $A^{(r_i)} \to k[S_i]$.  
It remains to observe that this action of $k[S]$ on the right hand side of
(\ref{E:torsionfree}) is torsion free by definition of $(H_1^{\otimes n})_{\rat}$.
\end{proof}

\vfill\eject

\vskip 1cm

Authors' addresses:

\begin{itemize}
\item[] M.K.: Department of Mathematics, Yale University, 10 Hillhouse Avenue, New Haven CT 06520 USA, 
email: {\tt mikhail.kapranov@yale.edu}

\item[] O. S.: D\'epartement de Math\'ematiques,
B\^atiment 425, 
Facult\'e des Sciences d'Orsay, 
Universit\'e Paris-Sud 11
F-91405 Orsay Cedex, France, email: {\tt  olivier.schiffmann@gmail.com}

\item[] E. V.:  Institut de Math\'ematiques de Jussieu, UMR 7586,  Universit\'e Paris-7 Denis Diderot, 
UFR de Math\'ematiques, Case 7012, 75205 Paris Cedex 13, France, email:
{\tt vasserot@math.jussieu.fr}

\end{itemize}

\end{document}